\pdfoutput=1
\def\arXiv{1} 
\documentclass[11pt]{article} 
\newcommand{\notarxiv}[1]{foo}
\newcommand{\arxiv}[1]{ba}
\ifdefined \arXiv
\renewcommand{\arxiv}[1]{#1}%
\renewcommand{\notarxiv}[1]{\ignorespaces}%
\else%
\renewcommand{\arxiv}[1]{\ignorespaces}%
\renewcommand{\notarxiv}[1]{#1}%
\fi%
\usepackage{microtype}
\usepackage{graphicx}
\usepackage{subfigure}
\usepackage{booktabs} %

\usepackage{hyperref}

\notarxiv{\usepackage[accepted]{icml2022}}

\usepackage{amsmath}
\usepackage{amssymb}
\usepackage{mathtools}
\notarxiv{\usepackage{amsthm}}

\usepackage[ruled,linesnumbered,algo2e,noend]{algorithm2e}
\usepackage[capitalize,noabbrev]{cleveref}

\notarxiv{
\theoremstyle{plain}
\newtheorem{theorem}{Theorem}[section]
\newtheorem{proposition}[theorem]{Proposition}
\newtheorem{lemma}[theorem]{Lemma}
\newtheorem{corollary}[theorem]{Corollary}
\theoremstyle{definition}
\newtheorem{definition}[theorem]{Definition}

\theoremstyle{remark}
\newtheorem{remark}[theorem]{Remark}}

\usepackage[textsize=tiny]{todonotes}

\usepackage{thm-restate}
\usepackage{xparse}
\usepackage{enumerate}
\usepackage{amssymb}
\usepackage{amsbsy}
\usepackage{amsthm}
\usepackage{amsfonts}
\usepackage{latexsym}
\usepackage{graphicx}
\usepackage{mathtools}
\usepackage{xcolor}
\usepackage{enumitem}
\usepackage{thmtools}
\usepackage{graphicx}
\usepackage{color}
\usepackage{bbm}
\usepackage{xifthen}
\usepackage{xspace}
\usepackage{mathtools}
\usepackage{titlesec}
\usepackage{setspace}
\usepackage{etoolbox}
\usepackage{xfrac}
\usepackage{esvect}
\usepackage{nicefrac}
\usepackage{cases}
\usepackage{empheq}
\usepackage{booktabs}
\usepackage{multirow}
\usepackage{caption}
\usepackage{bbm}
\usepackage{array}

 \notarxiv{
\theoremstyle{plain}
\undef\corollary
\undef\definition
\undef\lemma
\undef\proposition	
\declaretheorem[name=Lemma,sibling=theorem]{lemma}
\declaretheorem[name=Proposition,sibling=theorem]{proposition}
\declaretheorem[name=Corollary,sibling=theorem]{corollary}
\declaretheorem[name=Definition,sibling=theorem]{definition}
}

\arxiv{
	\usepackage{amsthm}

	\usepackage[top=1in, right=1in, left=1in, bottom=1in]{geometry}
	\usepackage[numbers,square]{natbib}
	\usepackage{url}
	\hypersetup{
		colorlinks=true,
		linkcolor=blue!70!black,
		citecolor=blue!70!black,
		urlcolor=blue!70!black}
	
	\theoremstyle{plain}
	
	\newtheorem{theorem}{Theorem}
	\newtheorem{lemma}{Lemma}
	
	\newtheorem{proposition}{Proposition}
	\newtheorem{corollary}{Corollary}
	\newtheorem{definition}{Definition}
	
	\theoremstyle{definition}

	\newtheorem*{example*}{Example}
	\newenvironment{psketch}{\noindent{\emph{Proof sketch.}}}{\qed\bigskip}
}

\ifdefined \arXiv
\makeatletter
\long\def\@makecaption#1#2{
	\vskip 0.8ex
	\setbox\@tempboxa\hbox{\small {\bf #1:} #2}
	\parindent 1.5em  %
	\dimen0=\hsize
	\advance\dimen0 by -3em
	\ifdim \wd\@tempboxa >\dimen0
	\hbox to \hsize{
		\parindent 0em
		\hfil 
		\parbox{\dimen0}{\def\baselinestretch{0.96}\small
			{\bf #1.} #2
		} 
		\hfil}
	\else \hbox to \hsize{\hfil \box\@tempboxa \hfil}
	\fi
}
\makeatother
\fi

\DeclarePairedDelimiter{\brk}{[}{]}
\DeclarePairedDelimiter{\crl}{\{}{\}}
\DeclarePairedDelimiter{\prn}{(}{)}

\newcommand{\Par}[1]{\left(#1\right)}

\newcommand{\norm}[1]{\left\lVert#1\right\rVert}
\newcommand{\inprod}[2]{\left\langle#1, #2\right\rangle}

\newcommand{\inner}[2]{\left<#1,#2\right>}

\newcommand{\E}{\mathbb{E}} %

\renewcommand{\P}{\mathbb{P}} %
\newcommand{\geoRV}{\mathsf{Geom}}
\NewDocumentCommand\Ex{s O{} m }{%
	\mathbb{E}%
	\begingroup
	\IfBooleanTF{#1}
	{\ExInn*{#3}}
	{\ExInn[#2]{#3}}%
	\endgroup
}

\DeclarePairedDelimiterX\ExInn[1]{[}{]}{%
	\activatebar
	#1%
}

\RenewDocumentCommand\Pr{sO{}r()}{%
	\mathbb{P}%
	\begingroup
	\IfBooleanTF{#1}
	{\PrInn*{#3}}
	{\PrInn[#2]{#3}}%
	\endgroup
}

\DeclarePairedDelimiterX\PrInn[1](){%
	\activatebar
	#1%
}

\newcommand{\activatebar}{%
	\begingroup\lccode`~=`|
	\lowercase{\endgroup\def~}{\;\delimsize\vert\;}%
	\mathcode`|=\string"8000
}

\newcommand{\xset}{\mathcal{X}}
\newcommand{\yset}{\mathcal{Y}}

\newcommand{\R}{\mathbb{R}}

\newcommand{\bx}{{\bar{x}}}

\newcommand{\tx}{{\widetilde{x}}}
\newcommand{\xopt}{x^\star}
\newcommand{\yopt}{y^\star}

\newcommand{\xinit}{x_{\mathrm{init}}}
\newcommand{\yinit}{y_{\mathrm{init}}}
\newcommand{\xprev}{x_{\mathrm{prev}}}
\newcommand{\xfull}{x_{\mathrm{full}}}
\newcommand{\ybr}{y^{\mathsf{br}}}

\newcommand{\proj}{\mathsf{Proj}}
\DeclareMathOperator*{\argmin}{arg\,min}
\DeclareMathOperator*{\argmax}{arg\,max}
\DeclareMathOperator*{\minimize}{minimize}

\newcommand{\ouralg}{\textsc{RECAPP}\ignorespaces}
\newcommand{\ApproxProx}{\textsc{ApproxProx}}
\newcommand{\UnbiasedProx}{\textsc{UnbiasedProx}}
\newcommand{\MLMC}{\textsc{MLMC}}
\newcommand{\MirrorProx}{\textsc{MirrorProx}}

\newcommand{\SVRGone}{\textsc{One}\textsc{Epoch}\textsc{Svrg}}
\newcommand{\SVRG}{\textsc{Svrg}}

\newcommand{\WarmStart}{\textsc{WarmStart}}
\newcommand{\fun}{\Phi}

\newcommand{\Flam}{F^{\lambda}}

\newcommand{\Veu}{V^\mathrm{e}}
\newcommand{\eps}{\epsilon}
\newcommand{\grad}{\nabla}

\newcommand{\brorcl}{\mathcal{O}^\mathsf{br}}

\newcommand{\calN}{\mathcal{N}}

\newcommand{\gopt}{g^\star}

\newcommand{\Otil}[1]{\widetilde{O}\prn*{#1}}

\newcommand{\codeStyle}[1]{{\bfseries #1} }
\newcommand{\codeInput}{\codeStyle{Input:}}	
\newcommand{\codeOutput}{\codeStyle{Outut:}}	
\newcommand{\codeReturn}{\codeStyle{Return:}}	
\newcommand{\codeParameter}{\codeStyle{Parameter:}}

\SetCommentSty{mycommfont}
\SetKwInput{KwInput}{Input} 
\SetKwInput{KwParameter}{Parameters}                %
\SetKwInput{KwOutput}{Output}              %
\SetKwInput{KwReturn}{Return}              %
\SetKwComment{Comment}{$\triangleright$\ }{}
\SetCommentSty{mycommfont}

\newcommand{\defeq}{:=}

\notarxiv{
\icmltitlerunning{\ouralg: Crafting a More Efficient Catalyst for Convex Optimization}

\begin{document}

\twocolumn[
\icmltitle{\ouralg: Crafting a More Efficient Catalyst for Convex Optimization}

\begin{icmlauthorlist}
\icmlauthor{Yair Carmon}{ta}
\icmlauthor{Arun Jambulapati}{s}
\icmlauthor{Yujia Jin}{s}
\icmlauthor{Aaron Sidford}{s}
\end{icmlauthorlist}

\icmlaffiliation{s}{Stanford University}
\icmlaffiliation{ta}{Tel Aviv University}

\icmlcorrespondingauthor{Yujia Jin}{yujiajin@stanford.edu}

\icmlkeywords{Machine Learning, ICML}

\vskip 0.3in
]

\printAffiliationsAndNotice{} %
}

\arxiv{
	\title{\ouralg: Crafting a More Efficient Catalyst\\ for Convex Optimization}

	\author{Yair Carmon ~~~ Arun Jambulapati ~~~ Yujia Jin ~~~ Aaron Sidford\\
	\href{mailto:ycarmon@tauex.tau.ac.il}{\texttt{ycarmon@tauex.tau.ac.il}}, 
	\texttt{\{\href{mailto:jmblpati@stanford.edu}{jmblpati},%
		\href{mailto:yujiajin@stanford.edu}{yujiajin},%
		\href{mailto:sidford@stanford.edu}{sidford}\}@stanford.edu}}
	\date{}
	
	\begin{document}
	
	\maketitle
}

\begin{abstract}
The accelerated proximal point algorithm (APPA), also known as ``Catalyst'', is a well-established reduction from convex optimization to approximate proximal point computation (i.e., regularized minimization). This reduction is conceptually elegant and yields strong convergence rate guarantees. However, these rates feature an extraneous logarithmic term arising from the need to compute each proximal point to high accuracy. In this work, we propose a novel Relaxed Error Criterion for Accelerated Proximal Point ($\ouralg$) that eliminates the need for high accuracy subproblem solutions. We apply $\ouralg$ to two canonical problems: finite-sum and max-structured minimization. For finite-sum problems, we match the best known complexity, previously obtained by carefully-designed problem-specific algorithms. For minimizing $\max_y f(x,y)$ where $f$ is convex in $x$ and strongly-concave in $y$, we improve on the best known (Catalyst-based) bound by a logarithmic factor.

\end{abstract}

\section{Introduction}\label{sec:intro}

\newcommand{\prox}{\mathrm{prox}_{F,\lambda}}

A fundamental approach to optimization algorithm design is to break down the problem of minimizing $F:\xset \to \R$ into a sequence of easier optimization problems, whose solution converges to $\xopt\in\argmin_{x\in \xset} F(x)$. A canonical ``easier''  problem is proximal point computation, i.e., computing
\begin{equation*}
	\prox(x) = \argmin_{x'\in\xset} \crl*{ F(x') +  \frac{\lambda}{2} \norm{x'-x}^2}\,.
\end{equation*}
Here  $\lambda > 0$ is a regularization parameter which balances between how close computing $\prox$ is to minimizing $F$, and how easy it is to compute---as $\lambda$ decreases  $\prox(x)$ tends to the true minimizer but becomes harder to compute.

The classical proximal point method~\citep{rockafellar1976monotone,parikh2014proximal} simply iterates $x_{t+1} = \prox(x_t)$, and (for convex $F$) minimizes $F$ with rate $F(x_{T}) - F(\xopt) = O\prn[\big]{\frac{\lambda R^2}{T}}$ for $R \defeq \norm{x_0 - \xopt}$. This rate can be improved, at essentially no additional cost, by carefully combining proximal steps, i.e. $\prox(x_t)$, with gradient steps using $\grad F(\prox(x_t)) = \lambda( x_t -  \prox(x_t) )$~\cite{guler1992new}. This accelerated method converges with rate $O\prn[\big]{\frac{\lambda R^2}{T^2}}$, a quadratic improvement over proximal point.

To turn this conceptual acceleration scheme into a practical algorithm, one must prescribe the accuracy to which each proximal point needs to be computed. \citet{guler1992new,salzo2012inexact} provided such conditions, which were later refined in independently-proposed accelerated approximate proximal point algorithm (APPA)~\citep{frostig2015regularizing} and Catalyst~\citep{lin2015universal,lin2017catalyst}. Furthermore, APPA/Catalyst obtained global convergence guarantees for concrete problems by using linearly convergent algorithms to compute the proximal points to their required accuracy. The APPA/Catalyst framework has since been used to accelerate full-batch gradient descent, coordinate methods, finite-sum variance reduction ~\citep{frostig2015regularizing,lin2017catalyst}, eigenvalue problems~\citep{garber2016faster}, min-max problems~\citep{yang2020catalyst}, and more. 

However, the simplicity and generality of APPA/Catalyst seems to come at a  practical and theoretical cost: satisfying existing proximal point accuracy conditions requires solving subproblems to fairly high accuracy. In practice, this means expending computation in subproblem solutions that could otherwise be used for more outer iterations. In theory, APPA/Catalyst complexity bounds feature a logarithmic term that appears unnecessary. For example, in the finite sum problem of minimizing $F(x) = \frac{1}{n}\sum_{i=1}^n f_i(x)$ where each $f_i$ is convex and $L$-smooth, APPA/Catalyst combined with SVRG~\citep{johnson2013accelerating} and $\lambda = L/n$ finds an $\epsilon$-optimal point of $F$ in $O\prn*{ (n + \sqrt{nLR^2/\epsilon})\log \frac {LR^2}{\epsilon}}$ computations of $\grad f_i$, a nearly optimal rate~\citep{woodworth2016tight}. A line of work devoted to designing accelerated method tailored to finite sum problems~\citep{shalev2014accelerated,allen2016katyusha,lan2019unified} attains progressively better practical performance and theoretical guarantees, culminating in an $O\prn*{ n\log\log n + \sqrt{nLR^2/\epsilon}}$ complexity bound~\citep{song2020variance}.

Given the power of APPA/Catalyst, it is natural to ask whether the additional logarithmic complexity term is fundamentally tied to the black-box structure that makes it generally-applicable? Indeed,~\citet{lin2017catalyst} speculate that the logarithmic term ``may be the price to pay for a generic acceleration scheme.'' 

Our work proves otherwise by providing a new Relaxed Error Condition for Accelerated Proximal Point ($\ouralg$) which standard subproblem solvers can satisfy without incurring an extraneous logarithmic complexity term. For finite-sum problems, our approach combined with SVRG recovers the best existing complexity bound of $O\prn*{ n\log\log n + \sqrt{nLR^2/\epsilon}}$.\footnote{Our framework gives accelerated linear convergence for strongly-convex objectives via a standard restarting argument; see~\Cref{prop:catalyst-sc}.} Preliminary experiments on logistic regression problem indicate that our method is competitive with Catalyst-SVRG in practice.\footnote{Code available at: \href{https://github.com/yaircarmon/recapp}{\texttt{github.com/yaircarmon/recapp}}.}

As an additional application of our framework, we consider the problem of minimizing $F(x)=\max_{y\in\yset} f(x,y)$ for a function $f$ that is $L$-smooth, convex in $x$ and $\mu$-strongly-convex in $y$. The best existing complexity bound for this problem is $O{\prn*{\frac{L R}{\sqrt{\mu \epsilon}} \log \frac{L R}{\sqrt{\mu \epsilon}}}}$ evaluations of $\grad f$, using an extension of APPA/Catalyst to min-max problems~\citep{yang2020catalyst}, and  mirror-prox~\citep{korpelevich1976extragradient,nemirovski2004prox,azizian2020tight}
as the subproblem solver.\footnote{\citet{yang2020catalyst} also establish the same rate for the dual problem of maximizing $\Psi(y) = \min_{x\in\xset} f(x,y)$ over $y$. Since our acceleration framework is primal-only, we are currently unable to remove the logarithmic factor from that rate.} Our framework (also combined with mirror-prox) removes this logarithmic factor, finding a point with expected suboptimality $\epsilon$ in $O\prn*{\frac{LR}{\sqrt{\mu \epsilon}}}$ gradient queries (up to lower order terms), which is asymptotically optimal~\citep{ouyang2021lower}. 
We summarize our complexity bounds in \Cref{table:summary}. 

\newcolumntype{x}[1]{>{\centering\arraybackslash\hspace{0pt}}p{#1}}
\notarxiv{
\newcommand{\mrow}[1]{\multirow{2}{*}[-8pt]{#1}}
}
\arxiv{
\newcommand{\mrow}[1]{\multirow{2}{*}[-8pt]{#1}}
}
\begin{table*}[t]
	\centering
	\renewcommand{\arraystretch}{1}
	\bgroup
	\everymath{\displaystyle}
		\begin{tabular}{ccx{3.0cm}x{2.5cm}x{5cm}c}
			\toprule
		Objective  $F$  &Reg. & $\ApproxProx$  complexity & $\WarmStart$  complexity & {Overall \linebreak complexity} & Ref. \\  \midrule
		 \mrow{general} & \mrow{$\lambda$} & \mrow{$\mathcal{T}_{\textsc{A}}$} & \mrow{$\mathcal{T}_{\textsc{W}}$} & $O\prn*{\mathcal{T}_{\textsc{A}}\sqrt{\frac{\lambda R^2}{\epsilon}}+ \mathcal{T}_{\textsc{W}}}$ & Thm.~\ref{thm:improved-catalyst}\\ 
		 & & & & $O\prn*{\mathcal{T}_{\textsc{A}}\sqrt{\frac{\lambda }{\gamma}}\log \frac{LR^2}{\epsilon}+ \mathcal{T}_{\textsc{W}}}$ & Prop.~\ref{prop:catalyst-sc}\\ \midrule
		\mrow{$\frac{1}{n}\sum_{i\in[n]}f_i(x)$} &  \mrow{$\frac{L}{n}$} & \mrow{$O(n)$} & \mrow{$O(n\log\log n)$} & $O\prn*{  \sqrt{\frac{nLR^2}{\epsilon}} + n\log\log n }$ & \mrow{Thm.~\ref{thm:OptCatfs}} \\ 
		& & & & $\Otil{ \sqrt{\frac{nL}{\gamma}} } +O\prn*{ n\log\log n }$  & \\ \midrule
	\mrow{ $\max_{y\in\yset}f(x,y)$} & \mrow{$\mu$} & \multirow{2}{*}[8pt]{${O}\Par{\frac{L}{\mu}}+\Otil{\sqrt{\frac{L}{\mu}}}$} & \mrow{$\widetilde{O}\Par{\frac{L}{\mu}}$} & $O\Par{\frac{LR}{\sqrt{\mu\epsilon}}} + \widetilde{O}\prn*{\frac{L}{\mu} + \sqrt{\frac{LR^2}{\epsilon}}}$ & 	\mrow{Thm.~\ref{thm:minimax-delta}}\\
	& & & &  $\widetilde{O}\prn*{\frac{L}{\sqrt{\mu\gamma}} + \frac{L}{\mu}} $ & \\
		\bottomrule
	\end{tabular}
	\caption{\label{table:summary}\textbf{Summary of our results.} Throughout, $\epsilon$ denotes the solution accuracy, and $\gamma$ denotes the strong convexity parameter of $F$. For finite sum problems (the middle row) we assume that each $f_i$ is $L$-smooth, and measure complexity in terms of $\grad f_i$ evaluations. For max-structured problem (the bottom) we assume that $f$ is $L$-smooth and $\mu$-strongly-concave in $y$, and measure complexity in $\grad f$ evaluations. The notation $\Otil{\cdot}$ hides a logarithmic term. See \Cref{sec:framework} for the definition of $\ApproxProx$ and $\WarmStart$. }
\egroup
\end{table*}

\paragraph{Technical overview.} Our development consists of four key parts. First, we define a criterion on the function-value error of the proximal point computation (\Cref{def:approx-prox}) that significantly relaxes the relative error conditions of prior work; see \Cref{sec:framework-comparison} for a detailed comparison. Second, instead of directly bounding the distance error of the approximate proximal points (as most prior works implicitly do), we follow \citet{asi2021stochastic} and require an \emph{unbiased} estimator of the proximal point whose variance is bounded similarly to the function-value error (\Cref{def:unbiased-prox}). We prove that any approximation satisfying these guarantees has the same convergence bounds on its (expected) error as the exact accelerated proximal points method. Third, we use the multilevel Monte Carlo technique~\citep{giles2015multilevel,blanchet2015unbiased,asi2021stochastic} to obtain the required unbiased proximal point estimator using (in expectation) a constant number of queries to any method satisfying the function-value error criterion. Finally, we show how to use SVRG and mirror-prox to efficiently meet our error criterion, allowing us to solve finite-sum and minimax optimization problems without the typical extra logarithmic factors incurred by previous proximal point frameworks.

Even though we maintain the same iteration structure as APPA/Catalyst, our novel error criterion induces two nontrivial modifications to the algorithm. First and foremost, our relaxed error bound depends on the previous approximate proximal point $x_{t-1}$ as well as the current query point $s_{t-1}$ (see \Cref{alg:catalyst}). This dependence strongly suggests that the subproblem solver should depend on $x_{t-1}$ somehow. For finite-sum problems we use $x_{t-1}$ as the reference point for variance reduction, while for max-structured problems we initialize mirror-prox with $x=s_{t-1}$ and (approximately) $y=\argmax_{y\in\yset} f(x_{t-1}, y)$. The second algorithmic 
consequence, which appeared previously in~\citet{asi2021stochastic}, stems from the fact that our function-value error and zero-bias/bounded-variance requirements are leveraged for distinct parts of the algorithm (the prox step and gradient step, respectively). This naturally leads to using distinct approximate prox points for each part: one directly obtained from the subproblem solver and one debiased via MLMC.

\arxiv{%
\section{Additional Related Work}\label{app:related}

Beyond the closely related work already described, our paper touches on several lines of literature. 

\paragraph{Finite-sum problems.}
The ubiquity of finite-sum optimization problems in machine learning has led to a very large body of work on developing efficient algorithms for solving them.  We refer the reader to~\citet{gower2020variance} for a broad survey and  focus on \emph{accelerated} finite-sum methods, i.e., with a leading order complexity term scaling as $\sqrt{n/\epsilon}$ (or as $\sqrt{n\kappa}$ for strongly-convex problems with condition number $\kappa$). 
Accelerated Proximal Stochastic Dual Coordinate Ascent~\citep{shalev2014accelerated} gave the first such accelerated rate for an important subclass of finite-sum problems. This method was subsequently interpreted as a special case of APPA/Catalyst~\cite{lin2015universal,frostig2015regularizing}, which  can also accelerate several other finite-sum optimization problems. Since then, research has focused on designing more practical and theoretically efficient accelerated algoirthms by opening the APPA/Catalyst black box. The algorithms Katyusha~\cite{allen2016katyusha}, Varag~\cite{lan2019unified} and VRADA~\citep{song2020variance} offer improved complexity bound at the price of the generality and simplicity of APPA/Catalyst. Our approach matches the best existing guarantee (due to VRADA) without paying this price.

\paragraph{Max-structured problems.}

Objectives of the form $F(x) = \max_{y\in\yset} f(x,y)$ are very common in machine learning and beyond. Such objectives arise from constraints (via Lagrange multipliers)~\citep{bertsekas1999nonlinear}, robustness requirements~\citep{bental2009robust,ganin2016domain,madry2018towards}, and game-theoretic considerations~\citep{morgenstern1953theory,silver2017mastering}. When $f$ is convex in $x$ and concave in $y$, the mirror-prox algorithm minimizes $F$ to accuracy epsilon in $O(LR R'/\epsilon)$ gradient evaluations (with respect to both $x$ and $y$), where $R'$ is the diameter of $\yset$. This rate can be improved when $f$ is $\mu$-strongly-concave in $y$. For the special bilinear case $f(x,y) = \phi(x) + \inner{y}{Ax} - \psi(y)$, where $\psi$ a ``simple'' $\mu$-strongly-convex function, an improved complexity bound of $O(LR/\sqrt{\mu \epsilon})$ has long been known~\citep{nesterov2005smooth}.

More recent work studies the case of general convex-strongly-concave $f$. \citet{thekumparampil2019efficient} and \citet{zhao2020primal} establish complexity bounds of $O(\frac{L^{3/2}}{\mu\sqrt{\epsilon}} \log^2 \frac{L^2 R R'}{\mu\epsilon})$, which~\citet{lin2020near} improve to $O(\frac{L}{\sqrt{\mu\epsilon}} \log^3 \frac{L^2 R R'}{\mu\epsilon})$ using an algorithm loosely based on APPA/Catalyst. \citet{yang2020catalyst} present a more direct application of APPA/Catalyst to min-max problems, further improving the complexity to $O(\frac{L}{\sqrt{\mu\epsilon}} \log \frac{L^2 R^2}{\mu\epsilon})$, with logarithmic dependence on $R'$ only in a lower order term. Similarly to standard APPA/Catalyst, the min-max variant requires highly accurate proximal point computation, e.g., to function-value error of $O(\frac{\mu^3 \epsilon^2}{L^4 R^2})$. In contrast, $\ouralg$ requires constant (relative) suboptimality and removes the final logarithmic factor from the leading-order complexity term.  \citet{yang2020catalyst} also provide extensions to finite-sum min-max problems and problems where $f$ is non-convex in $x$, which would likely benefit from out method as well (see \Cref{sec:discussion}).

\emph{Independent work.} In recent independent work, 
 \citet{kovalev2022first}  develop a method that minimizes  $\max_{y}f(x,y)$ assuming $\mu$-strong-concavity in $y$ and $\gamma$-strong-convexity in $x$. They attain an essentially optimal complexity proportional to $\frac{1}{\sqrt{\mu\gamma}}$ times a logarithmic factor depending on problem parameters. Their method is tailored to saddle point problems, working in  an expanded space by using point-wise conjugate function and applying recent advances in monotone operator theory. We note that RECAPP with restarts attains the same complexity bound (see \Cref{thm:minimax-delta}). However, it is unclear whether the algorithm of \citep{kovalev2022first} can recover the RECAPP's complexity bound in the setting where $f$ is not strongly convex in $x$.

\paragraph{Monteiro-Svaiter-type acceleration.}
Monteiro-Svaiter~\citep{monteiro2013accelerated} propose a variant of the accelerated proximal point method that uses an additional gradient evaluation to facilitate approximate proximal point computation. The Monteiro-Svaiter method and its extensions~\citep{gasnikov19near,bubeck2019complexity,bullins2020highly,carmon2020acceleration,song2021unified,kovalev2022first,carmon2022optimal} also allow for the regularization parameter $\lambda$ to be determined  dynamically by the procedure approximating the proximal point. \citet{ivanova2021adaptive} leverage this technique to develop a variant of Catalyst that offers improved adaptivity and, in certain cases, improved complexity. We provide additional comparison between the approximation condition of~\cite{monteiro2013accelerated,ivanova2021adaptive} and RECAPP in \Cref{sec:framework-comparison}.

\paragraph{Multilevel Monte Carlo (MLMC).}
MLMC is a method for debiasing function estimators by randomizing over the level of accuracy~\citep{giles2015multilevel}.
While originally conceived for PDEs and system simulation, a particular variant of MLMC due to~\citet{blanchet2015unbiased} has found recent applications in the theory of stochastic optimization~\citep{levy2020large,hu2021bias}.
Our method directly builds on the recent proposal of~\citet{asi2021stochastic} to use MLMC in order to obtain unbiased estimates of proximal points (or, equivalently, the Moreau envelope gradient). 
\citet{asi2021stochastic} apply this estimator to de-bias proximal points estimated via SGD and improve several structured acceleration schemes. 
In contrast, we apply MLMC on linearly convergent algorithms, allowing us to configure it much more aggressively and avoid the extraneous logarithmic factors that appeared in the rates of \citet{asi2021stochastic}.
}

\paragraph{Paper organization.} After providing some notation and preliminaries in \Cref{sec:prelims}, we present our improved inexact accelerated proximal point framework in \Cref{sec:framework}. We then instantiate our framework: in \Cref{sec:fs} we consider finite-sum problems and SVRG (providing preliminary empirical results in \Cref{ssec:experiments}) and in \Cref{sec:minimax} we consider min-max problems and mirror-prox. \arxiv{We conclude in \Cref{sec:discussion} by discussing limitations and possible extensions.}\notarxiv{We provide additional discussion of related work, including recent independent work by~\citet{kovalev2022first}, in~\Cref{app:related}. The rest of the appendix is composed of the proofs for each corresponding section, followed by \Cref{sec:discussion} which provides a discussion of limitations and possible extensions of this work.}

\section{Preliminaries}\label{sec:prelims}

\paragraph{General notation.} Throughout, $\xset$ and $\yset$ refer to closed, convex sets, with diameters denoted by $R$ and $R'$ respectively (when needed).
We use $F$ to denote a convex function defined on $\xset$. For any parameter  $\lambda>0$ and point $s \in \xset$, we let 
\begin{equation}\label{def:prox-function}
\Flam_{s}(x) \defeq F(x)+\frac{\lambda}{2}\norm{x-s}^2
\end{equation}
denote the proximal regularization of $F$ around $x$, and let $\prox(s) = \argmin_{x\in\xset} F_s^\lambda(x)$.

\paragraph{Distances and norms.} We consider Euclidean space throughout the paper and use $\norm{\cdot}$ to denote standard Euclidean norm. We denote a \emph{projection} of $x\in\R^d$ onto a closed subspace $\xset\subseteq\R^d$ by $\proj_{\xset}(x)= \arg\min_{x'\in\xset}\norm{x-x'}$. For a convex function $F : \xset \rightarrow \R$, we denote the \emph{Bregman divergence} induced by $F$ as
\arxiv{
\begin{equation*}
	V^F_{x}(x') \defeq F(x')-F(x)-\inprod{\nabla F(x)}{x'-x},
\end{equation*}
}
\notarxiv{
$V^F_{x}(x') \defeq F(x')-F(x)-\inprod{\nabla F(x)}{x'-x},$
}
for every $x,x'\in\xset$. We denote the Euclidean Bregman divergence by
\arxiv{
\begin{equation*}
	\Veu_{x}(x') \defeq V_x^{\frac{1}{2}\norm{\cdot}^2}(x') = \frac{1}{2}\norm{x'-x}^2.
\end{equation*}
}
\notarxiv{
$\Veu_{x}(x') \defeq V_x^{\frac{1}{2}\norm{\cdot}^2}(x') = \frac{1}{2}\norm{x'-x}^2.$
}

\paragraph{Smoothness, convexity and concavity.} Given a differentiable, convex function $F: \xset \rightarrow \R$, we say $F$ is \emph{$L$-smooth} if its gradient $\nabla F:\xset\rightarrow \xset^*$ is $L$-Lipschitz. We say $F$ is \emph{$\mu$-strongly-convex} if for all $x, x'\in \xset$,
\arxiv{\[
F(x')\ge F(x)+\langle\nabla F(x),x'-x\rangle+\frac{\mu}{2}\norm{x'-x}^2.
\]}
\notarxiv{
$F(x')\ge F(x)+\langle\nabla F(x),x'-x\rangle+\frac{\mu}{2}\norm{x'-x}^2$.
}
A function $\Psi$ is $\mu$-strongly concave if $-\Psi$ is $\mu$-strongly convex.
For $f(x,y)$ that is convex in $x$ and concave in $y$, the point $(\xopt,\yopt)$ is a saddle-point if $\max_{y\in\yset}f(\xopt,y)\le f(\xopt,\yopt)\le \min_{x\in\xset}f(x,\yopt)$ for all $x,y\in\xset\times\yset$.
\section{Framework}\label{sec:framework}

In this section, we present our Relaxed Error Criterion Accelerated Proximal Point ($\ouralg$) framework. We start by defining our central algorithms and relaxed error criteria (\Cref{sec:framework-defs}). Next, we state our main complexity bounds (\Cref{sec:framework-complexity}) and sketch its proof. Then, we illustrate our new relaxed error criterion by comparing it to the error requirements of prior work (\Cref{sec:framework-comparison}). Finally, as an illustrative warm-up, we show our framework easily recovers the complexity bound of Nesterov's classical accelerated gradient descent (AGD) method~\citep{nesterov1983method}.

\subsection{Methods and Key Definitions}\label{sec:framework-defs}

\arxiv{
\begin{figure}[t]
	\begin{minipage}[t]{0.47\linewidth}%
		\centering
\begin{algorithm2e}[H]
	\caption{$\ouralg$}
	\label{alg:catalyst}
	\DontPrintSemicolon
	{\bfseries Parameters:} $\lambda>0$, step budget $T$ \;
	\label{line:catalyst:warm_start_start}Initialize $\alpha_0\gets 1$ and  $x_0=v_0\leftarrow \WarmStart_{F,\lambda}(R^2)$\;
	\Comment*{To satisfy $\E F(x_0)-F(\xopt)\le \lambda R^2$}
	\For{$t=0$ {\bfseries{\textup{to}}} $T-1$ \label{line:catalyst:for_start}}{
		\label{line:catalyst:convex_start}Set $\alpha_{t+1}\in[0,1]$ to satisfy $\frac{1}{\alpha_{t+1}^2}-\frac{1}{\alpha_{t+1}} = \frac{1}{\alpha_t^2}$ \;
		\label{line:catalyst:convex_end}$s_t\gets \Par{1-\alpha_{t+1}}x_t + \alpha_{t+1}v_t  $ \;
		\label{line:catalyst:approx}$x_{t+1}\gets \ApproxProx_{F,\lambda}(s_t; s_t, x_t)$ \;
		\label{line:catalyst:unbias_start}$\tx_{t+1}\gets \UnbiasedProx_{F,\lambda}(s_t; x_t)$ \; %
		\label{line:catalyst:unbias_end}$v_{t+1}\gets \proj_{\xset}\prn[\big]{v_t-\frac{1}{\alpha_{t+1}}\Par{s_t-\tx_{t+1}}}$ %
	}\label{line:catalyst:for_end}
	\codeReturn $x_T$
\end{algorithm2e}
\end{minipage}\hfill
	\begin{minipage}[t]{0.52\linewidth}%
		\centering
\begin{algorithm2e}[H]
	\caption{$\UnbiasedProx$ via MLMC}
	\label{alg:unbiased-min-MLMC}
	\DontPrintSemicolon
	\codeInput $\ApproxProx$,  points $s,\xprev\in\xset$  \;
	\codeParameter Geometric distribution parameter $p\in[0,1)$ and integer offset $j_0 \geq 0$\;
	\codeOutput Unbiased estimator of $x^\star = \prox(s)$ \;
	$x^{(0)}\gets \ApproxProx_{F,\lambda}(s; s, \xprev)$ \;
	Sample $J_+ \sim \geoRV\left(1 - p\right)\in \{0, 1, 2,\ldots\}$\;
	$J \gets j_0 + J_+$\;
	\For{$j= 0$ {\bfseries{\textup{to}}} $J -1 $}{
		$x^{(j + 1)}\gets\ApproxProx_{F,\lambda}(s; x^{{(j)}}, x^{{(j)}})$\;
	}
 	$p_J \gets \P[\geoRV\left(1 - p\right) = J_+] = (1 - p) \cdot p^{J_+}$\;
	\codeReturn $x^{(j_0)} + p_J^{-1} (x^{(J)}-x^{(\max\{J-1,j_0\})})$ \;
\end{algorithm2e}
	\end{minipage}
\end{figure}
}

\Cref{alg:catalyst} describes our core accelerated proximal method. 
The algorithm follows the standard template of the (inexact) accelerated proximal point method, except that unlike most such methods (but similar to the methods of~\citet{asi2021stochastic}), our algorithm relies on two distinct approximations of $\prox(s_t)$ with different relaxed error criterion. We now define each approximation in turn.

Our first relaxed error criterion constrains the function value of the approximate proximal point and constitutes our key contribution.

\notarxiv{
\begin{algorithm2e}[t]
	\caption{$\ouralg$}
	\label{alg:catalyst}
	\DontPrintSemicolon
	{\bfseries Parameters:} $\lambda>0$, step budget $T$ \;
	\label{line:catalyst:warm_start_start}Initialize $\alpha_0\gets 1$ and  $x_0=v_0\leftarrow \WarmStart_{F,\lambda}(R^2)$\;
	\Comment*{To satisfy $\E F(x_0)-F(\xopt)\le \lambda R^2$}
	\For{$t=0$ {\bfseries{\textup{to}}} $T-1$ \label{line:catalyst:for_start}}{
		\label{line:catalyst:convex_start}Set $\alpha_{t+1}\in[0,1]$ to satisfy $\frac{1}{\alpha_{t+1}^2}-\frac{1}{\alpha_{t+1}} = \frac{1}{\alpha_t^2}$ \;
		\label{line:catalyst:convex_end}$s_t\gets \Par{1-\alpha_{t+1}}x_t + \alpha_{t+1}v_t  $ \;
		\label{line:catalyst:approx}$x_{t+1}\gets \ApproxProx_{F,\lambda}(s_t; s_t, x_t)$ \;
		\label{line:catalyst:unbias_start}$\tx_{t+1}\gets \UnbiasedProx_{F,\lambda}(s_t; x_t)$ \; %
		\label{line:catalyst:unbias_end}$v_{t+1}\gets \proj_{\xset}\prn[\big]{v_t-\frac{1}{\alpha_{t+1}}\Par{s_t-\tx_{t+1}}}$ %
	}\label{line:catalyst:for_end}
	\codeReturn $x_T$
\end{algorithm2e}
}

\begin{definition}[$\ApproxProx$]\label{def:approx-prox}
Given convex function $F:\xset\to \R$, parameter $\lambda>0$, and points $s, \xinit, \xprev\in\xset$, the point $x=\ApproxProx_{F,\lambda}(s;\xinit,\xprev)$ is an approximate minimizer of $F^\lambda_s(x) \defeq F(x) + \lambda \Veu_{s}(x)$ such that for $x^\star \defeq  \prox(s)= \arg\min_{x\in\xset} F^\lambda_s(x)$, 
\begin{gather}
\E \Flam_s(x) - \Flam_s(x^\star)\le \frac{\lambda\Veu_{x^\star}\Par{\xinit}+V^F_{x^\star}(\xprev)}{8}\,.\label{approx-prox-cond}
\end{gather}
\end{definition}

Beyond the prox-center $s$, our robust error criterion depends on two additional points: $\xinit$ (which in \Cref{alg:catalyst} is also set the prox-center $s_t$) and $\xprev$ (which in \Cref{alg:catalyst} is set to the previous iterate $x_t$). The criterion requires the suboptimality of the approximate solution to be bounded by weighted combination of two distances: the Euclidean distance between the true proximal point $\xopt$ and $\xinit$, and the Bregman divergence (induced by $F$) between $\xopt$ and $\xprev$. In \Cref{sec:framework-comparison} we provide a detailed comparison between our criterion and prior work, but note already that---unlike APPA/Catalyst---the relative error we require in~\eqref{approx-prox-cond} is \emph{constant}, i.e., independent of the desired accuracy or number of iterations. This constant level of error is key to enabling our improved complexity bounds.

Our second relaxed error criterion constrains the bias and variance of the approximate proximal point.

\begin{definition}[$\UnbiasedProx$]\label{def:unbiased-prox}
Given convex function $F:\xset\to \R$, parameter $\lambda>0$, and points $s$ and $\xprev\in\xset$, the point $x=\UnbiasedProx_{F,\lambda}(s;\xprev)$ is an approximate minimizer of $F^\lambda_s(x) = F(x) + \lambda \Veu_{s}(x)$ such that $\E~x = x^\star = \prox(s)= \arg\min_{x\in\xset} F^\lambda_s(x)$, and 
\begin{equation}
	\E\norm{x-\xopt}^2\le \frac{\lambda\Veu_{x^\star}\Par{s}+V^F_{x^\star}(\xprev)}{4\lambda}.
	\label{unbiased-prox-cond}
\end{equation}
\end{definition}

Note that the any $x=\ApproxProx_{F,\lambda}(s;s,\xprev)$ satisfies the distance bound~\eqref{unbiased-prox-cond} (due to $\lambda$-strong-convexity of $F_s^\lambda$), but the zero-bias criterion $\E x = \xopt$ is not guaranteed. Nevertheless, an MLMC technique (\Cref{alg:unbiased-min-MLMC}) can extract an $\UnbiasedProx$ from any $\ApproxProx$. \Cref{alg:unbiased-min-MLMC} repeatedly calls $\ApproxProx$ a geometrically-distributed number of times $J$ (every time with $\xinit$ and $\xprev$ equal to the last output), and outputs a point whose expectation equals to an infinite numbers of iterations of $\ApproxProx$, i.e., the exact $\prox(s)$. Moreover, we show that the linear convergence of the procedure implies that the variance of the result remains appropriately bounded (see \Cref{lem:MLMC} below). 
\Cref{alg:unbiased-min-MLMC} is a variation of an estimator by \citet{blanchet2015unbiased} that was previously used in a context similar to ours \citep{asi2021stochastic}. However, prior estimators typically have complexity exponential in $J$, whereas ours are linear in $J$. 

Finally, we define a warm start procedure required by our method.

\begin{definition}[$\WarmStart$]\label{def:warmstart}
	Given convex function $F:\xset\to \R$, parameter $\lambda>0$ and diameter bound $R$, $ x_0=\WarmStart_{F,\lambda} (R^2)$ is a procedure that outputs $x_0 \in \xset$ such that $\E F(x_0)-\min_{x' \in \xset}F(x')\le \lambda R^2$.
\end{definition}

Note that the exact proximal mapping $x=\prox(s)$ satisfies all the requirements above; replacing $\ApproxProx_{F,\lambda}$, $\UnbiasedProx_{F,\lambda}$, and $\WarmStart_{F,\lambda}$ with $\prox$ recovers the exact accelerated proximal method.

\notarxiv{
	\begin{algorithm2e}[t]
		\caption{$\UnbiasedProx$ via MLMC}
		\label{alg:unbiased-min-MLMC}
		\DontPrintSemicolon
		\codeInput $\ApproxProx$,  points $s,\xprev\in\xset$  \;
		\codeParameter Geometric distribution parameter $p\in[0,1)$ and integer offset $j_0 \geq 0$\;
		\codeOutput Unbiased estimator of $x^\star = \prox(s)$ \;
		$x^{(0)}\gets \ApproxProx_{F,\lambda}(s; s, \xprev)$ \;
		Sample $J_+ \sim \geoRV\left(1 - p\right)\in \{0, 1, 2,\ldots\}$\;
		$J \gets j_0 + J_+$\;
		\For{$j= 0$ {\bfseries{\textup{to}}} $J -1 $}{
			$x^{(j + 1)}\gets\ApproxProx_{F,\lambda}(s; x^{{(j)}}, x^{{(j)}})$\;
		}
		$p_J \gets \P[\geoRV\left(1 - p\right) = J_+] = (1 - p) \cdot p^{J_+}$\;
		\codeReturn $x^{(j_0)} + p_J^{-1} (x^{(J)}-x^{(\max\{J-1,j_0\})})$ \;
	\end{algorithm2e}
}

\subsection{Complexity Bounds}\label{sec:framework-complexity}

We begin with a complexity bound for implementing $\UnbiasedProx$ via \Cref{alg:unbiased-min-MLMC} (proved in \Cref{apdx:framework}).

\begin{restatable}[MLMC turns $\ApproxProx$ into $\UnbiasedProx$]{proposition}{lemMLMC}\label{lem:MLMC}
	For any convex $F$ and parameter $\lambda>0$,~\Cref{alg:unbiased-min-MLMC} with $p=1/2$ and $j_0 \geq 2$ implements $\UnbiasedProx$ and makes $2 + j_0$ calls to $\ApproxProx$ in expectation. 
\end{restatable}

We now give our complexity bound for $\ouralg$ and sketch its proof, deferring the full proof to~\Cref{apdx:framework}.

\begin{restatable}[$\ouralg$ complexity bound]{theorem}{catalyst}\label{thm:improved-catalyst} Given any convex function $F:\xset \to \R$ and parameters $\lambda,R>0$, $\ouralg$ (\Cref{alg:catalyst}) finds 
	 $x \in \xset$ with $\E F(x) - \min_{x'\in\xset}F(x')\le \eps$, within $O\prn[\big]{\sqrt{\lambda R^2/\epsilon}}$ iterations using one call to $\WarmStart$, and $O\prn[\big]{\sqrt{\lambda R^2/\epsilon}}$ calls to $\ApproxProx$ and $\UnbiasedProx$. If we implement $\UnbiasedProx$ using \Cref{alg:unbiased-min-MLMC} with $p=1/2$ and $j_0 = 2$, the total number of calls to $\ApproxProx$ is $O\prn[\big]{\sqrt{\lambda R^2/\epsilon}}$ in expectation. 
\end{restatable}

\notarxiv{\textbf{Proof sketch.}}
\arxiv{\begin{psketch}}
 We split the proof into two steps.

\newcommand{\vopt}{v^\star}

\emph{Step 1: Tight idealized potential decrease.} Consider iteration $t$ of the algorithm, and define the potential
\begin{equation*}
	P_t \defeq \E \brk*{ \alpha_t^{-2}\Par{F(x_{t}) - F(x')} +  \lambda\Veu_{v_{t}}\Par{x'}},
\end{equation*}
where $x'$ is a minimizer of $F$ in $\xset$. Let $\xopt_{t+1} \defeq \prox(s_t)$ and $\vopt_{t+1} = v_t - (\alpha_{t+1})^{-1} (s_t - \xopt_{t+1})$ be the ``ideal'' values of $x_{t+1}$ and $v_{t+1}$ obtained via an exact prox-point computation, where for simplicity we ignore the projection onto $\xset$. Using these points we define the idealized potential
\begin{equation*}
	P^\star_{t+1} \defeq \E \brk*{ \alpha_{t+1}^{-2}\Par{F(\xopt_{t+1}) - F(x')} +  \lambda\Veu_{\vopt_{t+1}}\Par{x'}}.
\end{equation*}
Textbook analyses of acceleration schemes show that $P^\star_{t+1} \le P_t$~\cite{nesterov2018lectures,monteiro2013accelerated}. Basic inexact accelerated prox-point analyses proceed by showing that the true potential is not much worse that the idealized potential, i.e., 
$P_{t+1} \le P^\star_{t+1} + \delta_t$, which immediately allows one to conclude that $P_T \le P_0 + \Delta_T$ for $\Delta_T =\sum_{t < T} \delta_t$, and therefore that $\E F(x_T) - F(x') \le \alpha_T^2 (\lambda R^2 + \Delta_T)$, implying the optimal rate of convergence as long $\Delta_T = O(\lambda R^2)$. However, obtaining such a small $\Delta_T$ to be that small naively requires approximating the proximal points to very high accuracy, which is precisely what we attempt to avoid.   

Our first step toward a relaxed error criterion is proving stronger idealized potential decrease. We show that, 
\begin{equation*}
	P^\star_{t+1} \le P_t - \E\brk*{\alpha_{t+1}^{-2} \lambda \Veu_{\xopt_{t+1}}(s_t) +  \alpha_{t}^{-2} V^F_{\xopt_{t+1}}(x_t)}.
\end{equation*}
While the potential decrease term $\alpha_{t+1}^{-2} \lambda \Veu_{\xopt_{t+1}}(s_t)$ is well known and has been thoroughly exploited by prior work~\citep{frostig2015regularizing,lin2017catalyst,monteiro2013accelerated}, making use of the term $ \alpha_{t}^{-2} V^F_{\xopt_{t+1}}(x_t)$ is, to the best of our knowledge, new to this work. 

\emph{Step 2: Matching approximation errors.}
With the improved potential decrease at bound in hand, our strategy is clear: make the approximation error cancel with the potential decrease. That is, we wish to show
\begin{equation*}
	P_{t+1} \le P^\star_{t+1} + \E\brk*{\alpha_{t+1}^{-2} \lambda \Veu_{\xopt_{t+1}}(s_t) + \alpha_{t}^{-2} V^F_{\xopt_{t+1}}(x_t)},
\end{equation*}
so that overall we have $P_{t+1} \le P_{t}$ and consequently 
 $\E F(x_T) - F(x') = O\prn*{\alpha_T^2 ( F(x_0) - F(x') + \lambda R^2 } = O\prn*{\lambda R^2 / T^2}$, with the last bound following form the warm-start condition and $\alpha_t = O(1/t)$. 

It remains to show that the $\ApproxProx$ and $\UnbiasedProx$ criteria provide the needed error bounds. For the function value, \Cref{def:approx-prox} implies
\arxiv{
\begin{flalign*}
	\E F(x_{t+1})  & = \E \brk*{F_{s_t}^\lambda(x_{t+1}) -\lambda\Veu_{x_{t+1}}\Par{s_t}}
	 \\&
	\le \E\brk*{F_{s_t}^\lambda(\xopt_{t+1}) + \tfrac{\lambda\Veu_{\xopt_{t+1}}\Par{s_t}+V^F_{\xopt_{t+1}}(x_t)}{8}-\lambda\Veu_{x_{t+1}}\Par{s_t}}
	\\&
	=
	\E \brk*{F(\xopt_{t+1}) + \tfrac{7}{8}\lambda \Veu_{\xopt_{t+1}}\Par{s_t} + \tfrac{1}{4}\lambda \Veu_{\xopt_{t+1}}\Par{s_t}+\tfrac{1}{8}V^F_{\xopt_{t+1}}(x_t)-\lambda\Veu_{x_{t+1}}\Par{s_t}}\\
	& \le \E \brk*{F(\xopt_{t+1}) + \tfrac{7}{8}\lambda \Veu_{\xopt_{t+1}}\Par{s_t} + \tfrac{1}{8}V^F_{\xopt_{t+1}}(x_t)+\tfrac{1}{2}\lambda \Veu_{\xopt_{t+1}}\Par{x_{t+1}}+\tfrac{1}{2}\lambda \Veu_{x_{t+1}}\Par{s_{t}}-\lambda \Veu_{x_{t+1}}\Par{s_t}}\\
	& \le F\left(\xopt_{t+1}\right)+\tfrac{7}{8}\lambda\Veu_{\xopt_{t+1}}\Par{s_t} + \tfrac{5}{24}V^F_{\xopt_{t+1}} \Par{x_t},
\end{flalign*}
}
\notarxiv{
\begin{flalign*}
	& \E F(x_{t+1}) - F(s_t)   = \E \brk*{F_{s_t}^\lambda(x_{t+1}) -\lambda\Veu_{x_{t+1}}\Par{s_t}}
	 \\& \hspace{2em}
	\le \E\brk*{F_{s_t}^\lambda(\xopt_{t+1}) + \tfrac{\lambda\Veu_{\xopt_{t+1}}\Par{s_t}+V^F_{\xopt_{t+1}}(x_t)}{8}-\lambda\Veu_{x_{t+1}}\Par{s_t}}
	\\
	& \hspace{2em}
	\le F\left(\xopt_{t+1}\right)-F(s_t )+\tfrac{7}{8}\lambda\Veu_{\xopt_{t+1}}\Par{s_t} + \tfrac{5}{24}V^F_{\xopt_{t+1}} \Par{x_t},
\end{flalign*}

}
where for the last inequality we use the property that $\E[\tfrac{1}{2}\lambda \Veu_{\xopt_{t+1}}\Par{x_{t+1}}]\le \E[\tfrac{1}{6}\lambda \Veu_{x_{t+1}}\Par{s_{t}}]+\tfrac{1}{12}V^F_{\xopt_{t+1}}\Par{x_t}$ due to the strong convexity of $F^\lambda$ and the approximate optimality of $x_{t+1}$ guaranteed by $\ApproxProx$ (see~\eqref{eq:catalyst-helper} and the derivations before it in Appendix).

For $\Veu_{\vopt_{t+1}}\Par{x'}$, \Cref{def:unbiased-prox} implies $\vopt_{t+1} = \E v_{t+1}$ and
\begin{flalign*}
	&\E\Veu_{v_{t+1}}\Par{x'} - \E\Veu_{\vopt_{t+1}}\Par{x'} =   (\alpha_{t+1})^{-2} \tfrac{1}{2} \E \norm{\xopt_{t+1} - \tx_{t+1}}^2
	\notarxiv{\\ & ~~~~~~~~}
	\le (\alpha_{t+1})^{-2} \cdot \tfrac{\lambda\Veu_{\xopt_{t+1}}\Par{s_t}+V^F_{\xopt_{t+1}}(x_t)}{8\lambda}.
\end{flalign*}
Substituting back into the expressions for $P_{t+1}$ and $P^\star_{t+1}$ and using the fact that $\frac{\alpha_{t+1}^2}{\alpha_t^2} \ge \frac{1}{3}$, we obtain the desired bound on $P_{t+1}-P^\star_{t+1}$ and conclude the proof.\footnote{The need to have a lower bound like $\frac{\alpha_{t+1}^2}{\alpha_t^2} \ge \frac{1}{3}$ is the reason $\ouralg$ does not take $\alpha_0 = \infty$ and requires a warm-start.}

\arxiv{\end{psketch}}

\notarxiv{\textbf{Complexity bound for strongly-convex functions.}}
\arxiv{\paragraph{Complexity bound for strongly-convex functions.}}
For completeness, we also include a guarantee for minimizing a strongly-convex function $F$ by restarting $\ouralg$. See \Cref{apdx:framework} for pseudocode and proofs.

\begin{restatable}[$\ouralg$ for strongly-convex functions]{proposition}{catalystsc}\label{prop:catalyst-sc}
	For any $\gamma$-strongly-convex function $F:\xset \to \R$, and parameters $\lambda\ge \gamma$, $R>0$, restarted $\ouralg$ (\Cref{alg:catalyst-sc}) finds $x$ such that $\E F(x) - \min_{x'\in\xset}F(x')\le \eps$, using one call to $\WarmStart$, and $O\prn[\big]{\sqrt{\lambda/\gamma}\log\frac{LR^2}{\epsilon}}$ calls to $\ApproxProx$ and $\UnbiasedProx$. If we implement $\UnbiasedProx$ using \Cref{alg:unbiased-min-MLMC} with $p=1/2$ and $j_0 = 2$, the number of calls to $\ApproxProx$ is $O\prn[\big]{\sqrt{\lambda/\gamma}\log\frac{LR^2}{\epsilon}}$ in expectation.
\end{restatable}

\subsection{Comparisons of Error Criteria}\label{sec:framework-comparison}

We now compare $\ApproxProx$ (\Cref{def:approx-prox}) to other proximal-point error criteria from the literature. Throughout, we fix a center-point $s$ and let $\xopt = \prox(s)$.

\notarxiv{\textbf{Comparison with \citet{frostig2015regularizing}.}}
\arxiv{\paragraph{Comparison with \citet{frostig2015regularizing}.}}
 The APPA framework, which focuses on $\gamma$-strongly-convex functions, requires the function-value error bound
\begin{equation*}
	F^\lambda_s\Par{x}-F^\lambda_s\Par{\xopt}\le O\prn*{\prn*{\frac{\gamma}{\lambda}}^{1.5}}\Par{F^\lambda_s\Par{\xinit}-F^\lambda_s\Par{\xopt}}
\end{equation*}
to hold for all $\xinit$. 
To compare this requirement with $\ApproxProx$, note that in the unconstrained setting
\begin{equation*}
\begin{aligned}
	\lambda\Veu_{x^\star}\Par{\xinit}+V^F_{x^\star}(\xinit) & = 
	V^{F_s^\lambda}_{x^\star}(\xinit)\\
	& = F_s^\lambda(\xinit) - F_s^\lambda(\xopt),
	\end{aligned}
\end{equation*}
where the last equality is due to the fact that $\xopt$ minimizes $F_s^\lambda$ and therefore $\inner{\grad F_s^\lambda(\xopt)}{\xopt-\xinit} = 0$. Consequently, the error of $\ApproxProx_{F,\lambda}(s;\xinit,\xinit)$ is 
\begin{equation*}
	F^\lambda_s\Par{x}-F^\lambda_s\Par{\xopt}\le \frac{1}{8}\Par{F^\lambda_s\Par{\xinit}-F^\lambda_s\Par{\xopt}}.
\end{equation*}
Therefore, in the unconstrained setting and the special case of $\xprev=\xinit$ we require a constant factor relative error decrease, while APPA requires decrease by a factor proportional to $(\gamma/\lambda)^{3/2}$, or $(\epsilon /(\lambda R^2))^{3/2}$ with the standard conversion $\gamma = \epsilon / R^2$. Thus, our requirement is significantly more permissive. 

\notarxiv{\textbf{Comparison with \citet{lin2017catalyst}.}}
\arxiv{\paragraph{Comparison with \citet{lin2017catalyst}.}}
 The Catalyst framework offers a number of error criteria. Most closely resembling $\ApproxProx$ is their relative error criterion (C2):
\begin{equation*}
	F^\lambda_s\Par{x}-F^\lambda_s\Par{\xopt}\le \delta_{t} \lambda \Veu_{x}(s),
\end{equation*}
with $\delta_{t} = (t+1)^{-2}$ which is of the order of $\epsilon / (\lambda R^2)$ for most iterations. Setting $\delta_t = 1/10$ in the Catalyst criterion would satisfy $\ApproxProx_{F,\lambda}(s;s,x')$ for any $x'$. Furthermore, $\ApproxProx$ allows for an additional error term proportional to $V^F_{x^\star}(\xinit)$, which does not exist in Catalyst. In our analysis in the next sections this additional term is essential to efficiently satisfy our criterion.

\notarxiv{\textbf{Comparison with the Monteiro-Svaiter (MS) condition.}}
\arxiv{\paragraph{Comparison with the Monteiro-Svaiter (MS) condition.}}
\citet{ivanova2021adaptive} and \citet{monteiro2013accelerated} consider the error criterion
\[\norm{\nabla F^\lambda_s(x)}\le \sigma \lambda \norm{x-s}.\]
This criterion implies the bound   $F^\lambda_s\Par{x}-F^\lambda_s\Par{\xopt}\le\frac{1}{\lambda}\norm{\nabla F^\lambda_s(x)}^2\le  2\sigma^2 \lambda \Veu_{x}(s)$, making it stronger that the Catalyst C2 criterion when $\sigma = \sqrt{\delta_t/2}$. However, \citet{monteiro2013accelerated} show that by updating of $v_t$ using  $\grad F(x_{t+1})$, any \emph{constant} value of $\sigma$ in $[0,1)$ suffices for obtaining rates similar to those of the exact accelerated proximal point method. The $\ApproxProx$ criterion is strictly weaker than the MS criterion with $\sigma=1/5$. 

\citet{ivanova2021adaptive} leverage the MS framework and its improved error tolerance to develop a reduction-based method that, for some problems, is more efficient than Catalyst by a logarithmic factor. However, for the finite sum and max-structured problems we consider in the following sections, it is unclear how to satisfy the MS condition without incurring an extraneous logarithmic complexity term.

\notarxiv{\textbf{Stochastic error criteria.}}
\arxiv{\paragraph{Stochastic error criteria.}}
It is important to note that in contrast to APPA and Catalyst, $\ouralg$ is inherently randomized. The unbiased condition of $\UnbiasedProx$ is critical to our analysis of the update to $v_t$. Although in many cases (such as in finite-sum optimization) efficient proximal point oracles require randomization anyhow, we extend the use of randomness to the acceleration framework's update itself. It is an interesting question to determine if this randomization is necessary and comparable performance to $\ouralg$ can be obtained based solely on deterministic applications of $\ApproxProx$.

\subsection{The AGD Rate as a Special Case}\label{sec:framework-agd}
For a quick demonstration of our framework, we show how to recover the classical $\sqrt{LR^2/\epsilon}$ complexity bound for minimizing an $L$-smooth function $F$ using exact gradient computations. To do so, we set $\lambda = L$ and note that $F^\lambda_s$ is $L$-strongly convex and $2L$-smooth. Therefore, for each $F^\lambda_s$ with $\xopt=\prox(s)$ we can implement $\ApproxProx$ by taking $4$ gradient steps starting from $\xinit$, since these steps produce an $x$ satisfying
\begin{equation*}
	\Veu_{\xopt}\Par{x} \le (1-\tfrac{1}{2})^4 \Veu_{\xopt}\Par{\xinit} =  \frac{1}{16}\Veu_{\xopt}\Par{\xinit}
\end{equation*}
and therefore
\begin{gather*}
	F^\lambda_s\Par{x}-F^\lambda_s\Par{\xopt}\le 2L \Veu_{\xopt}\Par{x}\le \frac{L}{8}\Veu_{\xopt}\Par{\xinit}.
\end{gather*}

Invoking \Cref{thm:improved-catalyst} with $\lambda=L$ shows that $\ouralg$ finds an $\epsilon$-approximate solution with $O(R\sqrt{L/\epsilon})$ gradient queries, recovering the result of~\citet{nesterov1983method}.

\section{Finite-sum Minimization}\label{sec:fs}

\arxiv{
\begin{figure}[t]
	\begin{minipage}[t]{0.465\linewidth}%
		\centering
\begin{algorithm2e}[H]
	\caption{$\SVRGone$}
	\label{alg:SVRG}
	\DontPrintSemicolon
	\codeInput $\Phi = \frac{1}{n}\sum_{i\in[n]}\phi_i$ (with component gradient oracles), center point $\xfull$, initial point $\xinit$, step-size $\eta$, iteration number $T$

	Query gradient $\nabla \fun(\xfull) = \frac{1}{n}\sum_{i\in[n]}\nabla \phi_i( \xfull)$ \;
	$x_0 \gets \xinit$\; \Comment*{KEY: $\xinit$, $\xfull$ may be different}
	\For{$t=0$ {\bfseries{\textup{to}}} $T-1$}{
		Sample $i_t\sim \mathrm{Unif}[n]$ \;
		$g_t \gets \nabla \phi_{i_t}(x_t)-\nabla \phi_{i_t} (\xfull) + \nabla \fun(\xfull)$ \;
		$x_{t+1} \gets \proj_{\xset}\Par{x_t-\eta g_t}$ \;
	}
	\codeReturn $\bx = \frac{1}{T}\sum_{t\in[T]}x_{t}$ 
\end{algorithm2e}
\end{minipage}\hfill
	\begin{minipage}[t]{0.525 \linewidth}%
		\centering
\begin{algorithm2e}[H]
	\caption{$\WarmStart$-$\SVRG$}
	\label{alg:SVRG-warmstart}
	\DontPrintSemicolon
	\codeInput $F = \frac{1}{n}\sum_{i\in[n]}f_i$, smoothness $L$, point $\xinit$\;
	\codeParameter Iteration number $T$\;
	$x^{(0)} \gets \xinit$ \; 
	\For{$k=0$ \codeStyle{to} $K-1$}{
		$\eta_{k+1}\gets \prn[\big]{8Ln^{2^{-k-1}}}^{-1}$ \;
		$x^{(k+1)}\gets \SVRGone\Par{ F, x^{(k)}, x^{(k)}, \eta_{k+1}, T}$\;
	}
	\codeReturn $x^{(K)}$
\end{algorithm2e}
	\end{minipage}
\end{figure}
}

In this section, we consider the following problem of finite-sum minimization:
\begin{gather}
	\minimize_{x\in\xset} F(x)\defeq\frac{1}{n}\sum_{i\in[n]} f_i(x),\label{def:problem-fs}%
\end{gather}
where each $f_i$ is $L$-smooth and convex.

We solve the problem by combining $\ouralg$ with a single epoch of SVRG~\citep{johnson2013accelerating}, shown in \Cref{alg:SVRG}. Our single point of departure from this classical algorithm is that the point we center our gradient estimator at ($\xfull$) is allowed to differ from the initial iterate ($\xinit$). Setting $\xfull$ to be the point $\xprev$ of $\ApproxProx$ allows us to efficiently meet our relaxed error criterion.  %

\begin{corollary}[$\ApproxProx$ for finite-sum minimization]\label{coro:approxprox-fs}
Given finite-sum problem~\eqref{def:problem-fs}, points $s, \xinit, \xfull \in \xset$, and $\lambda \in (0, L]$, $\SVRGone$ (\cref{alg:SVRG}) with  $\phi_i(x) \defeq f_i(x) + \frac{\lambda}{2}\norm{x-s}^2$, $\eta \le \frac{1}{32L}$, and $T=\lceil\tfrac{32}{\eta \lambda}\rceil  = O\Par{\frac{L}{\lambda}}$ implements $\ApproxProx_{F,\lambda}(s;\xinit,\xfull)$ (\cref{def:approx-prox}) using $O(n+L/\lambda)$ gradient queries.
\end{corollary}

\Cref{coro:approxprox-fs} follows from a slightly more general bound on $\SVRGone$ (\Cref{prop:SVRG} in \Cref{apdx:fs}). Below, we briefly sketch its proof. 

\notarxiv{\textbf{Proof sketch for \Cref{coro:approxprox-fs}.}}
\arxiv{\begin{psketch}} We begin by carefully bounding the variance of the gradient estimator $g_t$. In \Cref{lem:SVRG-grad} in the appendix we show that 
\begin{flalign}
	\norm{g_t - \grad F(x_t)}^2 \le 4L \cdot \left( V^F_{\xopt}\Par{\xfull} + V^{F^\lambda_s}_{\xopt}\Par{x_t}\right).\label{eq:proof-sketch-svrg}
\end{flalign}
 Next, standard analysis on variance reduced stochastic gradient method~\citep{xiao2014proximal} shows that
\begin{flalign*}
& \mathbb{E}\norm{x_{t+1} - x^\star}^2 \le \norm{x_t - x^\star}^2 \\
& \hspace{1em}- 2 \eta \mathbb{E} \left[  F^\lambda_s\Par{x_{t+1}}-F^\lambda_s\Par{\xopt} \right] + \eta^2 \norm{g_t - \grad F(x_t)}^2.
\end{flalign*}
Plugging in the variance bound~\eqref{eq:proof-sketch-svrg} at iteration $x_t$, rearranging terms and telescoping for $t\in[T]-1$, we obtain
\begin{flalign*}
 & \frac{1}{T} \E \sum_{t=1}^{T}\left( F_s^{\lambda}(x_t) - F_s^\lambda\Par{x^\star} \right)\\
& \hspace{1em}\le  \frac{2}{\eta T} \Veu_{x^\star}\Par{\xinit} + 4 \eta LV^F_{x^\star}\Par{\xfull}+\frac{4\eta L}{T}V^{F^\lambda_s}_{\xopt}\Par{\xinit}\\
& \hspace{1em}\stackrel{(i)}{\le} \frac{1}{8}V^F_{x^\star}\Par{\xfull}+\Par{\frac{2}{\eta T}+\frac{4\eta L(L+\lambda)}{T}}\Veu_{\xopt}\Par{\xinit}\\
& \hspace{1em}\stackrel{(ii)}{\le} \frac{1}{8}\lambda\Veu_{\xopt}\Par{\xinit}+ \frac{1}{8}V^F_{x^\star}\Par{\xfull},
\end{flalign*}
where we use  $(i)$ smoothness of $F^\lambda_s$, and $(ii)$ the choices of $\eta$ and $T$. Noting that $\bx = \frac{1}{T}\sum_{t\in[T]}x_{t}$ satisfies $F_s^\lambda(\bx) \le \frac{1}{T}\sum_{t=1}^{T}\left( F_s^{\lambda}(x_t)\right)$ by convexity concludes the proof sketch.
\arxiv{\end{psketch}}

\notarxiv{
	\begin{algorithm2e}[t]
		\caption{$\SVRGone$}
		\label{alg:SVRG}
		\DontPrintSemicolon
		\codeInput $\Phi = \frac{1}{n}\sum_{i\in[n]}\phi_i$ (with component gradient oracles), center point $\xfull$, initial point $\xinit$, step-size $\eta$, iteration number $T$

		Query gradient $\nabla \fun(\xfull) = \frac{1}{n}\sum_{i\in[n]}\nabla \phi_i( \xfull)$ \;
		$x_0 \gets \xinit$ \Comment*{KEY: $\xinit$, $\xfull$ may be different}
		\For{$t=0$ {\bfseries{\textup{to}}} $T-1$}{
			Sample $i_t\sim \mathrm{Unif}[n]$ \;
			$g_t \gets \nabla \phi_{i_t}(x_t)-\nabla \phi_{i_t} (\xfull) + \nabla \fun(\xfull)$ \;
			$x_{t+1} \gets \proj_{\xset}\Par{x_t-\eta g_t}$ \;
		}
		\codeReturn $\bx = \frac{1}{T}\sum_{t\in[T]}x_{t}$ 
	\end{algorithm2e}
}

\notarxiv{\textbf{Warm-start implementation.}}
\arxiv{\paragraph{Warm-start implementation.}}
We now explain how to reuse $\SVRGone$ for obtaining a valid warm-start for $\ouralg$ (\Cref{def:warmstart}).  Given any initial iterate with function error $\Delta$, we show that a careful choice of step size for $\SVRGone$  leads to a point with suboptimality $\sqrt{\frac{LR^2 \Delta}{n}}$ in $O(n)$ gradient computations (\Cref{lem:SVRG-warmstart}). Repeating this procedure $O(\log \log n)$ times produces a point with suboptimality $O(LR^2/n)$, which is a valid warm-start for $\lambda=L/n$. We remark that \citet{song2020variance} achieve the same $O(n\log \log n)$ complexity with a different procedure that entails changing the recursion for $\alpha_t$ in \Cref{line:catalyst:convex_start} of \Cref{alg:catalyst}. We believe that our approach is conceptually simpler and might be of independent interest.

\begin{restatable}[$\WarmStart$-$\SVRG$ for finite-sum minimization]{corollary}{warmstartfs}\label{coro:SVRG-warmstart}
Consider problem~\eqref{def:problem-fs} with minimizer $\xopt$, smoothness parameter $L$, and some initial point $\xinit$ with $R=\norm{\xinit-\xopt}$, for any $\lambda\ge L/n$, \Cref{alg:SVRG-warmstart} with $T = 32n$, $K = \log\log n$ implements $\WarmStart_{F,\lambda}(R^2)$ with $O(n\log\log n)$ gradient queries.%
\end{restatable}

By implementing $\ApproxProx$ and $\WarmStart$ using $\SVRGone$, $\ouralg$ provides the following state-of-the-art complexity bound for finite-sum problems.

\begin{restatable}[$\ouralg$ for finite-sum minimization]{theorem}{OptCatfs}\label{thm:OptCatfs}
Given a finite-sum problem~\eqref{def:problem-fs} on domain $\xset$ with diameter $R$, $\ouralg$ (\Cref{alg:catalyst}) with parameters $\lambda = \frac{L}{n}$ and $T = O(R\sqrt{Ln^{-1}\epsilon^{-1}})$, using $\SVRGone$ for $\ApproxProx$, and $\WarmStart$-$\SVRG$ for $\WarmStart$, outputs an $x$ such that 
$\E F(x)-\min_{x'\in\xset}F(x')\le \epsilon$. The total gradient query complexity is $O(n\log\log n+\sqrt{nLR^2 /\eps})$ in expectation. Further, if $F$ is $\gamma$-strongly-convex with $\gamma \le  O(L/n)$, restarted $\ouralg$ (\Cref{alg:catalyst-sc}) finds an $\epsilon$-approximate solution using $O(n \log \log n + \sqrt{nL/\gamma}\log(LR^2/(n\epsilon))$ gradient queries in expectation.
\end{restatable}

\notarxiv{	
	\begin{algorithm2e}[t]
		\caption{$\WarmStart$-$\SVRG$}
		\label{alg:SVRG-warmstart}
		\DontPrintSemicolon
		\codeInput $F = \frac{1}{n}\sum_{i\in[n]}f_i$, smoothness $L$, point $\xinit$\;
		\codeParameter Iteration number $T$\;
		$x^{(0)} \gets \xinit$ \; 
		\For{$k=0$ \codeStyle{to} $K-1$}{
			$\eta_{k+1}\gets \prn[\big]{8Ln^{2^{-k-1}}}^{-1}$ \;
			$x^{(k+1)}\gets \SVRGone\Par{ F, x^{(k)}, x^{(k)}, \eta_{k+1}, T}$\;
		}
		\codeReturn $x^{(K)}$
	\end{algorithm2e}
}

\subsection{Empirical Results}\label{ssec:experiments}

In this section, we provide an empirical performance comparisons between  RECAPP and SVRG~\cite{johnson2013accelerating} and Catalyst~\cite{lin2017catalyst}. Specifically, we compare to the C1* variant of Catalyst-SVRG, which~\citet{lin2017catalyst} report to have the best performance in practice. We implemented all algorithms in Python, using the Numba~\cite{lam2015numba} package for just-in-time compilation which significantly improved runtime. Our code is available at: \href{https://github.com/yaircarmon/recapp}{\texttt{github.com/yaircarmon/recapp}}.

\paragraph{Task and datasets.} We consider logistic regression on three datasets from libSVM~\cite{libsvm}: covertype ($n=581,012$, $d=54$), real-sim ($n=72,309$, $d=20,958$), and a9a ($n=32,561$, $d=123$). For each dataset we rescale the feature vectors to using unit Euclidean norm so that each $f_i$ is exactly $0.25$-smooth. We do not add $\ell_2$ regularization to the logistic regression objective.

\notarxiv{We defer readers to~\Cref{ssec:experiments-more} for detailed implementation of the algorithms and parameter tuning.}

\arxiv{
\paragraph{SVRG implementation.} We implement the SVRG iterates as in \Cref{alg:SVRG}, using $T=2n$ and $\eta=4$ (i.e., the inverse of the smoothness of each function). However, instead of outputting the average of all iterates, we return the average of the final $T/2=n$ iterates.

\paragraph{Catalyst implementation.} Our implementation follows closely Catalyst C1* as described in~\cite{lin2017catalyst}, where for the subproblem solver we use repeatedly called \Cref{alg:SVRG} with the parameters and averaging modification described above, checking the C1 termination criterion between each call.

\paragraph{RECAPP implementation.} Our RECAPP implementation follows \Cref{alg:catalyst,alg:unbiased-min-MLMC}, with \Cref{alg:SVRG} and \Cref{alg:SVRG-warmstart} implemented \ApproxProx and \WarmStart, respectively, and \Cref{alg:SVRG} configured and modified and described above. In \Cref{alg:unbiased-min-MLMC} we set the parameters $j_0=0$ and we test $p\in\{0,0.1,0.25,0.5\}$. The setting $p=0$ which corresponds to setting $\tilde{x}_{t+1} = x_{t+1}$ in~\Cref{alg:catalyst}) is a baseline meant to test whether MLMC is helpful at all. For $p>0$ we change the parameter $T$ in \Cref{alg:SVRG} such that the \emph{expected} amount of gradient computations is the same as for $p=0$. Slightly departing from the pseudocode of \Cref{alg:catalyst}, we take $x_{t+1}$ to be $x^{(J)}$ computed in \Cref{alg:unbiased-min-MLMC}, rather than $x^{(0)}$, since it is always a more accurate proximal point approximation. We note that our algorithm still has provable guarantees (with perhaps different constant factors) under this configuration.

\paragraph{Parameter tuning.} For RECAPP and Catalyst, we tune the proximal regularization parameter $\lambda$ (called $\kappa$ in~\cite{lin2017catalyst}). For each problem and each algorithm, we test $\lambda$ values of the form $\alpha L / n$, where $L=0.25$ is the objective smoothness, $n$ is the dataset size and $\alpha$ in the set $\{0.001, 0.003, 0.01, 0.03, 0.1, 0.3, 1.0, 3.0, 10.0\}$. We report results for the best $\lambda$ value for each problem/algorithm pair. 
}

\paragraph{Findings.}
We summarize our experimental findings in~\Cref{fig:exp}. The top row of \Cref{fig:exp} shows that RECAPP is competitive with Catalyst C1*: on covertype RECAPP is significantly faster, on a9a it is about the same, and on real-sim it is a bit slower. Note that Catalyst C1* incorporates carefully designed heuristics and parameters for choosing the SVRG initialization and stopping time, while the RECAPP implementation directly follows our theoretical development. The bottom row of \Cref{fig:exp} shows that $p\in\{0.1,0.25\}$ provides a modest but fairly consistent improvement over the no-MLMC baseline. This provides evidence that MLMC might be beneficial in practice. %

\begin{figure}
	\arxiv{\includegraphics[width = \textwidth]{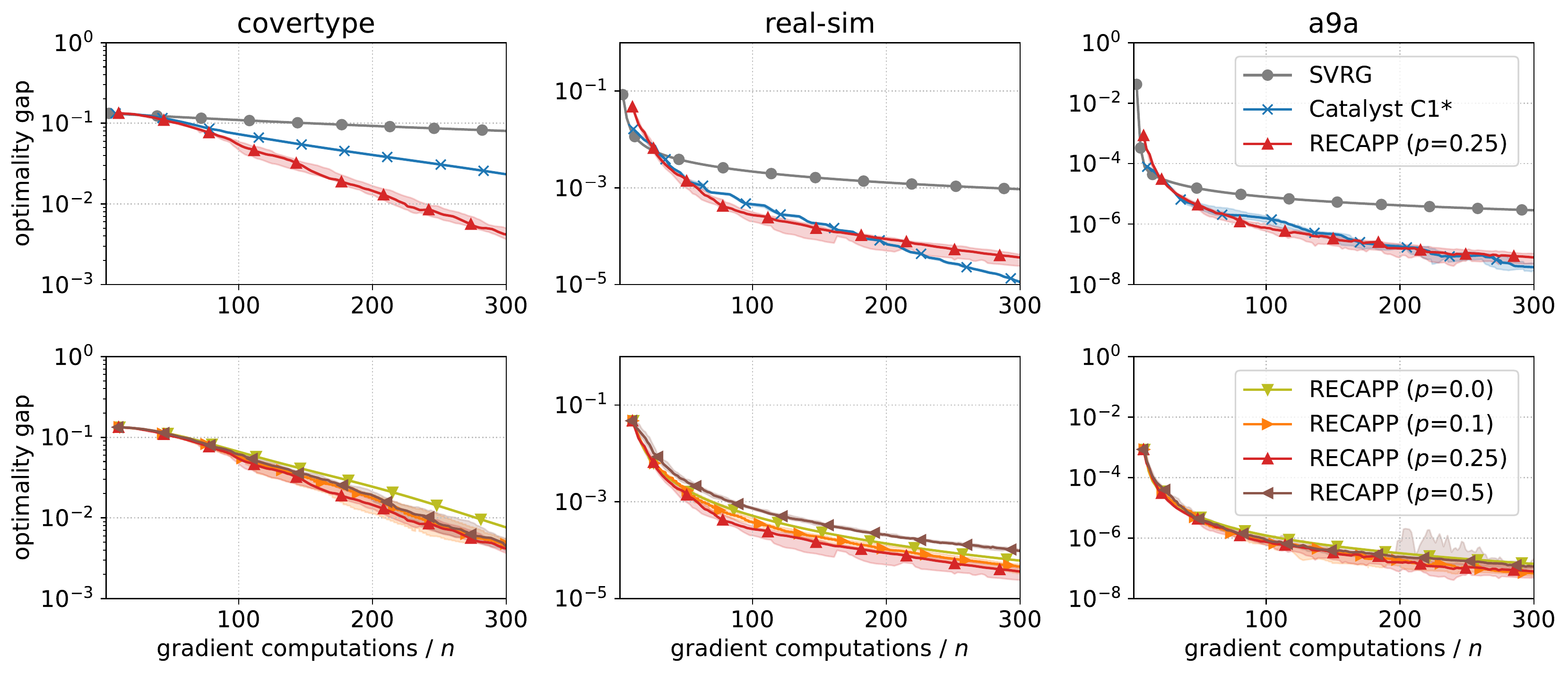}}
	\notarxiv{\includegraphics[width =\columnwidth]{figures/experiment.pdf}}
	\caption{Empirical evaluation of RECAPP on finite sum problems. 
		Columns represent different datasets, the top row compares RECAPP to SVRG and Catalyst, and the bottom row compares different MLMC $p$ parameters\notarxiv{ ($p=0$ corresponds to $\tilde{x}_{t+1} = x_{t+1}$ in~\Cref{alg:catalyst}, a baseline meant to test whether MLMC is helpful at all, see~\Cref{ssec:experiments-more} for more details)}. Solid lines show median over 20 seeds, and shaded regions show interquartile range.}\label{fig:exp}
\end{figure}
\section{Max-structured Minimization}\label{sec:minimax}

In this section, we consider the following problem of the max-structured function minimization:
\begin{gather}
	\minimize_{x\in\xset} F(x) \defeq \max_{y\in\yset}f(x,y),\label{def:problem-minimax}
\end{gather}
where $f:\xset\times\yset\to \R$ is $L$ smooth, convex in $x$ (for every $y$) and $\mu$-strongly-concave in $y$ (for every $x$).

We solve \eqref{def:problem-minimax} by combining $\ouralg$ with variants of mirror-prox method~\citep{nemirovski2004prox}, shown in~\Cref{alg:mirror-prox}. Given a convex-concave $L$-smooth objective $\phi$, $\MirrorProx$ (\Cref{alg:mirror-prox}) starts from initial point $\xinit, \yinit$, and finds in $O(T)$ gradient queries an approximate solution $x_T, y_T$ satisfying for any $x,y\in\xset\times\yset$,
\begin{align}
	\phi(x_T,y)- \phi(x,y_T)\le \frac{L}{T}\Par{\Veu_{\xinit}\Par{x}+\Veu_{\yinit}\Par{y}}.\label{eq:mirror-prox-guarantee-informal}
\end{align}
Our main observation is that applying such a mirror-prox method to the regularized objective $\phi(x,y) = f(x,y)+\mu \Veu_s(x)$, initialized at $(\xinit, \brorcl_F(\xprev))$ where we define the best response oracle $\brorcl_F(x)\defeq \argmax_{y\in\yset}f(x,y)$, outputs solution satisfying the relaxed error criterion of $\ApproxProx_{F,\mu}$ after $T = O(L/\mu)$ steps. We formalize this observation in~\Cref{lem:approxprox-minimax}.

\arxiv{
\begin{algorithm2e}[t]
	\caption{$\MirrorProx$}
	\label{alg:mirror-prox}
	\DontPrintSemicolon
	\codeInput  Gradient oracle for $\phi:\xset\times\yset\to \R$,  smoothness $L$, points $\xinit,\yinit$, iteration number $T$
	\Comment*{To implement $\ApproxProx_{F,\mu}$, we let $\yinit \approx \argmax_{y\in\yset}f(x,y)$}
	\codeParameter Step-size $\eta$\;
	Initialize $x_0\gets \xinit$, $y_0\gets \yinit$~\label{line:minimax-br-agd}\;
	\For{$t=0$ {\bfseries{\textup{to}}} $T-1$}{
		$u_{t}\gets \arg\min_{x\in\xset}\langle\eta \nabla_x\phi(x_t,y_t),x\rangle +\Veu_{x_{t}}(x)$\;
		$v_{t}\gets \arg\min_{y\in\yset}\langle-\eta\nabla_y\phi(x_t,y_t),y\rangle +\Veu_{y_{t}}(y)$\;
		$x_{t+1}\gets \arg\min_{x\in\xset}\langle\eta \nabla_x\phi(u_{t},v_{t}),x\rangle +\Veu_{x_{t}}(x)$\;
		$y_{t+1}\gets \arg\min_{y\in\yset}\langle-\eta\nabla_y\phi(u_{t},v_{t}),y\rangle +\Veu_{y_{t}}(y)$\;
	}
	\codeReturn  $x_T, y_T$
\end{algorithm2e}
}

\begin{restatable}[$\ApproxProx$ for max-structured minimization]{lemma}{lemapproxproxminimax}\label{lem:approxprox-minimax}
Given max-structured minimization problem~\eqref{def:problem-minimax} and an oracle $\brorcl_F(x)$ that outputs $\ybr_x\defeq \max_{y\in\yset}f(x,y)$ for any $x$, $\MirrorProx$ in~\Cref{alg:mirror-prox} initialized at $(\xinit,\brorcl_F(\xprev))$ implements the procedure $\ApproxProx_{F,\mu}(s;\xinit,\xprev)$ using a total of $O\Par{L/\mu}$ gradient queries and one call to $\brorcl_F(\cdot)$.
\end{restatable}

Before providing a proof sketch for the lemma, let us remark on the cost of implementing the best response oracle. Since for any fixed $x$ the function $f(x,\cdot)$ is $\mu$-strongly-concave and $L$-smooth, we can use AGD to find an $\delta$-accurate best response $y'$ to $x$ in $O\prn*{\sqrt{\frac{L}{\mu}}\log \frac{F(x)-f(x,y')}{\delta}}$ gradient queries. Therefore, even for extremely small values of $\delta$ we can expect the best-response computation cost to be negligible compared to the $O(L/\mu)$ complexity of the mirror-prox iterations required to implement $\ApproxProx$. 

\notarxiv{\textbf{Proof sketch for \Cref{lem:approxprox-minimax}.}}
\arxiv{\begin{psketch}}
We run $\MirrorProx$ for $T=O(L/\mu)$ steps on $\phi(x,y)=f(x,y) + \mu \Veu_s(x)$. By~\eqref{eq:mirror-prox-guarantee-informal} the output $(x_T,y_T)$ satisfies for $x=\xopt =\prox(s)$, $y=\ybr_{x_T}$ and arbitrary constant $c$,
\begin{equation*}
\phi(x_T,\ybr_{x_T})- \phi(\xopt,y_T)\le c\mu\Par{\Veu_{\xinit}\Par{\xopt}+\Veu_{\yinit}\Par{\ybr_{x_T}}}.
\end{equation*}
The optimality of $\xopt$ gives $\phi(\xopt,y_T) - \phi(\xopt,\ybr_{\xopt})\le 0$. Combining with the above implies
\begin{equation}\label{eq:minimax-approxprox-1}
F^\mu_s\Par{x_T}- F^\mu_s\Par{\xopt}\le c\mu\Par{\Veu_{\xinit}\Par{\xopt}+\Veu_{\yinit}\Par{\ybr_{x_T}}}.
\end{equation}
We now bound the two sides of~\eqref{eq:minimax-approxprox-1} separately. The strong concavity of $\phi$ in $y$ allows us to show that
\arxiv{\[
F^\mu_s\Par{x_T}- F^\mu_s\Par{\xopt}\ge \mu\Veu_{\ybr_{x_T}}(\yopt).
\]}
\notarxiv{
$F^\mu_s\Par{x_T}- F^\mu_s\Par{\xopt}\ge \mu\Veu_{\ybr_{x_T}}(\yopt).$
}
For the right-hand side, the definition of $\yinit$ as a best response to $\xprev$ yields
\arxiv{
\begin{flalign*}
	& \mu\Veu_{\yinit}\Par{\yopt} \le  V^F_{\xopt}\Par{\xprev}.
\end{flalign*}
}
\notarxiv{
$\mu\Veu_{\yinit}\Par{\yopt} \le  V^F_{\xopt}\Par{\xprev}.$
}
Plugging both inequalities into~\eqref{eq:minimax-approxprox-1}, and choosing sufficiently small $c$, we see the output satisfies the condition of a $\ApproxProx_{F,\mu}$ oracle, concluding the proof sketch.
\arxiv{\end{psketch}}

\notarxiv{\textbf{Warm-start implementation.}}
\arxiv{\paragraph{Warm-start implementation.}}
We now explain how to apply accelerated gradient descent (AGD) and a recursive use of $\MirrorProx$ for obtaining a valid warm-start for $\ouralg$ (\Cref{def:warmstart}). 

\notarxiv{
\begin{algorithm2e}[t]
	\caption{$\MirrorProx$}
	\label{alg:mirror-prox}
	\DontPrintSemicolon
	\codeInput  Gradient oracle for $\phi:\xset\times\yset\to \R$,  smoothness $L$, points $\xinit,\yinit$, iteration number $T$
	\Comment*{To implement $\ApproxProx_{F,\mu}$, we let $\yinit \approx \argmax_{y\in\yset}f(x,y)$}
	\codeParameter Step-size $\eta$\;
	Initialize $x_0\gets \xinit$, $y_0\gets \yinit$~\label{line:minimax-br-agd}\;
	\For{$t=0$ {\bfseries{\textup{to}}} $T-1$}{
		$u_{t}\gets \arg\min_{x\in\xset}\langle\eta \nabla_x\phi(x_t,y_t),x\rangle +\Veu_{x_{t}}(x)$\;
		$v_{t}\gets \arg\min_{y\in\yset}\langle-\eta\nabla_y\phi(x_t,y_t),y\rangle +\Veu_{y_{t}}(y)$\;
		$x_{t+1}\gets \arg\min_{x\in\xset}\langle\eta \nabla_x\phi(u_{t},v_{t}),x\rangle +\Veu_{x_{t}}(x)$\;
		$y_{t+1}\gets \arg\min_{y\in\yset}\langle-\eta\nabla_y\phi(u_{t},v_{t}),y\rangle +\Veu_{y_{t}}(y)$\;
	}
	\codeReturn  $x_T, y_T$
\end{algorithm2e}
}
\arxiv{
\begin{algorithm2e}[t]
	\caption{$\WarmStart$-$\MirrorProx$}
	\label{alg:warmstart-minimax}
	\DontPrintSemicolon
	\codeInput  Gradient oracle for $\phi:\xset\times\yset\to \R$, strong concavity $\mu$, smoothness $L$, point $(\xinit, \yinit)$\;
	\codeParameter Iteration number $T$, epoch number $K$\;
	Find $\yinit'$ so that $f(\xinit,\yinit')-f(\xinit,\ybr_{\xinit})\le \frac{1}{2}L R^2$ \label{line:warmstart-minimax-bestresponse}\Comment*{Implemented via AGD}
	Let $\phi\defeq f(x,y)+\mu \Veu_{\xinit}(x)$ and $L'=L+\mu$\;
	Initialize $x^{(0)}\gets \xinit$, $y^{(0)}\gets \yinit$\;
	\For{$k=0$ {\bfseries{\textup{to}}} $K-1$}{
		$x^{(k+1)},y^{(k+1)}\gets\MirrorProx(\phi,L',x^{(k)}, y^{(k)}, T)$\;
	}
	\codeReturn  $x^{(K)}$
\end{algorithm2e}
}

\begin{restatable}[$\WarmStart$ for max-structured minimization]{lemma}{lemwarmstartminimax}\label{lem:minimax-warm-start}
	Consider problem~\eqref{def:problem-minimax} where $R$, $R'$ are diameter bounds for $\xset$, $\yset$ respectively. Given initial point $\xinit, \yinit$, \Cref{alg:warmstart-minimax}, with parameters $T = O(L/\mu)$, $K=O(\log(L/\mu))$ and \Cref{line:warmstart-minimax-bestresponse} implemented using AGD, implements~$\WarmStart_{F,\mu}(R^2)$ with \begin{equation*}
		O\Par{L/\mu\log(L/\mu)+\sqrt{L/\mu}\log\Par{R'/R}}
	\end{equation*}
	gradient queries.
\end{restatable}

By implementing $\ApproxProx$ and $\WarmStart$ using $\MirrorProx$ and~$\WarmStart$-$\textsc{Minimax}$, $\ouralg$ provides the following state-of-the-art complexity bounds for minimizing the max-structured problems.

\begin{restatable}[$\ouralg$ for minimizing the max-structured problem]{theorem}{thmminimax}\label{thm:minimax}
Given $F=\max_{y\in\yset}f(x,y)$ with diameter bounds $R$, $R'$ on $\xset$, $\yset$, respectively, $\ouralg$ (\Cref{alg:catalyst}) with parameters $\lambda = \mu$ and $T = O(R\sqrt{\mu/\eps})$ and  $\MirrorProx$ (\Cref{alg:mirror-prox}) to implement $\ApproxProx$, and $\WarmStart$-$\MirrorProx$ (\Cref{alg:warmstart-minimax}) to implement $\WarmStart$, outputs a solution $x$ such that 
$\E~F(x)-\min_{x'\in\xset}F(x')\le \epsilon$. The algorithm uses \arxiv{\begin{equation*}
	O\prn[\big]{LR/\sqrt{\mu\epsilon}+L/\mu\log(L/\mu)+\sqrt{L/\mu}\log\Par{R'/R}}
\end{equation*} }
\notarxiv{
$O\prn[\big]{LR/\sqrt{\mu\epsilon}+L/\mu\log(L/\mu)+\sqrt{L/\mu}\log\Par{R'/R}}$
}
gradient queries in expectation and $O\prn[\big]{R\sqrt{\mu/\eps}}$ calls to a best-response oracle $\brorcl_{f}(\cdot)$. 
Further, if $F$ is $\gamma$-strongly-convex, restarted $\ouralg$ (\Cref{alg:catalyst-sc}) with parameters $\lambda = \mu$, $T = O\Par{\sqrt{\mu/\gamma}}$, $K = O\Par{\log \Par{LR^2/\epsilon}}$ finds an $\eps$-approximate solution using 
	$
	O(LR/\sqrt{\mu\gamma}\log\Par{LR^2/\epsilon}+L/\mu\log(L/\mu)+\sqrt{L/\mu}\log\Par{R'/R})
	$
 gradient queries in expectation and $O\prn[\big]{\sqrt{\frac{\mu}{\gamma}}\log\frac{LR^2}{\epsilon}}$ calls to $\brorcl_{f}(\cdot)$.
\end{restatable}

We remark that for strongly-convex $F$, the restarted ~\Cref{alg:catalyst-delta-sc} not only yields a good approximate solution for $F$, but also can be transferred to a good approximate primal-dual solution for $f(x,y)$ by taking the best-response to the high-accuracy solution $x$.

\notarxiv{
\begin{algorithm2e}[t]
	\caption{$\WarmStart$-$\MirrorProx$}
	\label{alg:warmstart-minimax}
	\DontPrintSemicolon
	\codeInput  Gradient oracle for $\phi:\xset\times\yset\to \R$, strong concavity $\mu$, smoothness $L$, point $(\xinit, \yinit)$\;
	\codeParameter Iteration number $T$, epoch number $K$\;
	Find $\yinit'$ so that $f(\xinit,\yinit')-f(\xinit,\ybr_{\xinit})\le \frac{1}{2}L R^2$ \label{line:warmstart-minimax-bestresponse}\;
	\Comment*{Implemented via AGD}
	Let $\phi\defeq f(x,y)+\mu \Veu_{\xinit}(x)$ and $L'=L+\mu$\;
	Initialize $x^{(0)}\gets \xinit$, $y^{(0)}\gets \yinit$\;
	\For{$k=0$ {\bfseries{\textup{to}}} $K-1$}{
		$x^{(k+1)},y^{(k+1)}\gets\MirrorProx(\phi,L',x^{(k)}, y^{(k)}, T)$\;
	}
	\codeReturn  $x^{(K)}$
\end{algorithm2e}
}

\notarxiv{\textbf{Generalization to the framework.}}
\arxiv{\paragraph{Generalization to the framework.}} To obtain complexity bounds strictly in terms of gradient queries, we extend the framework of~\Cref{sec:framework} to handle small additive errors $\delta\approx\Omega(1/t^4)$ at iteration $t$ when implementing the $\ApproxProx$ procedure as defined in~\eqref{approx-prox-cond}. For $x^\star = \arg\min_{x\in\xset} F^\lambda_s(x)$ we allow $\ApproxProx$ to return $x$ satisfying
 \begin{equation}\label{approx-prox-cond-delta}
\begin{aligned}
\E \Flam_s(x) - \Flam_s(x^\star)\le \frac{1}{8} & \Par{\lambda\Veu_{x^\star}\Par{\xinit}+V^F_{x^\star}(\xprev)}+\delta.
\end{aligned}
\end{equation}
This way, in~\Cref{lem:approxprox-minimax} one can implement the best-response oracle $\brorcl_{f}(\cdot)$ (in~\Cref{line:minimax-br-agd}) to a sufficient high accuracy using $\widetilde{O}(\sqrt{L/\mu})$ gradient queries, using the standard accelerated gradient method~\citep{nesterov1983method}. This turns the method in~\Cref{thm:minimax} into a complete algorithm for solving \eqref{def:problem-minimax}, and only incurs an additional cost of  $\widetilde{O}\left(R\sqrt{L/\mu}\cdot\sqrt{\mu/\eps}\right) = \widetilde{O}\left(R\sqrt{L/\eps}\right)$ gradient queries.

We state the main result here and refer readers to~\Cref{apdx:minimax-delta} for the generalization of $\ouralg$ (\Cref{alg:catalyst-delta},~\Cref{alg:catalyst-delta-sc} and~\Cref{prop:improved-catalyst-delta}) and more detailed discussion.

\begin{restatable}[$\ouralg$ for minimizing the max-structured problem, without $\brorcl_{f}$]{theorem}{thmminimax}\label{thm:minimax-delta}
Under the same setting of~\Cref{thm:minimax},~\Cref{alg:catalyst-delta} with accelerated gradient descent to implement $\brorcl_{f}(\cdot)$, outputs a primal $\eps$-approximate solution $x$ and has expected  gradient query complexity of  
\arxiv{\[O\Par{\frac{LR}{\sqrt{\mu\eps}}+\underbrace{\frac{L}{\mu}\log\frac{L}{\mu}+\sqrt{\frac{L}{\mu}}\log\frac{R'}{R}+R\sqrt{\frac{L}{\epsilon}}\log\frac{L(R+R')^2}{\epsilon}}_{\text{lower-order terms}}}.\] }
\notarxiv{$O\biggl(\frac{LR}{\sqrt{\mu\eps}}+\frac{L}{\mu}\log\frac{LR'}{\mu R}+R\sqrt{\frac{L}{\epsilon}}\log\frac{L(R+R')^2}{\epsilon}\biggr).$}
Further, if $F$ is $\gamma$-strongly-convex, restarted $\ouralg$ (\Cref{alg:catalyst-delta-sc}) finds an $\eps$-approximate solution and has expected gradient query complexity of 
 \arxiv{\[O\Par{\frac{LR}{\sqrt{\mu\gamma}}\log\Par{\frac{LR^2}{\epsilon}}+\underbrace{\frac{L}{\mu}\log\Par{\frac{L}{\mu}}+\sqrt{\frac{L}{\mu}}\log\Par{\frac{R'}{R}}+\sqrt{\frac{L}{\gamma}}\log\Par{\frac{\mu L(R+R')^2}{\gamma\epsilon}}\log\Par{\frac{LR^2}{\epsilon}}}_{\text{lower-order terms}}}.\]}
 \notarxiv{
 $O\biggl(\frac{LR}{\sqrt{\mu\gamma}}\log\Par{\frac{LR^2}{\epsilon}}+\frac{L}{\mu}\log\Par{\frac{LR'}{\mu R}}+\sqrt{\frac{L}{\gamma}}\log\Par{\frac{\mu L(R+R')^2}{\gamma\epsilon}}\log\Par{\frac{LR^2}{\epsilon}}\biggr)$.
 }
\end{restatable}

 \arxiv{%
\section{Discussion}\label{sec:discussion}

This paper proposes an improvement of the APPA/Catalyst acceleration framework, providing an efficiently attainable Relaxed Error Criterion for the Accelerated Prox Point method ($\ouralg$) that eliminates logarithmic complexity terms from previous result while maintaining the elegant black-box structure of APPA/Catalyst.

The main conceptual drawback of our proposed framework (beyond its reliance on randomization) is that efficiently attaining our relaxed error criterion requires a certain degree of problem-specific analysis as well as careful subproblem solver initialization. In contrast, APPA/Catalyst rely on more standard and readily available linear convergence guarantees (which of course also suffice for $\ouralg$).

Nevertheless, we believe there are many more situations where efficiently meeting the relaxed criterion is possible. These include variance reduction for min-max problems, smooth min-max problems which are (strongly-)concave in $y$ but not convex in $x$, and problems amenable to coordinate methods. All of these are settings where APPA/Catalyst is effective~\citep{yang2020catalyst,frostig2015regularizing,lin2017catalyst} and our approach can likely be provably better. 

Moreover, even when proving improved rates is difficult, $\ApproxProx$ can still serve as an improved stopping criterion. This motivates further research into practical variants of $\ApproxProx$ that depend only on observable quantities (rather than, e.g.\ the distance to the true proximal point).

}

\section*{Acknowledgements}
The authors thank anonymous reviewers for helpful suggestions. YC was supported in part by the Israeli Science Foundation (ISF) grant no.\ 2486/21 and the Len Blavatnik and the Blavatnik Family foundation. YJ was supported in part by a Stanford Graduate Fellowship and the Dantzig-Lieberman Fellowship. 
AS was supported in part by a Microsoft Research Faculty Fellowship, NSF CAREER Award CCF-1844855, NSF Grant CCF-1955039, a PayPal research award, and a Sloan Research Fellowship.

\bibliographystyle{icml2022}

\begin{thebibliography}{50}
\providecommand{\natexlab}[1]{#1}
\providecommand{\url}[1]{\texttt{#1}}
\expandafter\ifx\csname urlstyle\endcsname\relax
  \providecommand{\doi}[1]{doi: #1}\else
  \providecommand{\doi}{doi: \begingroup \urlstyle{rm}\Url}\fi

\bibitem[lib()]{libsvm}
The {LIBSVM} data webpage.
\newblock URL \url{https://www.csie.ntu.edu.tw/~cjlin/libsvmtools/datasets/}.

\bibitem[Allen-Zhu(2016)]{allen2016katyusha}
Allen-Zhu, Z.
\newblock Katyusha: The first direct acceleration of stochastic gradient
  methods.
\newblock \emph{arXiv preprint arXiv:1603.05953}, 2016.

\bibitem[Asi et~al.(2021)Asi, Carmon, Jambulapati, Jin, and
  Sidford]{asi2021stochastic}
Asi, H., Carmon, Y., Jambulapati, A., Jin, Y., and Sidford, A.
\newblock Stochastic bias-reduced gradient methods.
\newblock \emph{arXiv preprint arXiv:2106.09481}, 2021.

\bibitem[Azizian et~al.(2020)Azizian, Mitliagkas, Lacoste-Julien, and
  Gidel]{azizian2020tight}
Azizian, W., Mitliagkas, I., Lacoste-Julien, S., and Gidel, G.
\newblock A tight and unified analysis of gradient-based methods for a whole
  spectrum of games.
\newblock In \emph{International Conference on Artificial Intelligence and
  Statistics (AISTATS)}, 2020.

\bibitem[Ben-Tal et~al.(2009)Ben-Tal, El~Ghaoui, and
  Nemirovski]{bental2009robust}
Ben-Tal, A., El~Ghaoui, L., and Nemirovski, A.
\newblock \emph{Robust optimization}.
\newblock Princeton University Press, 2009.

\bibitem[Bertsekas(1999)]{bertsekas1999nonlinear}
Bertsekas, D.
\newblock \emph{Nonlinear Programming}.
\newblock Athena Scientific, 1999.

\bibitem[Blanchet \& Glynn(2015)Blanchet and Glynn]{blanchet2015unbiased}
Blanchet, J.~H. and Glynn, P.~W.
\newblock Unbiased {M}onte {C}arlo for optimization and functions of
  expectations via multi-level randomization.
\newblock In \emph{2015 Winter Simulation Conference (WSC)}, pp.\  3656--3667,
  2015.

\bibitem[Bubeck et~al.(2019)Bubeck, Jiang, Lee, Li, and
  Sidford]{bubeck2019complexity}
Bubeck, S., Jiang, Q., Lee, Y.~T., Li, Y., and Sidford, A.
\newblock Complexity of highly parallel non-smooth convex optimization.
\newblock In \emph{Advances in Neural Information Processing Systems}, 2019.

\bibitem[Bullins(2020)]{bullins2020highly}
Bullins, B.
\newblock Highly smooth minimization of non-smooth problems.
\newblock In \emph{Conference on Learning Theory}, pp.\  988--1030, 2020.

\bibitem[Carmon et~al.(2020)Carmon, Jambulapati, Jiang, Jin, Lee, Sidford, and
  Tian]{carmon2020acceleration}
Carmon, Y., Jambulapati, A., Jiang, Q., Jin, Y., Lee, Y.~T., Sidford, A., and
  Tian, K.
\newblock Acceleration with a ball optimization oracle.
\newblock In \emph{Advances in Neural Information Processing Systems}, 2020.

\bibitem[Carmon et~al.(2022)Carmon, Hausler, Jambulapati, Jin, and
  Sidford]{carmon2022optimal}
Carmon, Y., Hausler, D., Jambulapati, A., Jin, Y., and Sidford, A.
\newblock Optimal and adaptive monteiro-svaiter acceleration.
\newblock \emph{arXiv:2205.15371}, 2022.

\bibitem[Frostig et~al.(2015)Frostig, Ge, Kakade, and
  Sidford]{frostig2015regularizing}
Frostig, R., Ge, R., Kakade, S., and Sidford, A.
\newblock Un-regularizing: approximate proximal point and faster stochastic
  algorithms for empirical risk minimization.
\newblock In \emph{International Conference on Machine Learning}, 2015.

\bibitem[Ganin et~al.(2016)Ganin, Ustinova, Ajakan, Germain, Larochelle,
  Laviolette, Marchand, and Lempitsky]{ganin2016domain}
Ganin, Y., Ustinova, E., Ajakan, H., Germain, P., Larochelle, H., Laviolette,
  F., Marchand, M., and Lempitsky, V.
\newblock Domain-adversarial training of neural networks.
\newblock \emph{The journal of machine learning research}, 17\penalty0
  (1):\penalty0 2096--2030, 2016.

\bibitem[Garber et~al.(2016)Garber, Hazan, Jin, Musco, Netrapalli, Sidford,
  et~al.]{garber2016faster}
Garber, D., Hazan, E., Jin, C., Musco, C., Netrapalli, P., Sidford, A., et~al.
\newblock Faster eigenvector computation via shift-and-invert preconditioning.
\newblock In \emph{International Conference on Machine Learning (ICML)}, 2016.

\bibitem[Gasnikov et~al.(2019)Gasnikov, Dvurechensky, Gorbunov, Vorontsova,
  Selikhanovych, Uribe, Jiang, Wang, Zhang, Bubeck, Jiang, Lee, Li, and
  Sidford]{gasnikov19near}
Gasnikov, A., Dvurechensky, P., Gorbunov, E., Vorontsova, E., Selikhanovych,
  D., Uribe, C.~A., Jiang, B., Wang, H., Zhang, S., Bubeck, S., Jiang, Q., Lee,
  Y.~T., Li, Y., and Sidford, A.
\newblock Near optimal methods for minimizing convex functions with {L}ipschitz
  $p$-th derivatives.
\newblock In \emph{Conference on Learning Theory (COLT)}, 2019.

\bibitem[Giles(2015)]{giles2015multilevel}
Giles, M.~B.
\newblock Multilevel {M}onte {C}arlo methods.
\newblock \emph{Acta Numerica}, 24:\penalty0 259–328, 2015.

\bibitem[Gower et~al.(2020)Gower, Schmidt, Bach, and
  Richt{\'a}rik]{gower2020variance}
Gower, R.~M., Schmidt, M., Bach, F., and Richt{\'a}rik, P.
\newblock Variance-reduced methods for machine learning.
\newblock \emph{Proceedings of the IEEE}, 108\penalty0 (11):\penalty0
  1968--1983, 2020.

\bibitem[G{\"u}ler(1992)]{guler1992new}
G{\"u}ler, O.
\newblock New proximal point algorithms for convex minimization.
\newblock \emph{SIAM Journal on Optimization}, 2\penalty0 (4):\penalty0
  649--664, 1992.

\bibitem[Hu et~al.(2021)Hu, Chen, and He]{hu2021bias}
Hu, Y., Chen, X., and He, N.
\newblock On the bias-variance-cost tradeoff of stochastic optimization.
\newblock \emph{Advances in Neural Information Processing Systems (NeurIPS)},
  2021.

\bibitem[Ivanova et~al.(2021)Ivanova, Pasechnyuk, Grishchenko, Shulgin,
  Gasnikov, and Matyukhin]{ivanova2021adaptive}
Ivanova, A., Pasechnyuk, D., Grishchenko, D., Shulgin, E., Gasnikov, A., and
  Matyukhin, V.
\newblock Adaptive catalyst for smooth convex optimization.
\newblock In \emph{International Conference on Optimization and Applications},
  pp.\  20--37. Springer, 2021.

\bibitem[Johnson \& Zhang(2013)Johnson and Zhang]{johnson2013accelerating}
Johnson, R. and Zhang, T.
\newblock Accelerating stochastic gradient descent using predictive variance
  reduction.
\newblock \emph{Advances in neural information processing systems (NeurIPS)},
  2013.

\bibitem[Korpelevich(1976)]{korpelevich1976extragradient}
Korpelevich, G.~M.
\newblock The extragradient method for finding saddle points and other
  problems.
\newblock \emph{Ekonomika i Matematicheskie Metody}, 12:\penalty0 747--756,
  1976.

\bibitem[Kovalev \& Gasnikov(2022)Kovalev and Gasnikov]{kovalev2022first}
Kovalev, D. and Gasnikov, A.
\newblock The first optimal algorithm for smooth and
  strongly-convex-strongly-concave minimax optimization.
\newblock \emph{arXiv:2205.05653}, 2022.

\bibitem[Lam et~al.(2015)Lam, Pitrou, and Seibert]{lam2015numba}
Lam, S.~K., Pitrou, A., and Seibert, S.
\newblock Numba: A llvm-based python {JIT} compiler.
\newblock In \emph{Proceedings of the Second Workshop on the LLVM Compiler
  Infrastructure in HPC}, pp.\  1--6, 2015.

\bibitem[Lan et~al.(2019)Lan, Li, and Zhou]{lan2019unified}
Lan, G., Li, Z., and Zhou, Y.
\newblock A unified variance-reduced accelerated gradient method for convex
  optimization.
\newblock \emph{Advances in Neural Information Processing Systems (NeurIPS)},
  32, 2019.

\bibitem[Levy et~al.(2020)Levy, Carmon, Duchi, and Sidford]{levy2020large}
Levy, D., Carmon, Y., Duchi, J.~C., and Sidford, A.
\newblock Large-scale methods for distributionally robust optimization.
\newblock \emph{Advances in Neural Information Processing Systems}, 2020.

\bibitem[Lin et~al.(2015)Lin, Mairal, and Harchaoui]{lin2015universal}
Lin, H., Mairal, J., and Harchaoui, Z.
\newblock A universal catalyst for first-order optimization.
\newblock In \emph{Advances in Neural Information Processing Systems
  (NeurIPS)}, 2015.

\bibitem[Lin et~al.(2017)Lin, Mairal, and Harchaoui]{lin2017catalyst}
Lin, H., Mairal, J., and Harchaoui, Z.
\newblock Catalyst acceleration for first-order convex optimization: from
  theory to practice.
\newblock \emph{The Journal of Machine Learning Research}, 18\penalty0
  (1):\penalty0 7854--7907, 2017.

\bibitem[Lin et~al.(2020)Lin, Jin, and Jordan]{lin2020near}
Lin, T., Jin, C., and Jordan, M.~I.
\newblock Near-optimal algorithms for minimax optimization.
\newblock In \emph{Conference on Learning Theory (COLT)}, 2020.

\bibitem[Madry et~al.(2018)Madry, Makelov, Schmidt, Tsipras, and
  Vladu]{madry2018towards}
Madry, A., Makelov, A., Schmidt, L., Tsipras, D., and Vladu, A.
\newblock Towards deep learning models resistant to adversarial attacks.
\newblock In \emph{International Conference on Learning Representations}, 2018.

\bibitem[Monteiro \& Svaiter(2013)Monteiro and
  Svaiter]{monteiro2013accelerated}
Monteiro, R.~D. and Svaiter, B.~F.
\newblock An accelerated hybrid proximal extragradient method for convex
  optimization and its implications to second-order methods.
\newblock \emph{SIAM Journal on Optimization}, 23\penalty0 (2):\penalty0
  1092--1125, 2013.

\bibitem[Morgenstern \& Von~Neumann(1953)Morgenstern and
  Von~Neumann]{morgenstern1953theory}
Morgenstern, O. and Von~Neumann, J.
\newblock \emph{Theory of games and economic behavior}.
\newblock Princeton university press, 1953.

\bibitem[Nemirovski(2004)]{nemirovski2004prox}
Nemirovski, A.
\newblock Prox-method with rate of convergence ${O}(1/t)$ for variational
  inequalities with {L}ipschitz continuous monotone operators and smooth
  convex-concave saddle point problems.
\newblock \emph{SIAM Journal on Optimization}, 15\penalty0 (1):\penalty0
  229--251, 2004.

\bibitem[Nesterov(1983)]{nesterov1983method}
Nesterov, Y.
\newblock A method of solving a convex programming problem with convergence
  rate ${O}(1/k^2)$.
\newblock \emph{Soviet Mathematics Doklady}, 27\penalty0 (2):\penalty0
  372--376, 1983.

\bibitem[Nesterov(2005)]{nesterov2005smooth}
Nesterov, Y.
\newblock Smooth minimization of non-smooth functions.
\newblock \emph{Mathematical programming}, 103\penalty0 (1):\penalty0 127--152,
  2005.

\bibitem[Nesterov(2018)]{nesterov2018lectures}
Nesterov, Y.
\newblock \emph{Lectures on convex optimization}, volume 137.
\newblock Springer, 2018.

\bibitem[Ouyang \& Xu(2021)Ouyang and Xu]{ouyang2021lower}
Ouyang, Y. and Xu, Y.
\newblock Lower complexity bounds of first-order methods for convex-concave
  bilinear saddle-point problems.
\newblock \emph{Mathematical Programming}, 185\penalty0 (1):\penalty0 1--35,
  2021.

\bibitem[Paquette et~al.(2017)Paquette, Lin, Drusvyatskiy, Mairal, and
  Harchaoui]{paquette2017catalyst}
Paquette, C., Lin, H., Drusvyatskiy, D., Mairal, J., and Harchaoui, Z.
\newblock Catalyst acceleration for gradient-based non-convex optimization.
\newblock \emph{arXiv preprint arXiv:1703.10993}, 2017.

\bibitem[Parikh \& Boyd(2014)Parikh and Boyd]{parikh2014proximal}
Parikh, N. and Boyd, S.
\newblock Proximal algorithms.
\newblock \emph{Foundations and Trends in optimization}, 2014.

\bibitem[Rockafellar(1976)]{rockafellar1976monotone}
Rockafellar, R.~T.
\newblock Monotone operators and the proximal point algorithm.
\newblock \emph{SIAM journal on control and optimization}, 14\penalty0
  (5):\penalty0 877--898, 1976.

\bibitem[Salzo \& Villa(2012)Salzo and Villa]{salzo2012inexact}
Salzo, S. and Villa, S.
\newblock Inexact and accelerated proximal point algorithms.
\newblock \emph{Journal of Convex analysis}, 19\penalty0 (4):\penalty0
  1167--1192, 2012.

\bibitem[Shalev-Shwartz \& Zhang(2014)Shalev-Shwartz and
  Zhang]{shalev2014accelerated}
Shalev-Shwartz, S. and Zhang, T.
\newblock Accelerated proximal stochastic dual coordinate ascent for
  regularized loss minimization.
\newblock In \emph{International conference on machine learning (ICML)}, 2014.

\bibitem[Silver et~al.(2017)Silver, Schrittwieser, Simonyan, Antonoglou, Huang,
  Guez, Hubert, Baker, Lai, Bolton, et~al.]{silver2017mastering}
Silver, D., Schrittwieser, J., Simonyan, K., Antonoglou, I., Huang, A., Guez,
  A., Hubert, T., Baker, L., Lai, M., Bolton, A., et~al.
\newblock Mastering the game of go without human knowledge.
\newblock \emph{Nature}, 550\penalty0 (7676):\penalty0 354--359, 2017.

\bibitem[Song et~al.(2020)Song, Jiang, and Ma]{song2020variance}
Song, C., Jiang, Y., and Ma, Y.
\newblock Variance reduction via accelerated dual averaging for finite-sum
  optimization.
\newblock In \emph{Advances in Neural Information Processing Systems
  (NeurIPS)}, 2020.

\bibitem[Song et~al.(2021)Song, Jiang, and Ma]{song2021unified}
Song, C., Jiang, Y., and Ma, Y.
\newblock Unified acceleration of high-order algorithms under {H}\"older
  continuity and uniform convexity.
\newblock \emph{SIAM journal of optimization}, 2021.

\bibitem[Thekumparampil et~al.(2019)Thekumparampil, Jain, Netrapalli, and
  Oh]{thekumparampil2019efficient}
Thekumparampil, K.~K., Jain, P., Netrapalli, P., and Oh, S.
\newblock Efficient algorithms for smooth minimax optimization.
\newblock \emph{Advances in Neural Information Processing Systems (NeurIPS)},
  2019.

\bibitem[Woodworth \& Srebro(2016)Woodworth and Srebro]{woodworth2016tight}
Woodworth, B. and Srebro, N.
\newblock Tight complexity bounds for optimizing composite objectives.
\newblock In \emph{Advances in Neural Information Processing Systems}, pp.\
  3646--3654, 2016.

\bibitem[Xiao \& Zhang(2014)Xiao and Zhang]{xiao2014proximal}
Xiao, L. and Zhang, T.
\newblock A proximal stochastic gradient method with progressive variance
  reduction.
\newblock \emph{SIAM Journal on Optimization}, 24\penalty0 (4):\penalty0
  2057--2075, 2014.

\bibitem[Yang et~al.(2020)Yang, Zhang, Kiyavash, and He]{yang2020catalyst}
Yang, J., Zhang, S., Kiyavash, N., and He, N.
\newblock A catalyst framework for minimax optimization.
\newblock \emph{Advances in Neural Information Processing Systems}, 2020.

\bibitem[Zhao(2020)]{zhao2020primal}
Zhao, R.
\newblock A primal dual smoothing framework for max-structured nonconvex
  optimization.
\newblock \emph{arXiv:2003.04375}, 2020.

\end{thebibliography}

\newpage
\appendix
\onecolumn

\notarxiv{%
\section{Additional Related Work}\label{app:related}

Beyond the closely related work already described, our paper touches on several lines of literature. 

\paragraph{Finite-sum problems.}
The ubiquity of finite-sum optimization problems in machine learning has led to a very large body of work on developing efficient algorithms for solving them.  We refer the reader to~\citet{gower2020variance} for a broad survey and  focus on \emph{accelerated} finite-sum methods, i.e., with a leading order complexity term scaling as $\sqrt{n/\epsilon}$ (or as $\sqrt{n\kappa}$ for strongly-convex problems with condition number $\kappa$). 
Accelerated Proximal Stochastic Dual Coordinate Ascent~\citep{shalev2014accelerated} gave the first such accelerated rate for an important subclass of finite-sum problems. This method was subsequently interpreted as a special case of APPA/Catalyst~\cite{lin2015universal,frostig2015regularizing}, which  can also accelerate several other finite-sum optimization problems. Since then, research has focused on designing more practical and theoretically efficient accelerated algoirthms by opening the APPA/Catalyst black box. The algorithms Katyusha~\cite{allen2016katyusha}, Varag~\cite{lan2019unified} and VRADA~\citep{song2020variance} offer improved complexity bound at the price of the generality and simplicity of APPA/Catalyst. Our approach matches the best existing guarantee (due to VRADA) without paying this price.

\paragraph{Max-structured problems.}

Objectives of the form $F(x) = \max_{y\in\yset} f(x,y)$ are very common in machine learning and beyond. Such objectives arise from constraints (via Lagrange multipliers)~\citep{bertsekas1999nonlinear}, robustness requirements~\citep{bental2009robust,ganin2016domain,madry2018towards}, and game-theoretic considerations~\citep{morgenstern1953theory,silver2017mastering}. When $f$ is convex in $x$ and concave in $y$, the mirror-prox algorithm minimizes $F$ to accuracy epsilon in $O(LR R'/\epsilon)$ gradient evaluations (with respect to both $x$ and $y$), where $R'$ is the diameter of $\yset$. This rate can be improved when $f$ is $\mu$-strongly-concave in $y$. For the special bilinear case $f(x,y) = \phi(x) + \inner{y}{Ax} - \psi(y)$, where $\psi$ a ``simple'' $\mu$-strongly-convex function, an improved complexity bound of $O(LR/\sqrt{\mu \epsilon})$ has long been known~\citep{nesterov2005smooth}.

More recent work studies the case of general convex-strongly-concave $f$. \citet{thekumparampil2019efficient} and \citet{zhao2020primal} establish complexity bounds of $O(\frac{L^{3/2}}{\mu\sqrt{\epsilon}} \log^2 \frac{L^2 R R'}{\mu\epsilon})$, which~\citet{lin2020near} improve to $O(\frac{L}{\sqrt{\mu\epsilon}} \log^3 \frac{L^2 R R'}{\mu\epsilon})$ using an algorithm loosely based on APPA/Catalyst. \citet{yang2020catalyst} present a more direct application of APPA/Catalyst to min-max problems, further improving the complexity to $O(\frac{L}{\sqrt{\mu\epsilon}} \log \frac{L^2 R^2}{\mu\epsilon})$, with logarithmic dependence on $R'$ only in a lower order term. Similarly to standard APPA/Catalyst, the min-max variant requires highly accurate proximal point computation, e.g., to function-value error of $O(\frac{\mu^3 \epsilon^2}{L^4 R^2})$. In contrast, $\ouralg$ requires constant (relative) suboptimality and removes the final logarithmic factor from the leading-order complexity term.  \citet{yang2020catalyst} also provide extensions to finite-sum min-max problems and problems where $f$ is non-convex in $x$, which would likely benefit from out method as well (see \Cref{sec:discussion}).

\emph{Independent work.} In recent independent work, 
 \citet{kovalev2022first}  develop a method that minimizes  $\max_{y}f(x,y)$ assuming $\mu$-strong-concavity in $y$ and $\gamma$-strong-convexity in $x$. They attain an essentially optimal complexity proportional to $\frac{1}{\sqrt{\mu\gamma}}$ times a logarithmic factor depending on problem parameters. Their method is tailored to saddle point problems, working in  an expanded space by using point-wise conjugate function and applying recent advances in monotone operator theory. We note that RECAPP with restarts attains the same complexity bound (see \Cref{thm:minimax-delta}). However, it is unclear whether the algorithm of \citep{kovalev2022first} can recover the RECAPP's complexity bound in the setting where $f$ is not strongly convex in $x$.

\paragraph{Monteiro-Svaiter-type acceleration.}
Monteiro-Svaiter~\citep{monteiro2013accelerated} propose a variant of the accelerated proximal point method that uses an additional gradient evaluation to facilitate approximate proximal point computation. The Monteiro-Svaiter method and its extensions~\citep{gasnikov19near,bubeck2019complexity,bullins2020highly,carmon2020acceleration,song2021unified,kovalev2022first,carmon2022optimal} also allow for the regularization parameter $\lambda$ to be determined  dynamically by the procedure approximating the proximal point. \citet{ivanova2021adaptive} leverage this technique to develop a variant of Catalyst that offers improved adaptivity and, in certain cases, improved complexity. We provide additional comparison between the approximation condition of~\cite{monteiro2013accelerated,ivanova2021adaptive} and RECAPP in \Cref{sec:framework-comparison}.

\paragraph{Multilevel Monte Carlo (MLMC).}
MLMC is a method for debiasing function estimators by randomizing over the level of accuracy~\citep{giles2015multilevel}.
While originally conceived for PDEs and system simulation, a particular variant of MLMC due to~\citet{blanchet2015unbiased} has found recent applications in the theory of stochastic optimization~\citep{levy2020large,hu2021bias}.
Our method directly builds on the recent proposal of~\citet{asi2021stochastic} to use MLMC in order to obtain unbiased estimates of proximal points (or, equivalently, the Moreau envelope gradient). 
\citet{asi2021stochastic} apply this estimator to de-bias proximal points estimated via SGD and improve several structured acceleration schemes. 
In contrast, we apply MLMC on linearly convergent algorithms, allowing us to configure it much more aggressively and avoid the extraneous logarithmic factors that appeared in the rates of \citet{asi2021stochastic}.
}

\section{Proofs for~\Cref{sec:framework}}\label{apdx:framework}

We first give a formal proof of~\Cref{lem:MLMC}.

\lemMLMC*

\begin{proof}[Proof of~\Cref{lem:MLMC}]
	Let $\xopt \defeq \prox(s)$ and $E_0 \defeq \lambda\Veu_{x^\star}\Par{\xinit}+V^F_{x^\star}(\xprev)$. By definition of $\ApproxProx$ and the strong $\lambda$ strong-convexity of $F_s^\lambda$ we have  for $j=0$, 
	\begin{align}
		\label{eq:MLMC_base}
		& \E\left[\frac{\lambda}{2}\norm{x^{(0)}-\xopt}^2\right] \le\E \Flam_s\Par{x^{(0)}} - \Flam_s\Par{x^\star} \le \frac{1}{8} E_0.
    \end{align}
    Further, for all $j\ge1$,
    \begin{align}
	\E \left[\frac{\lambda}{2}\norm{x^{(j)}-\xopt}^2\right] & \le \E \Flam_s\Par{x^{(j)}} - \Flam_s\Par{x^\star} \le \frac{1}{8}~\E \Par{\lambda\Veu_{x^\star}\Par{x^{(j-1)}}+V^F_{x^\star}(x^{(j-1)})} \nonumber \\
	& \stackrel{(i)}{=} \frac{1}{8}~ \E\Par{V^{F^\lambda_s}_{x^\star}(x^{(j-1)})}
	\stackrel{(ii)}{\le} \frac{1}{8}~\E\Par{F^\lambda_s\Par{x^{(j-1)}}-F^\lambda_s\Par{\xopt}}
	 \stackrel{(iii)}{\le}  \Par{\frac{1}{8}}^{j+1} E_0, 
	 \label{eq:MLMC_induct}
	\end{align}
	where we use  $(i)$ the equality that $\norm{a-b}^2+\norm{b-c}^2-2\inprod{c-b}{a-b} = \norm{a-c}^2$, $(ii)$, the optimality of $\xopt$ which implies $\inprod{\nabla F^\lambda_s(\xopt)}{x-\xopt}\ge 0$ for any $x\in\xset$, $(iii)$ induction over $j$ and \eqref{eq:MLMC_base}.
	Consequently, $\E x^{(j)}\rightarrow \xopt$ as $j\rightarrow\infty$. Further, since $\P[J = k] = p_J$ for all $k \geq j_0$, the algorithm returns a point $x$ satisfying
	\[
	\E x = \E_J\left[x^{(j_0 - 1)}+ p_J^{-1}(x^{(J)}-x^{(J-1)})\right] = \lim_{j\rightarrow\infty} x^{(j)} = \xopt,
	\]
	which shows the output is an unbiased estimator of $\xopt$. 
	
	Next, to bound the variance, we use that $p_J = 2^{-(J_+ + 1)}$ for $p = 1/2$. Applying \eqref{eq:MLMC_base} and \eqref{eq:MLMC_induct} yields that for all $j > j_0$
	\begin{align*}
	\E \norm{x^{(j)} - x^{ (j - 1)}}^2
	&= 
	\E \norm{(x^{(j)} - \xopt) - (x^{(j - 1)} - \xopt) }^2 
	\leq 2 \E \norm{x^{(j)} - \xopt}^2
	+2 \E \norm{x^{(j - 1)} - \xopt}^2 \\
	&\leq \left(\frac{2}{8} +2 \right) \Par{\frac{1}{8}}^{j} \Par{\frac{2 E_0}{\lambda}}
	= 4.5 \cdot \Par{\frac{1}{8}}^{j} \Par{\frac{E_0}{\lambda}}.
	\end{align*}
Consequently,
	\begin{align*}
	\E \norm{p_J^{-1} \Par{x^{(J)} - x^{(\max\{J - 1,j_0\})}}}^2
	&=
	\sum_{j = j_0 + 1}^\infty p_j^{-1} \E \norm{x^{(j)} - x^{(j - 1)}}^2 
	= 
	\sum_{j = 1}^{\infty} 2^{j + 1} \E \norm{x^{(j_0 + j)} - x^{(j_0 + j- 1)} }^2 \\
	&\leq 4.5 \sum_{j = 1}^{\infty} \frac{2^{j + 1}}{8^{j_0 + j}} \Par{\frac{E_0}{\lambda}}
	= \frac{3}{8^{j_0}} \cdot \Par{\frac{E_0}{\lambda}}
	\end{align*}
and therefore,
	\begin{flalign*}
	& \E\norm{x^{(j_0)} + p_J^{-1} \left(x^{(J)}-x^{(J-1)}\right)-\xopt}^2\\
	& \hspace{1em} \leq 2 \E \norm{x^{(j_0)} - \xopt}^2 + 2 \E \norm{p_J^{-1} \Par{x^{(J)} - x^{(\max\{J - 1,j_0\})}}}^2 \\
	& \hspace{1em} \leq 
	\left[
	4 \Par{\frac{1}{8}}^{j_0 + 1} +  2  \cdot \frac{3}{ 8^{j_0}} 
	\right]
	\Par{\frac{E_0}{\lambda}}
	 = \frac{6.5 \cdot E_0}{\lambda \cdot 8^{j_0}}\,.
\end{flalign*}
Since $j_0 \ge  2$ this implies that the algorithm implements $\UnbiasedProx$ as claimed.

Finally, note that the expected number of calls made to $\ApproxProx$ is $\E J = j_0 + \E J_+$. Further, $\E J_+ = \frac{1}{1 - p}$ since $J_+$ is geometrically distributed with success probability $1 - p$. Consequently, the expected number of calls made to $\ApproxProx$ is $j_0 + \frac{1}{1 - p}$ as desired.
\end{proof}

\catalyst*

\newcommand{\filt}[1][t]{\mathcal{F}_{#1}}
\newcommand{\epsbar}{\bar{\varepsilon}}

\paragraph{Notation.} 
We first define the filtration 
$\filt = \sigma(x_1, v_1, \ldots, x_t, v_t)$ and use the notation $\xopt_{t} =\arg\min_{x\in\xset} F^\lambda_{s_{t-1}}(x)$ to denote the exact proximal mapping which iteration $x_{t}$ of the algorithm approximates. We note that $s_t, \xopt_{t+1}, \in \filt$, i.e., they are deterministic when conditioned on $x_t, v_t$.  We also recall in literature it is well-known that the coefficients $\alpha_t$ we pick satisfy the condition that $\alpha_t\in[\frac{4/\sqrt{3}}{t+3},\frac{2}{t+2}]$~\citep{paquette2017catalyst,yang2020catalyst}.

For each iteration of Algorithm~\ref{alg:catalyst}, we obtain the following bound on  potential decrease (a special case and more careful analysis of its variant in Lemma 5 of~\citet{asi2021stochastic}).

\begin{proposition}\label{lem:sapm-step-bound}
	Under the assumptions of \Cref{thm:improved-catalyst}, let $x'$ be a minimizer of $F$. For every $t$, the iterates of \Cref{alg:catalyst} satisfy
	\begin{flalign*}
		&\Ex*{
			\frac{1}{\alpha^2_{t+1}} (F(x_{t+1}) - F(x')) + \frac{\lambda}{2}\norm{v_{t+1}-x'}^2  | \filt
		} \le 
		\frac{1}{\alpha^2_t} (F(x_{t}) - F(x')) +  \frac{\lambda}{2}\norm{v_{t}-x'}^2.
	\end{flalign*}
\end{proposition}

This proposition can be proved directly by combining the following two lemmas.

\arxiv{\newcommand{\vopt}{v^\star}}

\begin{lemma}[Potential decrease guaranteed by exact proximal step]\label{lem:sapm-step-bound-1}
	At $t$-th iteration of~\Cref{alg:catalyst}, let $\xopt_{t+1} \defeq \prox(s_t)$ and $\vopt_{t+1} = v_t - (\alpha_{t+1})^{-1} (s_t - \xopt_{t+1})$ be the ``ideal'' values of $x_{t+1}$ and $v_{t+1}$ obtained via an exact prox-point computation, then we have 
	\begin{flalign}\label{eq:catalyst-eq-1}
& \frac{1}{\alpha_{t+1}^2}\left(F\left(\xopt_{t+1}\right)-F\left(x'\right)\right)+\frac{\lambda}{2}\norm{v_{t+1}^\star-x'}^2\nonumber\\
  & \hspace{1em} \le \frac{1}{\alpha_{t}^2}\prn*{F(x_{t})-F(x')}+\frac{\lambda}{2}\norm{v_{t}-x'}^2-\frac{\lambda}{\alpha^2_{t+1}}\Veu_{\xopt_{t+1}}\Par{s_t}-\frac{1}{\alpha_{t}^2}V^F_{\xopt_{t+1}}\Par{x_t}.
	\end{flalign}
\end{lemma}

\begin{proof}
 We let $\gopt_{t+1}=\lambda\left(s_{t}-\xopt_{t+1}\right)$ and $v_{t+1}^\star = v_t- (\alpha_{t+1})^{-1} (s_t - \xopt_{t+1})$. Now we bound both sides of the quantity $\inner{\gopt_{t+1}}{v_{t}-x'}$. First, note that
	\[
	\frac{1}{\alpha_{t+1}}\Par{v_{t}-x'}=\frac{1}{\alpha_{t+1}}\Par{\xopt_{t+1}-x'}+\frac{1}{\alpha_{t}^2}\left(\xopt_{t+1}-x_{t}\right)-\frac{1}{\alpha_{t+1}^2}\left(\xopt_{t+1}-s_t\right).
	\]
	Since $\gopt_{t+1}\in\partial F\left(\xopt_{t+1}\right)$ (see Fact 1.4 by~\citet{asi2021stochastic}), $F$ is convex and $\inner{\gopt_{t+1}}{\xopt_{t+1}-s_t}=-\lambda\norm{\xopt_{t+1}-s_t}^{2}$,
	we have by update of $\alpha_t$ and $v_t$ that
	\begin{flalign}
		\frac{1}{\alpha_{t+1}}\inner{\gopt_{t+1}}{v_t-x'} & =\frac{1}{\alpha_{t+1}}\inner{\gopt_{t+1}}{\xopt_{t+1}-x'}+\frac{1}{\alpha_t^2}\inner{\gopt_{t+1}}{\xopt_{t+1}-x_{t}}-\frac{1}{\alpha_{t+1}^2}\inner{\gopt_{t+1}}{\xopt_{t+1}-s_t}\nonumber\\
		& \ge \frac{1}{\alpha_{t+1}}\Par{F\left(\xopt_{t+1}\right)-F\left(x'\right)}+\frac{1}{\alpha_{t}^2}\prn*{F(\xopt_{t+1})-F(x_{t})+V^F_{\xopt_{t+1}}\Par{x_t}}+\frac{2\lambda}{\alpha^2_{t+1}}\Veu_{\xopt_{t+1}}\Par{s_t}\nonumber\\
		& =\frac{1}{\alpha_{t+1}^2}\left(F\left(\xopt_{t+1}\right)-F\left(x'\right)\right)-\frac{1}{\alpha_{t}^2}\prn*{F(x_{t})-F(x')}+\frac{2\lambda}{\alpha^2_{t+1}}\Veu_{\xopt_{t+1}}\Par{s_t}+\frac{1}{\alpha_{t}^2}V^F_{\xopt_{t+1}}\Par{x_t}.\label{eq:catalyst-1}
	\end{flalign}
	
	On the other hand to upper bound $\frac{1}{\alpha_{t+1}}\inner{\gopt_{t+1}}{v_t-x'}$, note by definition of $\gopt_{t+1}$ and $v_{t+1}^\star $,
\[
\norm{v_{t+1}^\star-x'} ^{2}=\norm{ v_t-\frac{1}{\alpha_{t+1}\lambda}g^\star_{t+1}-x'}^{2}= \norm{v_{t}-x'} ^{2}-\frac{2}{\alpha_{t+1}\lambda}\inner{g_{t+1}^\star}{v_{t}-x'}+\frac{1}{\alpha_{t+1}^2}\norm{ s_t-x^\star_{t+1}} ^{2}.
\]

Combining the last two displays and rearranging, we obtain
	\begin{flalign*}
& \frac{1}{\alpha_{t+1}^2}\left(F\left(\xopt_{t+1}\right)-F\left(x'\right)\right)+\frac{\lambda}{2}\norm{v_{t+1}^\star-x'}^2\nonumber\\
 & \hspace{1em} \le \frac{1}{\alpha_{t}^2}\prn*{F(x_{t})-F(x')}+\frac{\lambda}{2}\norm{v_{t}-x'}^2-\frac{2\lambda}{\alpha^2_{t+1}}\Veu_{\xopt_{t+1}}\Par{s_t}-\frac{1}{\alpha_{t}^2}V^F_{\xopt_{t+1}}\Par{x_t}+\frac{\lambda}{2\alpha_{t+1}^2}\norm{ s_t-x_{t+1}^\star} ^{2}\nonumber\\
  & \hspace{1em} \le \frac{1}{\alpha_{t}^2}\prn*{F(x_{t})-F(x')}+\frac{\lambda}{2}\norm{v_{t}-x'}^2-\frac{\lambda}{\alpha^2_{t+1}}\Veu_{\xopt_{t+1}}\Par{s_t}-\frac{1}{\alpha_{t}^2}V^F_{\xopt_{t+1}}\Par{x_t}.
	\end{flalign*}	
\end{proof}

\begin{lemma}[Potential difference between exact and approximate proximal step]\label{lem:sapm-step-bound-2}
Following the same notation as in~\Cref{lem:sapm-step-bound-1}, for $x_{t+1}$ and $v_{t+1}$ defined as in~\Cref{alg:catalyst}, we have 
	\begin{flalign}\label{eq:catalyst-eq-2}
& \Ex*{\frac{1}{\alpha_{t+1}^2}\left(F\left(x_{t+1}\right)-F\left(x'\right)\right)+\frac{\lambda}{2}\norm{v_{t+1}-x'}^2|\filt}\nonumber\\
& \hspace{2em} \le \frac{1}{\alpha_{t+1}^2}\left(F\left(\xopt_{t+1}\right)-F\left(x'\right)\right)+\frac{\lambda}{2}\norm{v_{t+1}^\star-x'}^2+\frac{\lambda}{\alpha^2_{t+1}}\Veu_{\xopt_{t+1}}\Par{s_t}+\frac{1}{\alpha_{t}^2}V^F_{\xopt_{t+1}}\Par{x_t},
\end{flalign}
\end{lemma}

\begin{proof}
To prove~\eqref{eq:catalyst-eq-2}, we consider the effect of approximation errors of $x_{t+1}$, $v_{t+1}$ in terms of $x^\star_{t+1}$, $v^\star_{t+1}$, respectively. First, by definition of $x_{t+1}$ and $\ApproxProx$ we have that
	\begin{flalign}\Ex*{F\left(x_{t+1}\right)|\filt}&=\Ex*{F\left(x_{t+1}\right)+\frac{\lambda}{2}\left\Vert x_{t+1}-s_t\right\Vert ^{2}|\filt}-\Ex*{\frac{\lambda}{2}\left\Vert x_{t+1}-s_t\right\Vert ^{2}|\filt}\nonumber\\
	& \le F\left(\xopt_{t+1}\right)+\lambda\Veu_{\xopt_{t+1}}\Par{s_t} + \frac{1}{8} \Par{\lambda\Veu_{x^\star_{t+1}}\Par{s_t}+V^F_{x_{t+1}^\star}(x_t)}-\Ex*{\lambda \Veu_{x_{t+1}}\Par{s_t} |\filt}\label{eq:catalyst-eq-3}
	\end{flalign}

Now we further have 
\begin{flalign*}
\Ex*{\lambda \Veu_{\xopt_{t+1}}\Par{x_{t+1}}|\filt} & \stackrel{(i)}{\le} \Ex*{F^\lambda_{s_t}(x_{t+1})|\filt}-F^\lambda_{s_t}(x_{t+1}^\star)\stackrel{(ii)}{\le} \frac{1}{8}\lambda \Veu_{\xopt_{t+1}}\Par{s_t}+\frac{1}{8}V^F_{\xopt_{t+1}}\Par{x_t}\\
& \stackrel{(iii)}{\le} \Ex*{\frac{1}{4}\lambda \Veu_{\xopt_{t+1}}\Par{x_{t+1}}+\frac{1}{4}\lambda \Veu_{x_{t+1}}\Par{s_{t}}|\filt}+\frac{1}{8}V^F_{\xopt_{t+1}}\Par{x_t},
\end{flalign*}
where we once again use $(i)$ the strong convexity of $F^\lambda$,  $(ii)$ the definition of $\ApproxProx$ and $(iii)$ the triangle inequality that $\norm{a + b}^2 \leq 2 \norm{a}^2 + 2 \norm{b}^2$ for any vectors $a,b$. 
By rearranging terms and rescaling by a factor of $1/2$ this implies equivalently
\begin{flalign}\label{eq:catalyst-helper}
	\Ex*{\frac{1}{2}\lambda \Veu_{\xopt_{t+1}}\Par{x_{t+1}}|\filt}\le \Ex*{\frac{1}{6}\lambda \Veu_{x_{t+1}}\Par{s_{t}}|\filt}+\frac{1}{12}V^F_{\xopt_{t+1}}\Par{x_t}.
\end{flalign}

Combining the above inequalities we have 
\begin{flalign}& \Ex*{F\left(x_{t+1}\right)|\filt}=\Ex*{F\left(x_{t+1}\right)+\frac{\lambda}{2}\left\Vert x_{t+1}-s_t\right\Vert ^{2}|\filt}-\Ex*{\frac{\lambda}{2}\left\Vert x_{t+1}-s_t\right\Vert ^{2}|\filt}\nonumber\\
& \hspace{3em}\stackrel{(i)}{\le} F\left(\xopt_{t+1}\right)+\frac{7}{8}\lambda\Veu_{\xopt_{t+1}}\Par{s_t} + \Ex*{\frac{1}{2}\lambda \Veu_{\xopt_{t+1}}\Par{x_{t+1}}+\frac{1}{2}\lambda \Veu_{x_{t+1}}(s_t)|\filt}+\frac{1}{8}V^F_{x_{t+1}^\star}(x_t)-\Ex*{\lambda \Veu_{x_{t+1}}\Par{s_t} |\filt}\nonumber\\
	& \hspace{3em} \stackrel{(ii)}{\le} F\left(\xopt_{t+1}\right)+\frac{7}{8}\lambda\Veu_{\xopt_{t+1}}\Par{s_t} + \frac{5}{24}V^F_{\xopt_{t+1}} \Par{x_t}+\Ex*{\Par{\frac{\lambda}{2}+\frac{\lambda}{6}-\lambda}\Veu_{x_{t+1}}\Par{s_t}|\filt}\nonumber\\
	& \hspace{3em} \le F\left(\xopt_{t+1}\right)+\frac{7}{8}\lambda\Veu_{\xopt_{t+1}}\Par{s_t} + \frac{5}{24}V^F_{\xopt_{t+1}} \Par{x_t},\label{eq:catalyst-eq-4}
\end{flalign}
where we use $(i)$ rearranging of terms and using the triangle inequality $\lambda\Veu_{\xopt_{t+1}}\Par{s_t} + \frac{1}{8} \lambda\Veu_{x^\star_{t+1}}\Par{s_t} \le  \frac{7}{8}\lambda \Veu_{x^\star_{t+1}}\Par{s_t}+  \Ex*{\frac{1}{2}\lambda \Veu_{\xopt_{t+1}}\Par{x_{t+1}}+\frac{1}{2}\lambda \Veu_{x_{t+1}}(s_t)|\filt}$ in~\eqref{eq:catalyst-eq-3}, and $(ii)$ plugging back the inequality~\eqref{eq:catalyst-helper}.
	
Now that given definition of $\UnbiasedProx$ for $v_{t+1}$ so that $\Ex*{v_{t+1}|\filt} = v_{t+1}^\star$ and consequently 
\begin{flalign}
\Ex*{\frac{\lambda}{2}\norm{v_{t+1}-x'}^2|\filt} & = \frac{\lambda}{2}\norm{v_{t+1}^\star-x'}^2+\Ex*{\frac{\lambda}{2}\norm{v_{t+1}-v^\star_{t+1}}^2|\filt}\nonumber\\
& \le \frac{\lambda}{2}\norm{v_{t+1}^\star-x'}^2+\frac{\lambda}{2\alpha_{t+1}^2}\Ex*{\norm{\tx_{t+1}-x^\star_{t+1}}^2|\filt}\nonumber\\
& \le \frac{\lambda}{2}\norm{v_{t+1}^\star-x'}^2+\frac{\lambda}{8\alpha_{t+1}^2}\Veu_{x^\star_{t+1}}(s_t)+ \frac{1}{8\alpha_{t+1}^2}V^F_{x^\star_{t+1}}(x_t).\label{eq:catalyst-eq-4}
\end{flalign}

Rescaling and summing up inequalities~\eqref{eq:catalyst-eq-3} and~\eqref{eq:catalyst-eq-4}, together with the bound that $\Par{\alpha_t}^2/\Par{\alpha_{t+1}}^2\le \frac{4(t+4)^2}{16(t+2)^2/3}\le 3$, this proves~\eqref{eq:catalyst-eq-2}, which also concludes the proof.
\end{proof}

\begin{proof}[Proof of~\Cref{thm:improved-catalyst}]
	By requirement of $\WarmStart$ function, we have $F(x_0)-F(x')\le \lambda R^2$. Applying the potential decreasing argument in~\Cref{lem:sapm-step-bound} recursively on $t=0,1,\cdots, T-1$ thus gives
	\begin{flalign*}
		\frac{1}{\alpha^2_{T}}\E \left[F(x_{T}) - F(x')\right] & \le \E \left[\frac{1}{\alpha^2_{T}} \Par{F(x_{T}) - F(x')} + \frac{\lambda}{2}\norm{v_{T}-x'}^2 \right]\\
		& \le 
		\frac{1}{\alpha^2_0} (F(x_{0}) - F(x')) +  \frac{\lambda}{2}\norm{x_0-x'}^2\le \frac{3}{2}\lambda R^2.
	\end{flalign*}
	Multiplying both by $\alpha_T^2$ and using the fact that for $T \ge \lceil \sqrt{6\lambda R^2/\epsilon}\rceil$, $\alpha_T^2 \le  \frac{4}{(T+2)^2}\le \frac{2\epsilon}{3\lambda R^2}$, we show that $x_T$ output by~\Cref{alg:catalyst} satisfy that 
	\[
	\E \left[F(x_{T}) - F(x')\right] \le \epsilon.
	\] 
	
	The number of calls to each oracles follow immediately.
	
	When implementing $\UnbiasedProx_{F,\lambda}$ using $\MLMC$, guarantees of~\Cref{lem:MLMC} immediately implies the correctness and the total number of (expected) calls to~$\ApproxProx_{F,\lambda}$.
\end{proof}

We now show an adaptation of our framework to the strongly-convex setting in~\Cref{alg:catalyst-sc}. We prove its guarantee as follows.

\begin{algorithm2e}[h]
	\caption{Restarted $\ouralg$}
	\label{alg:catalyst-sc}
	\DontPrintSemicolon
	\codeInput $F$: $\xset\rightarrow \R$, $\ouralg$ \;
	{\bfseries Parameter:} $\lambda,R>0$, iteration number $T$, epoch number $K$\;
	\label{line:catalyst:warm_start_start-sc}Initialize $x^{(0)}\leftarrow \WarmStart_{F,\lambda}(R^2)$\Comment*{To satisfy $\E F(x_0)-\min_{x'}F(x')\le \lambda R^2$}
	\For{$k=0$ {\bfseries{\textup{to}}} $K-1$}
	{Run $\ouralg$ on $F$ with $x_0=v_0=x^{(k)}$ without $\WarmStart$ (\Cref{line:catalyst:warm_start_start}) for $T$ iterations\;
		\Comment*{Halving error to true optimizer in each iteration and recurse}
	}
	\codeReturn $x^{(K)}$
\end{algorithm2e}

\catalystsc*

\begin{proof}[Proof of~\Cref{prop:catalyst-sc}]
	Let $x'$ be the minimizer of $F$, we show by induction that for the choice of $T=O\Par{\sqrt{\lambda/\gamma}}$, the iterates $x^{(k)}$ satisfy the condition that 
	\begin{equation}\label{eq:catalyst-sc-guarantee}
	\E\left[F\Par{x^{(k)}}-F(x')+\frac{\lambda}{2}\norm{x^{(k)}-x'}^2\right]\le \frac{3}{2^{k-1}}\lambda R^2,~\text{for}~k=0,1,\cdots,K.
	\end{equation}
	
	For the base case $k=0$, we have the inequality holds immediately by definition of procedure $\WarmStart$. Now suppose the inequality~\eqref{eq:catalyst-sc-guarantee} holds for $k$. For $k+1$, by~\Cref{lem:sapm-step-bound} and proof of~\Cref{thm:improved-catalyst} we obtain
	\[
	\frac{1}{\alpha_T^2}\E\left[F\Par{x^{(k+1)}}-F\Par{x'}\right]\le \E\left[F\Par{x^{(k)}}-F\Par{x'}\right]+\frac{\lambda}{2}\norm{x^{(k)}-x'}^2\le \frac{3}{2^{k-1}}\lambda R^2.
	\]
	By our choice of $T =O\Par{\sqrt{\lambda/\gamma}}$, we have 
	\[
	\E\left[F\Par{x^{(k+1)}}-F\Par{x'}\right]\le \frac{3}{2\cdot 2^{k}}\gamma R^2,
	\]
	and consequently by $\gamma$-strong convexity it holds that
	\[
	\E\left[\frac{\gamma}{2}\norm{x^{(k+1)}-x'}^2\right]\le \E\left[F\Par{x^{(k+1)}}-F\Par{x'}\right]\le \frac{3}{2\cdot 2^{k}}\gamma R^2\implies \E\left[\frac{\lambda}{2}\norm{x^{(k+1)}-x'}^2\right]\le \frac{3}{2\cdot 2^{k}}\lambda R^2.
	\]
	Summing up the two inequalities together we obtain
	\[
	\E\left[F\Par{x^{(k+1)}}-F(x')+\frac{\lambda}{2}\norm{x^{(k+1)}-x'}^2\right]\le \frac{3}{2^{k}}\lambda R^2,
	\]
	which shows by math induction that the inequality~\eqref{eq:catalyst-sc-guarantee} holds for $k=0,1,\cdots, K$.
	
	Now by choice of $K=O\Par{\log\Par{\lambda R^2/\epsilon}}$, we have 
	\begin{equation*}
	\E\left[F\Par{x^{(K)}}-F(x')\right]\le \E\left[F\Par{x^{(K)}}-F(x')+\frac{\lambda}{2}\norm{x^{(K)}-x'}^2\right]\le \frac{3}{2^{K-1}}\lambda R^2\le \epsilon,
	\end{equation*}
	which proves the correctness of the algorithm.
	
	The algorithm uses one call to~procedure $\WarmStart$ in~\Cref{line:catalyst:warm_start_start-sc}. The number of calls to procedures $\ApproxProx$, $\UnbiasedProx$ is $K$ times the number of calls within each epoch $k\in[K]$, which is bounded by $O(T)$. The case when implementing $\UnbiasedProx$ by $\MLMC$ and $\ApproxProx$ follows immediately from~\Cref{lem:MLMC}.
	
\end{proof}

\section{Proofs for~\Cref{sec:fs}}\label{apdx:fs}

\begin{proposition}[Guarantee for $\SVRGone$]\label{prop:SVRG}
For any convex, $L$-smooth $f_i(x):\xset\rightarrow\R$, and parameter $\lambda\ge0$, consider the finite-sum problem $\fun(x) \defeq \sum_{i\in[n]}\frac{1}{n}\phi_i(x)$ where $\phi_i(x) \defeq f_i(x) + \frac{\lambda}{2}\norm{x-s}^2$. Given a centering point $s$, an initial point $\xinit$, and an anchor point $\xfull$,~\cref{alg:SVRG} with instantiation of $\phi_i(x) = f_i(x) + \frac{\lambda}{2}\norm{x-s}^2$, outputs a point $\bx = \frac{1}{T}\sum_{t\in[T]}x_{t-1}$ satisfying
\begin{align*}
 \E F_s^{\lambda}\Par{\bx} - F_s^\lambda\Par{x^\star}  \le \frac{2}{\eta T} \Veu_{x^\star}\Par{\xinit} + 4 \eta LV^F_{x^\star}\Par{\xfull}
 ~
 \text{ where }
 ~
 \xopt \defeq \arg\min_{x\in\xset}\fun(x)\,.
\end{align*}
The algorithm uses a total of $O(n+T)$ gradient queries.
\end{proposition}

To prove~\Cref{prop:SVRG}, we first recall the following basic fact for smooth functions.

\begin{lemma}
Let $f:\xset\rightarrow\R$ be an $L$-smooth convex function. For any $x,x'\in\xset$ we have 
\begin{equation}\label{eq:prop-smoothness-1}
	\frac 1L \norm { \nabla f(x) - \nabla f(x') }^2 \leq  \left\langle \nabla f(x) - \nabla f(x') , x-x' \right\rangle
\end{equation}
and 
\begin{equation}\label{eq:prop-smoothness-2}
\frac{1}{2L}  \norm { \nabla f(x) - \nabla f(x') }^2  \leq f(x') - f(x) + \left\langle \nabla f(x), x-x'\right\rangle.
\end{equation}
\end{lemma}

We also observe the following few properties of~\Cref{alg:SVRG}.

\begin{lemma}[Gradient estimator properties]\label{lem:SVRG-grad}
For any $x,\xfull\in\xset$, sample $i\in[n]$ uniformly and let $g(x) = \nabla f_i(x) - \nabla f_i(\xfull) + \nabla F(\xfull) + \lambda (x-s)$. It holds that $\E\left[ g(x) \right]= \nabla F^{\lambda}_s(x)$, and for~$x^\star = \arg\min_{x\in\xset}F^\lambda_s(x)$,
\begin{align}
\label{eq:grad_est_props}
\mathbb{E} \left[ \norm{g(x)-\nabla F^{\lambda}_s(x)}^2 \right] & \leq  \mathbb{E} \left[ \norm{ \nabla f_i(\xfull) - \nabla f_i(x) }^2 \right] \leq 4L \left( V^F_{\xopt}\Par{\xfull} + V^{F^\lambda_s}_{\xopt}\Par{x} \right).
\end{align}
\end{lemma}

\begin{proof}
First, by definition of how we construct $g$ it holds that
\begin{flalign*}
\mathbb{E}\left[ g(x) \right] & = \lambda (x-s) +  \nabla F(\xfull) + \frac{1}{n} \sum_{i=1}^n \Par{\nabla f_i(x) - \nabla f_i(\xfull)}\\
& = \lambda (x-s) +  \nabla F(\xfull) + \nabla F(x) - \nabla F(\xfull) = \nabla F^{\lambda}_s(x). 
\end{flalign*}
This proves that  $g$ is an unbiased estimator of $\nabla F^\lambda_s$.

Next, for such unbiased estimator $g$ we have 
\begin{align*}
	\mathbb{E} \left[ \norm{g(x)-\nabla F^{\lambda}_s(x)}^2 \right]& = \mathbb{E} \left[ \norm{\nabla F(\xfull) - \nabla F(x) - \nabla f_i(\xfull) + \nabla f_i(x) }^2 \right]\\
	&  \stackrel{}{\leq} \mathbb{E} \left[ \norm{\nabla f_i(\xfull) - \nabla f_i(x) }^2 \right],
\end{align*}
where  we used that $\mathbb{E}\left[ \norm{Z - \mathbb{E}Z}^2 \right] \leq \mathbb{E} \left[ \norm{Z}^2 \right]$ for any random $Z$. This proves the first inequality in \eqref{eq:grad_est_props}.  

For the second inequality, note
\begin{equation}\label{eq:SVRG-bound-todo}
\begin{aligned}
& \mathbb{E} \left[ \norm{\nabla f_i(\xfull) - \nabla f_i(x) }^2 \right] \stackrel{(i)}{\leq} 2  \mathbb{E} \left[ \norm{\nabla f_i(\xfull) - \nabla f_i(x^\star) }^2 + \norm{\nabla f_i(x) - \nabla f_i(x^\star) }^2 \right] \\
&\quad\quad\quad\stackrel{(ii)}{\leq} 4 L \cdot \mathbb{E} \left[ f_i(\xfull) - f_i(x^\star) + \left\langle \nabla f_i(x^\star), x^\star - \xfull \right\rangle + f_i(x) - f_i(x^\star) + \left\langle \nabla f_i(x^\star), x^\star - x \right\rangle \right] \\
&\quad\quad\quad = 4L \cdot \left( V^F_{\xopt}\Par{\xfull} + F(x) - F(x^\star) + \left\langle \nabla F(x^\star), x^\star - x \right\rangle\right),
\end{aligned}
\end{equation}
where we use $(i)$ Cauchy-Schwarz inequality for Euclidean norms, and $(ii)$ the property of smoothness of $f_i$ (see~Eq.~\eqref{eq:prop-smoothness-2}).

Using the standard inequality that $2 \left\langle a-b,b-c \right\rangle = \norm{a-c}^2 - \norm{b-c}^2 - \norm{a-b}^2$, we also have 
\begin{align*}
&F(x) - F(x^\star) + \left\langle \nabla  F(x^\star), x^\star - x\right\rangle \\
&\quad= F^{\lambda}_s(x) - F^{\lambda}_s(x^\star)+ \inprod{\nabla F(x^\star) + \lambda (x^\star - s)}{x^\star-x} - \lambda \left\langle x^\star - s, x^\star - x \right\rangle  - \frac{\lambda}{2} \norm{x-s}^2 + \frac{\lambda}{2} \norm{ x^\star - s}^2 \\
&\quad\le V^{F^{\lambda}_s}_{\xopt}\Par{x}-\lambda\Veu_{\xopt}(x)\le V^{F^{\lambda}_s}_{\xopt}\Par{x}.
\end{align*}
 Substituting this in~\eqref{eq:SVRG-bound-todo} proves the second inequality in \eqref{eq:grad_est_props}.
\end{proof}

The following proof on progress per step follows from the standard analysis by~\citet{xiao2014proximal} (as the constrained finite-sum minimization we consider is a special case of theirs), which we include the full statement here for completeness.

\begin{lemma}[Progress per step of $\SVRGone$, cf.~\citet{xiao2014proximal}]\label{lem:SVRG-onestep-progress}
Let $\xopt \in  \arg\min_{x\in\xset}\fun(x)$ where  each $\phi_i$ is $L$-smooth and convex. For $\eta\le 1/L$, consider step $t$ in $\SVRGone$, define $\Delta_t = g_t-\nabla\fun(x_t)$ and $\xopt_{t+1} = \proj_{\xset}\Par{x_t-\eta g_t}$, it holds that 
\begin{align*}
\mathbb{E} \norm{x_{t+1}-\xopt}^2\le \norm{x_{t}-\xopt}^2-2\eta\left[\fun(x_{t+1})-\fun(\xopt)\right]+2\eta^2\E\norm{\Delta_t}^2.
\end{align*}
\end{lemma}

With these helper lemmas, we are ready to formally prove~\Cref{prop:SVRG}.

\begin{proof}[Proof of~\Cref{prop:SVRG}]

Consider the $t$-th step of~\Cref{alg:SVRG}, by~\Cref{lem:SVRG-onestep-progress} and $\Phi = F^\lambda_s$, for $\eta\le \frac{1}{2L}\le \frac{1}{\lambda+L}$ we have
\[ 
\mathbb{E} \norm{x_{t+1} - x^\star}^2 \le \norm{x_t - x^\star}^2 - 2 \eta \mathbb{E} \left[  F^\lambda_s\Par{x_{t+1}}-F^\lambda_s\Par{\xopt} \right] + \eta^2 \mathbb{E} \left[ \norm{g_t-\nabla F^{\lambda}_s(x_t)}^2 \right]. 
\]
Our bounds on the variance of SVRG plus the $L$-smoothness of $F$ yields
\begin{align*}
\mathbb{E} \left[ \norm{g_t-\nabla F^{\lambda}_s(x_t)}^2 \right] & \leq 4L \left( V^F_{\xopt}\Par{\xfull} + V^{F^{\lambda}_s}_{\xopt}\Par{x_t}\right).
\end{align*}
Thus we have by definition of $x_{t+1}$ and divergence,
\begin{align}\label{eq:SVRG-one-step-progress}
\mathbb{E} \Veu_{x^\star}\Par{x_{t+1}} \leq & \Veu_{x^\star}\Par{x_{t}} -  \eta \mathbb{E} \left[  F^\lambda_s\Par{x_{t+1}}-F^\lambda_s\Par{\xopt} \right] + 2\eta^2  LV^{F^{\lambda}_s}_{\xopt}\Par{x_t} + 2\eta^2 L V^F_{\xopt}\Par{\xfull}.
\end{align}

Note $V^F_{x^\star}\Par{\xfull}$ is independent of $x_t$. Telescoping bounds~\eqref{eq:SVRG-one-step-progress} and using optimality of $\xopt$ so that $\inprod{\nabla F^\lambda_s\Par{\xopt}}{\xopt-x}\le 0$ for $t\ge 1$, we obtain 
\[
\mathbb{E} \Veu_{x^\star}\Par{x_{T}} \leq \Veu_{x^\star}\Par{\xinit} - \eta \Par{1-2\eta L}\mathbb{E} \left[\sum_{t=0}^{T-1} \Par{ F^\lambda_s\Par{x_{t+1}}-F^\lambda_s\Par{\xopt}} \right] + 2\eta^2 LV^{F^{\lambda}_s}_{\xopt}\Par{\xinit}+ 2\eta^2 L T V^F_{x^\star}\Par{\xfull}.
\]
Rearranging terms, dividing over $\eta T/2$, and using convexity of $F^\lambda_s$, we have for $\eta \le \tfrac{1}{32L}$ and $T\ge \frac{32}{\eta\lambda}\ge \frac{128\eta L^2}{\lambda}$
\begin{align*}
 & \E F_s^{\lambda}\Par{\frac{1}{T}\sum_{t\in[T]}x_{t}} - F_s^\lambda\Par{x^\star} \le \frac{1}{T} \E \sum_{t=1}^{T}\left( F_s^{\lambda}(x_t) - F_s^\lambda\Par{x^\star} \right) \leq \frac{2(1-2\eta L)}{T}\E \left[\sum_{t=0}^{T-1}   \left(F^{\lambda}_s(x_t)  - F^{\lambda}_s(x^\star) \right)\right] \\
& \quad\quad\quad \le \frac{2}{\eta T} \Veu_{x^\star}\Par{\xinit} + 4 \eta LV^F_{x^\star}\Par{\xfull}+\frac{4\eta L}{T}V^{F^{\lambda}_s}_{\xopt}\Par{\xinit}\\
& \quad\quad\quad \stackrel{(\star)}{\le} \frac{1}{8}V^F_{x^\star}\Par{\xfull}+\Par{\frac{2}{\eta T}+\frac{4\eta L(L+\lambda)}{T}}\Veu_{\xopt}\Par{\xinit}\\
& \quad\quad\quad\le \frac{1}{8}V^F_{x^\star}\Par{\xfull}+\frac{1}{8}\lambda\Veu_{\xopt}\Par{\xinit},
\end{align*}
where for inequality $(\star)$ we use the fact that $F^\lambda_s$ is $(L+\lambda)$-smooth, and the property of smoothness.
\end{proof}

To show how we implement the $\WarmStart$ procedure required in~\Cref{alg:catalyst}, we first show the guarantee of the low-accuracy solver for finite-sum minimization of~\Cref{alg:SVRG-warmstart}.

\begin{lemma}[Low-accuracy solver for finite-sum minimization]\label{lem:SVRG-warmstart}
For any problem~\eqref{def:problem-fs} with minimizer $\xopt$, smoothness parameter $L$, initial point $\xinit$ so that $R=\norm{\xopt-\xinit}$, and any $\alpha\ge L/n$, \Cref{alg:SVRG-warmstart} with $T = 32n$, finds a point $x^{(K)}$ after $K$  epochs such that $\E~F\Par{x^{(K)}}-F(\xopt)\le \frac{1}{2}n^{-1+2^{-K}}L R^2$.%
\end{lemma}

\begin{proof}[Proof of~\Cref{lem:SVRG-warmstart}]

We prove the argument by math induction and let $c^{(k)} = n^{-1+2^{-k}}$. Note that for the base case $k=0$, $F(x^{(0)})-F(\xopt)\le \frac{c^{(0)}}{2}LR^2$ by Eq.~\eqref{eq:prop-smoothness-2}. Now suppose the above inequality holds for $k$, i.e. $F(x^{(k)})-F(\xopt)\le \tfrac{c^{(k)}}{2} LR^2$. Then for epoch $k$ by guarantee of~\Cref{prop:SVRG} together  given choice of $\eta_{k+1} = \tfrac{1}{8L\sqrt{n c^{(k)}}}$ and $T = 32n$ we have
\[
\E F\Par{x^{(k+1)}} - F\Par{x^\star} \le \frac{2}{\eta_{k+1} T} \Veu_{x^\star}\Par{x^{(k)}} + 4 \eta_{k+1} LV^F_{x^\star}\Par{x^{(k)}}\le \frac{LR^2}{4}\sqrt{\frac{c^{(k)}}{n}}+ \frac{LR^2}{4}\sqrt{\frac{c^{(k)}}{n}} = \frac{c^{(k+1)}}{2}LR^2.
\]
where for the last inequality we note that series $c^{(k)}$ satisfies $c^{(k+1)} = \sqrt{c^{(k)}/n}$. 
\end{proof}

Consequently, after $K = O(\log\log n)$ epochs, we have 
\[
c^{(K)}\le \frac{2}{n}\implies \E F\Par{x^{(K)}} - F\Par{x^\star}\le \frac{L}{n}R^2\le \alpha R^2,
\]
for any $\alpha\ge L/n$, which immediately proves the following corollary.

\warmstartfs*

Now we give the formal proof of~\Cref{thm:OptCatfs}, the main theorem showing one can use our accelerated scheme $\ouralg$ to solve the finite-sum minimization problem~\eqref{def:problem-fs} efficiently.

\OptCatfs*

\begin{proof}[Proof of~\Cref{thm:OptCatfs}]
We first consider the objective function $F$ without strong convexity. The correctness of the algorithm follows directly from~\Cref{thm:improved-catalyst}, together with~\Cref{coro:approxprox-fs} and~\Cref{coro:SVRG-warmstart}. For the query complexity, calling $\WarmStart$-$\SVRG$ to implement the procedure of $\WarmStart_{F,\lambda}(R^2)$  requires gradient queries $O\Par{n\log \log n}$ following~\Cref{coro:SVRG-warmstart}. The main~\Cref{alg:catalyst} calls $O\Par{R\sqrt{\lambda/\epsilon}}=O(R\sqrt{Ln^{-1}\epsilon^{-1}})$ of procedure $\ApproxProx_{F,\lambda}$, which by implementation of $\SVRGone$ each requires $O(n+L/\lambda)=O(n)$ gradient queries following~\Cref{coro:approxprox-fs}. Summing them together gives the claimed gradient complexity in total. 

When the objective $F$ is $\gamma$-strongly-convex, the proof follows by the same argument as above and the guarantee of restarted $\ouralg$ in~\Cref{thm:improved-catalyst}.
\end{proof}

\notarxiv{
\subsection{Additional Details on Empirical Results}\label{ssec:experiments-more}

Here we provide additional details for the empirical results in~\Cref{ssec:experiments}.

\paragraph{SVRG implementation.} We implement the SVRG iterates as in \Cref{alg:SVRG}, using $T=2n$ and $\eta=4$ (i.e., the inverse of the smoothness of each function). However, instead of outputting the average of all iterates, we return the average of the final $T/2=n$ iterates.

\paragraph{Catalyst implementation.} Our implementation follows closely Catalyst C1* as described in~\cite{lin2017catalyst}, where for the subproblem solver we use repeatedly called \Cref{alg:SVRG} with the parameters and averaging modification described above, checking the C1 termination criterion between each call.

\paragraph{RECAPP implementation.} Our RECAPP implementation follows \Cref{alg:catalyst,alg:unbiased-min-MLMC}, with \Cref{alg:SVRG} and \Cref{alg:SVRG-warmstart} implemented \ApproxProx and \WarmStart, respectively, and \Cref{alg:SVRG} configured and modified and described above. In \Cref{alg:unbiased-min-MLMC} we set the parameters $j_0=0$ and we test $p\in\{0,0.1,0.25,0.5\}$. The setting $p=0$ which corresponds to setting $\tilde{x}_{t+1} = x_{t+1}$ in~\Cref{alg:catalyst}) is a baseline meant to test whether MLMC is helpful at all. For $p>0$ we change the parameter $T$ in \Cref{alg:SVRG} such that the \emph{expected} amount of gradient computations is the same as for $p=0$. Slightly departing from the pseudocode of \Cref{alg:catalyst}, we take $x_{t+1}$ to be $x^{(J)}$ computed in \Cref{alg:unbiased-min-MLMC}, rather than $x^{(0)}$, since it is always a more accurate proximal point approximation. We note that our algorithm still has provable guarantees (with perhaps different constant factors) under this configuration.

\paragraph{Parameter tuning.} For RECAPP and Catalyst, we tune the proximal regularization parameter $\lambda$ (called $\kappa$ in~\cite{lin2017catalyst}). For each problem and each algorithm, we test $\lambda$ values of the form $\alpha L / n$, where $L=0.25$ is the objective smoothness, $n$ is the dataset size and $\alpha$ in the set $\{0.001, 0.003, 0.01, 0.03, 0.1, 0.3, 1.0, 3.0, 10.0\}$. We report results for the best $\lambda$ value for each problem/algorithm pair. 
}

\section{Proofs for~\Cref{sec:minimax}}\label{apdx:minimax}

 We first consider a special case of standard mirror-prox-type methods~\citep{nemirovski2004prox} with Euclidean $\ell_2$-divergence on $x$ and $y$ domains separately, i.e. $V_{x,y}(x',y') = \Veu_{x}(x')+\Veu_{x}(y')$ . This ensures each step of the methods can be implemented efficiently. Below we state its guarantees, which is standard from literature and we include here for completeness.

\begin{restatable}[$T$-step guarantee of $\MirrorProx$, cf. also~\citet{nemirovski2004prox}]{lemma}{lemsmmirrorprox}\label{lem:mirror-prox}
	Let any $\phi(x,y), (x,y)\in\xset\times\yset$ be a convex-concave, $L$-smooth function, $\MirrorProx(\phi,L,\xinit,\yinit, T)$ in \Cref{alg:mirror-prox} with initial points $(\xinit, \yinit)$ and step size $\eta= 1/L$ outputs a point $(x_T,y_T)$ satisfying for any $x,y\in\xset\times \yset$
	\begin{equation}\label{eq:mirror-prox-non-sm}
	\begin{aligned}
	\phi(x_T,y)- \phi(x,y_T)\le \frac{L}{T}\Par{\Veu_{\xinit}\Par{x}+\Veu_{\yinit}\Par{y}}.
	\end{aligned}
	\end{equation}
\end{restatable}

Next we give complete proofs on the implementation of $\ApproxProx_{F,\mu}$ and $\WarmStart_{F,\mu}$, using $\MirrorProx$ (\Cref{alg:mirror-prox}) with proper choices of initialization $\xinit$, $\yinit$, and~$\WarmStart$-$\textsc{Minimax}$ (\Cref{alg:warmstart-minimax}), respectively.

\lemapproxproxminimax*

\begin{proof}[Proof of~\Cref{lem:approxprox-minimax}]
	We incur $\MirrorProx$ with $\phi(x,y) = f(x,y) + \frac{\mu}{2}\norm{x-s}^2$, smoothness $L+\mu,$ initial point $\xinit, \yinit$, and number of iterations $T = \frac{64(L+\mu)}{\mu} $. By guarantee~\eqref{eq:mirror-prox-non-sm} in~\Cref{lem:mirror-prox}, the algorithm outputs $x_T, y_T$ such that for any $x\in\xset$, $y\in\yset$
	\begin{align*}
		\phi(x_T,y)-\phi(x, y_T)\le \frac{1}{64}\mu \Par{\Veu_{\xinit}\Par{x}+\Veu_{\yinit}\Par{y}}.
	\end{align*}
	
	Suppose $\phi$ has $(x^\star,y^\star)$ as its unique saddle-point, in particular we pick $x = x^\star$ and $y =\ybr_{x_T}\defeq \arg\max_{y\in\yset} f(x_T,y) = \arg\max_{y\in\yset} \phi(x_T,y)$ in the above inequality to obtain
	\begin{equation}\label{eq:minimax-ineq}
		\begin{aligned}
	F_s^\mu\Par{x_T}-F_s^\mu\Par{x^\star} & = \phi\Par{x_T,\ybr_{x_T}} - \phi\Par{x^\star,y^\star} = \Par{\phi\Par{x_T,\ybr_{x_T}} - \phi\Par{x^\star,y_T}+\Par{\phi\Par{x^\star,y_T}} - \phi\Par{x^\star,y^\star}}\\
	 & \le \phi\Par{x_T,\ybr_{x_T}} - \phi\Par{x^\star,y_T}
	 \le \frac{1}{64}\mu \Par{\Veu_{\xinit}\Par{\xopt}+\Veu_{\yinit}\Par{\ybr_{x_T}}},
	\end{aligned}
	\end{equation}
	where for the first inequality we use the definition that $y^\star = \arg\max_{y\in\yset}\phi(x^\star,y)$.
Now for the LHS of~\eqref{eq:minimax-ineq}, we have
\begin{flalign}
	& F_s^\mu\Par{x_T}-F_s^\mu\Par{x^\star}\nonumber\\
	& \hspace{2em} \ge \Par{1-\frac{1}{32}}\Par{F_s^\mu\Par{x_T}-F_s^\mu\Par{x^\star}} +\frac{1}{32}\Par{\phi\Par{x_T,\ybr_{x_T}}-\phi\Par{x_T,\yopt}}\nonumber +\frac{1}{32}\Par{\phi\Par{x_T,\yopt}-\phi\Par{\xopt,\yopt}}\nonumber\\
	& \hspace{2em} \ge \Par{1-\frac{1}{32}}\Par{F_s^\mu\Par{x_T}-F_s^\mu\Par{x^\star}}+\frac{\mu}{32} \Veu_{\ybr_{x_T}}\Par{\yopt}.\label{eq:minimax-ineq-1}
\end{flalign}
	Plugging~\eqref{eq:minimax-ineq-1} back to~\eqref{eq:minimax-ineq} and rearranging terms, we obtain
		\begin{equation}\label{eq:minimax-almost-final}
\begin{aligned}
	\Par{F_s^\mu\Par{x_T}-F_s^\mu\Par{x^\star}} \le & \frac{\mu/64}{1-1/32}\Par{\Veu_{\xinit}\Par{\xopt}+\Veu_{\yinit}\Par{\ybr_{x_T}}}-\frac{\mu/32}{1-1/32}\Veu_{\ybr_{x_T}}\Par{\yopt}\\
	\stackrel{(\star)}{\le} & \frac{\mu/64}{1-1/32}\Veu_{\xinit}\Par{\xopt}+ \frac{\mu/32}{1-1/32}\Veu_{\yinit}\Par{\yopt} +\frac{\mu/32}{1-1/32}\Veu_{\ybr_{x_T}}\Par{\yopt} - \frac{\mu/32}{1-1/32}\Veu_{\ybr_{x_T}}\Par{\yopt}\\
	\le & \frac{1}{16}\mu\Par{\Veu_{\xinit}\Par{\xopt}+ \Veu_{\yinit}\Par{\yopt}},
\end{aligned}
\end{equation}
where we use $(\star)$ Cauchy-Schwarz inequality for Euclidean norm. 

Further, to bound RHS of~\eqref{eq:minimax-almost-final}, we note that by definition of $F$ and $\yinit \gets \brorcl_{f}(\xprev)$,
\begin{align*}
\mu\Veu_{\yinit}\Par{\yopt} = & \mu \Veu_{\ybr_{\xprev}}\Par{\yopt}\stackrel{(i)}{\le}  f\Par{\xprev,\ybr_{\xprev}} - f\Par{\xprev,\yopt}\\
 \stackrel{(ii)}{\le}  &f\Par{\xprev,\ybr_{\xprev}} -f\Par{\xopt,\yopt}-\langle\nabla _xf\Par{\xopt,\yopt},\xprev-\xopt\rangle  = V^F_{\xopt}\Par{\xprev},
\end{align*}
where we use $(i)$ strong convexity in $y$ of $-f$ and $(ii)$ convexity of $f(\cdot,\yopt)$.

Plugging this back in~\eqref{eq:minimax-almost-final}, we obtain 
\begin{align*}
F_s^\mu\Par{x_T}-F_s^\mu\Par{x^\star} \le \frac{\mu}{16}\Veu_{\xinit}\Par{\xopt}+\frac{1}{8}V^F_{\xopt}\Par{\xprev} \le \frac{1}{8}\Par{\mu\Veu_{\xopt}\Par{\xinit}+V^F_{\xopt}\Par{\xprev}}.
\end{align*}

Thus we prove that $\ApproxProx_{F,\mu}$ can be implemented via $\MirrorProx$ properly. The total complexity includes one call to $\brorcl_{f}(\cdot)$ and $O(T) = O(L/\mu)$ gradient queries as each iteration in $\MirrorProx$ requires two gradients.
\end{proof}

\lemwarmstartminimax*

\begin{proof}[Proof of~\Cref{lem:minimax-warm-start}]
	Given domain diameter $R,R'$ and the initialization $\xinit$, $\yinit$, we first use accelerated gradient descent (cf.~\citet{nesterov1983method}) to find a $\Theta(LR^2)$-approximate solution of $\max_{y\in\yset}f(\xinit,y)$ (which we set to be $\yinit'$) using $O\Par{\sqrt{L/\mu}\log\Par{R'/R}}$ gradient queries. We recall the definition of $\ybr_{\xinit}\defeq \arg\max_{y\in\yset}f(\xinit,y)$ and thus
	\begin{equation}\label{eq:minimax-opt-of-yinit}
	\begin{aligned}
		\mu\Veu_{\yinit'}\Par{\ybr_{\xinit}} & \le f(\xinit,\yinit')-f(\xinit,\ybr_{\xinit}) \le \frac{1}{2}LR^2.
	\end{aligned}
	\end{equation}
	
	Now we incur $\MirrorProx$ with objective $\phi(x,y) = f(x,y) + \frac{\mu}{2}\norm{x-\xinit}^2$, smoothness $L+\mu$, initial points $(\xinit,\yinit')$. We let $\xopt\Par{\phi},\yopt$ to denote its unique saddle point.  Thus, we have by iterating guarantee of~\eqref{eq:mirror-prox-non-sm} in~\Cref{lem:mirror-prox} with $T = O(L/\mu)$ iterations, after $K = O\Par{\log{L/\mu}}$ calls to $\MirrorProx$ we have
	\begin{align*}
	& \Veu_{x^{(K)}}\Par{x^\star\Par{\phi}}+ \Veu_{y^{(K)}}\Par{x^\star}\le  \frac{1}{40}\Par{\frac{\mu}{L}}^4\Par{\Veu_{\xinit}\Par{x^\star\Par{\phi}}+ \Veu_{\yinit'}\Par{y^\star}} \\
	& \quad\quad\quad\stackrel{(i)}{\le} \frac{1}{40}\Par{\frac{\mu}{L}}^4\Par{\Veu_{\xinit}\Par{x^\star\Par{\phi}}+ 2\Veu_{\yinit'}\Par{\ybr_{\xinit}}+2\Veu_{\ybr_{\xinit}}\Par{y^\star}}\\
	& \quad\quad\quad\stackrel{(ii)}{\le} \frac{1}{40}\Par{\frac{\mu}{L}}^4 \Par{\frac{1}{2}R^2+ \frac{LR^2}{\mu}+2\frac{L^2}{\mu^2}\Veu_{\xinit}\Par{\xopt\Par{\phi}}}\\
	& \quad\quad\quad\stackrel{}{\le}  \frac{1}{4}\Par{\frac{\mu}{L}}^4\frac{L^2}{\mu^2}R^2 =  \frac{\mu^2}{4L^2}R^2,
	\end{align*}
	where we use $(i)$ Cauchy-Schwarz inequality for Euclidean norms, and $(ii)$ condition \eqref{eq:minimax-opt-of-yinit} and the fact that $\ybr_{x}$ is $L/\mu$-Lipschitz in $x$.

Thus given $F(x) = \max_{y}f(x,y)$ being $(L + L^2/\mu)$-smooth, we have
\begin{align}\label{eq:minimax-bound-on-F}
F\Par{x^{(K)}}-F\Par{x^\star\Par{\phi}}\le \Par{L+L^2/\mu}\Veu_{\xopt\Par{\phi}}\Par{x^{(K)}}\le \frac{\mu}{2}R^2.
\end{align}
Note we also have
\begin{align*}
F\Par{x^\star\Par{\phi}}\le F^{\mu}_{\xinit}\Par{x^\star\Par{\phi}}\stackrel{(\star)}{\le} F^{\mu}_{\xinit}\Par{x^\star}= F\Par{x^\star}+\frac{\mu}{2}\norm{\xopt-\xinit}^2\le F\Par{x^\star}+\frac{\mu}{2}R^2,
\end{align*}
where we use $(\star)$ that $x^\star\Par{\phi}$ minimizes $F^{\mu}_{\xinit}(x)$. Plugging this back to~\eqref{eq:minimax-bound-on-F}, we obtain $F\Par{x^{(K)}}-F\Par{\xopt}\le \mu R^2$.

The gradient complexity of mirror-prox part is $O(KT) = O(L/\mu\log(L/\mu))$. Summing this together with the gradient complexity for accelerated gradient descent used in obtaining $\yinit'$ gives the claimed query complexity.

\end{proof}

\section{Generalization of Framework and Proof of~\Cref{thm:minimax-delta}}\label{apdx:minimax-delta}

In this section, we present a generalization of the framework, where we allow additive errors when implementing 
$\ApproxProx$ and $\UnbiasedProx$ (\Cref{def:approx-prox-delta} and~\ref{def:unbiased-prox-delta}). When the additive error is small enough, it would contributes to at most $O(\epsilon)$ additive error in the function error and thus generalize our framework (\Cref{alg:catalyst-delta} and~\Cref{prop:improved-catalyst-delta}). In comparison to prior works APPA/Catalyst~\citep{frostig2015regularizing,lin2015universal,lin2017catalyst}, in the application to solving max-structured problems our additive error comes from some efficient method with cheap total gradient costs, thus only contributing to the low-order terms in the oracle complexity (\Cref{thm:minimax-delta}). 

We first re-define the following procedures of $\ApproxProx$ and $\UnbiasedProx$, which also tolerates additive ($\delta$-)error. 

\begin{definition}[$\ApproxProx^\delta$]\label{def:approx-prox-delta}
	Given convex function $F$: $\xset\rightarrow\R$, regularization parameter $\lambda>0$, a centering point $s \in\xset$ and two points $\xinit, \xprev\in\xset$, $\ApproxProx_{F,\lambda}^\delta(s;\xinit,\xprev)$ is a procedure that outputs an approximate solution $x$ such that for $x^\star = \prox(s) = \arg\min_{x\in\xset} F^\lambda_s(x)$,
\begin{equation}\label{approx-prox-cond-gen}
\begin{aligned}
\E \Flam_s(x) - \Flam_s(x^\star)\le \frac{1}{8} & \Par{\lambda\Veu_{x^\star}\Par{\xinit}+V^F_{x^\star}(\xprev)}+\delta.
\end{aligned}
\end{equation}
\end{definition}

\begin{definition}[$\UnbiasedProx^\delta$]\label{def:unbiased-prox-delta}
	Given convex function $F$: $\xset\rightarrow\R$, regularization parameter $\lambda>0$, a centering point $s \in\xset$, two points $\xinit, \xprev\in\xset$, $\UnbiasedProx_{F,\lambda}^\delta\Par{s;\xprev}$ is a procedure that outputs an approximate solution $x$ such that $\E~x = x^\star= \prox(s) = \arg\min_{x\in\xset} F^\lambda_s(x)$, and 
\begin{equation}
\begin{aligned}
\E\norm{x-\xopt}^2 \le & \frac{1}{4\lambda}\Par{\lambda\Veu_{x^\star}\Par{s}+V^F_{x^\star}(\xprev)}+\frac{2\delta}{\lambda}.
\end{aligned}
\end{equation}
\end{definition}

\SetKwProg{Fn}{function}{}{}
\begin{algorithm2e}[h]
	\caption{$\ouralg$ with Additive Error}
	\label{alg:catalyst-delta}
	\DontPrintSemicolon
	\codeInput $F : \xset \rightarrow \R$, $\ApproxProx^\delta$, $\UnbiasedProx^\delta$\;
	{\bfseries Parameter:} $\lambda, R>0$, iteration number $T$, $\alpha_0=1$ \;
	Initialize $x_0\leftarrow \WarmStart_{F,\lambda}(R^2)$ \Comment*{To satisfy $\E~F(x_0)-F(\xopt)\le \lambda R^2$}\label{line:recapp-error-warmstart}
	\For{$t=0$ {\bfseries{\textup{to}}} $T-1$}{
		Update parameters $\alpha_{t+1}\in[0,1]$ to satisfy $\frac{1}{\alpha_{t+1}^2}-\frac{1}{\alpha_{t+1}} = \frac{1}{\alpha_t}$ \;
		$s_t\gets \Par{1-\alpha_{t+1}}x_t + \alpha_{t+1}v_t  $ \;
		$\delta_{t+1} \gets \frac{1}{4t^2}\alpha_t^2\lambda R^2$ \Comment*{Additive error decreases as $O(1/t^4)$ to ensure convergence}
		$x_{t+1}\gets \ApproxProx^{\delta_{t+1}}_{F,\lambda}(s_t; s_t, x_t)$ \;
		$\tx_{t+1}\gets \UnbiasedProx_{F,\lambda}^{\delta_{t+1}}(s_t; x_t)$ \Comment*{We implement this using $\MLMC^{\delta_{t+1}}(F, \lambda,s_t, x_t)$}
		$v_{t+1}\gets \proj_{\xset}\Par{v_t-\frac{1}{\alpha_{t+1}}\Par{s_t-\tx_{t+1}}}$ \;
	}
	\codeReturn $x_T$
	
	\Fn{$\MLMC^\delta(F, \lambda, s, \xprev)$}{
	$\delta_{0}\gets2^{-3}\delta$\;
$x^{(0)}\gets \ApproxProx^{\delta_{0}}_{F,\lambda}(s; s, \xprev)$ \;
	Sample random epoch number $J\sim  1+\geoRV\left(\tfrac{1}{2}\right)\in\{2,3,\cdots\}$ \;
	\For{$j=1$ \codeStyle{to} $J$}{
		$\delta_{j}\gets \frac{1}{4}\delta_{j-1}$\;
		$x^{(j)}\gets \ApproxProx^{\delta_{j}}_{F,\lambda}(s; x^{{(j-1)}}, x^{{(j-1)}})$\;
	}
	\codeReturn $x^{(1)}+2^J(x^{(J)}-x^{(J-1)})$ \;
}
\end{algorithm2e}	

\begin{algorithm2e}[h]
	\caption{Restarted $\ouralg$ with Additive Error}
	\label{alg:catalyst-delta-sc}
	\DontPrintSemicolon
	\codeInput $F : \xset \rightarrow \R$, $\ouralg$ with additive error \;
	{\bfseries Parameter:} $\lambda, R>0$, iteration number $T$, epoch number $K$, $\alpha_0=1$\;
	\label{line:catalyst:warm_start_start-sc-delta}Initialize $x^{(0)}\leftarrow \WarmStart_{F,\lambda}(R^2)$\Comment*{To satisfy $\E~F(x_0)-F(\xopt)\le \lambda R^2$}
	\For{$k=0$ {\bfseries{\textup{to}}} $K-1$}{
	\Comment*{Halving error to true optimizer and recurse}
Run $\ouralg$ (\Cref{alg:catalyst-delta}) on $F$ with $x_0=v_0=x^{(k)}$ without $\WarmStart$ (\Cref{line:recapp-error-warmstart}) for $T$ iterations}
	\codeReturn $x^{(K)}$
\end{algorithm2e}	

With the new definitions of $\delta$-additive proximal oracles and $\delta$-additive unbiased proximal point estimators, we can formally give the guarantee of~\Cref{alg:catalyst-delta} in~\Cref{prop:improved-catalyst-delta}. 

\begin{restatable}[$\ouralg$ with additive error]{proposition}{catalyst-delta}\label{prop:improved-catalyst-delta} For any convex function $F$: $\xset\rightarrow\R$, parameters $\lambda, R>0$, $\ouralg$ with additive error (\Cref{alg:catalyst-delta}) finds $x$ such that $\E F(x) - \min_{x'\in\xset}F(x')\le \eps$ within $O\Par{R\sqrt{\lambda/\eps}}$ iterations.  The algorithm uses one call to $\WarmStart$, and an expectation of oracle complexity 
\begin{align*}
O\Par{\sum_{t\in[T]}\sum_{j=0}^\infty\frac{1}{2^j}\calN \Par{\ApproxProx, F, \lambda,{2^{-2j}\delta_{t}}}},
\end{align*}
where we let $\delta_t = \frac{1}{4t^2}\alpha_t^2\lambda R^2=\Omega(\epsilon^2/(\lambda R^2))$, and use $\calN(\ApproxProx,F,\lambda,\delta)$ to denote some oracle complexity for calling each $\ApproxProx^{\delta}_{F,\lambda}$. 

For any $\gamma$-strongly-convex $F$: $\xset\rightarrow\R$, parameters $\lambda, R>0$, restarted $\ouralg$ (\Cref{alg:catalyst-delta-sc}) finds $x$ such that $\E F(x) - \min_{x'\in\xset}F(x')\le \eps$ , using one call to $\WarmStart$ and an expectation of oracle complexity 
\begin{align*}
O\Par{\sum_{k\in[K]}\sum_{t\in[T]}\sum_{j=0}^\infty\frac{1}{2^j}\calN \Par{\ApproxProx, F, \lambda,{2^{-2j}\delta^{(k)}_{t}}}},
\end{align*}
where $K = O(\log LR^2/\epsilon)$, $T = O\Par{\sqrt{\lambda/\gamma}}$, and $\delta^{(k)}_t = \frac{1}{2^k\cdot 4t^2}\alpha_t^2\lambda R^2 = \Omega\Par{\frac{1}{2^kt^4}\lambda R^2}$ for $t\in[T], k\in[K]$.
\end{restatable}

To prove the correctness of~\Cref{prop:improved-catalyst-delta}, we first show in~\Cref{lem:MLMC-delta} that $\MLMC^\delta$ implements an $\UnbiasedProx^\delta$ for given $F$, $\lambda>0$, with the corresponding inputs. In comparison with the $\delta=0$ case presented in~\Cref{sec:framework}, the key difference is we need to ensure when we sample a large index $j$ (with tiny probability), the algorithm calls $\ApproxProx^{\delta_j}$ to smaller additive error $\delta_j\approx \Theta(4^{-j})\cdot\delta$, so as to ensure it contributes in total a finite $O(\delta)$ additive term in the variance.

\begin{lemma}[$\MLMC$ turns $\ApproxProx^{O(\delta)}$ into $\UnbiasedProx^\delta$]\label{lem:MLMC-delta}
Given convex $F$: $\xset\rightarrow\R$, $\lambda>0$, $s,\xprev\in\xset$, function~$\MLMC^\delta(F,\lambda,s,\xprev)$ in~\Cref{alg:catalyst-delta} implements $\UnbiasedProx^\delta_{F,\lambda}(s;\xprev)$. Denote $\calN \Par{\ApproxProx, F, \lambda,{2^{-3}\delta}}$ as some oracle complexity for calling each $\ApproxProx^{\delta}_{F,\lambda}$, then the oracle complexity $\calN$ for $\UnbiasedProx^\delta_{F,\lambda}$ is $\E\calN = \calN \Par{\ApproxProx, F, \lambda,{2^{-3}\delta}}+\calN \Par{\ApproxProx, F, \lambda,{2^{-5}\delta}}+ \sum_{j=2}^\infty \frac{1}{2^{j-2}}  \calN \Par{\ApproxProx, F, \lambda,{2^{-(3+2j)}\delta}}$.
\end{lemma}

\begin{proof}[Proof of~\Cref{lem:MLMC-delta}]
	Let $\xopt = \arg\min_{x\in\xset}F^\lambda_s(x)$, by definition of $\ApproxProx^\delta$, we have  
	\begin{align*}
		\text{for}~~j=0,~~\E~\left[ \frac{\lambda}{2}\norm{x^{(0)}-\xopt}^2\right]  & \le\E~\Flam_s\Par{x^{(0)}} - \Flam_s\Par{x^\star} \le \frac{1}{8}~\Par{\lambda\Veu_{x^\star}\Par{\xinit}+V^F_{x^\star}(\xprev)}+\frac{\delta}{8},\\
		\text{for}~~j\ge1,~~\E~\left[ \frac{\lambda}{2}\norm{x^{(j)}-\xopt}^2\right] &  \stackrel{(i)}{\le} \E~V^{F^\lambda_s}_{\xopt}\Par{x^{(j)}}  \le\E~\Flam_s\Par{x^{(j)}} - \Flam_s\Par{x^\star} \\
		& \le \frac{1}{8}~\E \Par{\lambda\Veu_{x^\star}\Par{x^{(j-1)}}+V^F_{x^\star}(x^{(j-1)})}+\frac{\delta}{2\cdot4^{j+1}}\\
		& \stackrel{(ii)}{\le} \frac{1}{8}~ \E\Par{V^{F^\lambda_s}_{x^\star}(x^{(j-1)})}+\frac{\delta}{2\cdot4^{j+1}}\\
		& \stackrel{}{\le} \frac{1}{8}~\E\Par{F^\lambda_s\Par{x^{(j-1)}}-F^\lambda_s\Par{\xopt}}+\frac{\delta}{2\cdot4^{j+1}}\\
		& \stackrel{(iii)}{\le}  \Par{\frac{1}{8}}^{j+1} \Par{\lambda\Veu_{x^\star}\Par{\xinit}+V^F_{x^\star}(\xprev)} + \delta\Par{\sum_{j'=0}^j\frac{1}{2\cdot 8^{j'}\cdot 4^{j+1}}},
	\end{align*}
	where we use $(i)$ the optimality of $\xopt$ which implies $\inprod{\nabla F(\xopt)}{x-\xopt}\ge 0$ for any $x\in\xset$,  $(ii)$ the equality that $\norm{a-b}^2+\norm{b-c}^2-2\inprod{c-b}{a-b} = \norm{a-c}^2$,  $(iii)$ the induction over $j$.
	
	In conclusion, this shows that $\E x^{(j)}\rightarrow \xopt$ as $j\rightarrow\infty$, and thus by choice of $p_j = 1/2^{j-1}$ for $j\ge2$, the algorithm returns a point $x$ satisfying
	\[
	\E~x = \E_J\left[x^{(1)}+2^J(x^{(J)}-x^{(J-1)})\right] = \lim_{j\rightarrow\infty} x^{(j)} = \xopt,
	\]
	which shows the output is an unbiased estimator of $\xopt$. 
	
	For the variance, we have by a same calculation as in the proof of~\Cref{lem:MLMC},
	\begin{align*}
	\E & \norm{x^{(1)}+2^J\left(x^{(J)}-x^{(J-1)}\right)-\xopt}^2 \le \frac{5}{2}\cdot\E\left[\norm{x^{(1)}-\xopt}^2\right]+ \sum_{j=2}^{\infty} \frac{9}{2}\cdot2^{j}\cdot\E\left[\norm{x^{(j)}-\xopt}^2\right]\\
	 & \quad\quad\quad\le \frac{2}{\lambda}\Par{\frac{5}{2}\Par{\frac{1}{8}}^{2}+\sum_{j=2}^\infty\frac{9/8}{2}\Par{\frac{1}{4}}^j}\Par{\lambda\Veu_{x^\star}\Par{\xinit}+V^F_{x^\star}(\xprev)}+\frac{5\delta}{16\lambda}+ \frac{9\delta}{8\lambda}\sum_{j=2}^\infty\frac{5}{4}\cdot\frac{1}{2^j}\\
	 & \quad\quad\quad\le \frac{1}{4\lambda}\Par{\lambda\Veu_{x^\star}\Par{\xinit}+V^F_{x^\star}(\xprev)}+\frac{2\delta}{\lambda},
	\end{align*}
	which proves the bound as claimed.
	
The query complexity is in expectation
\begin{align*}
& \E\calN = \sum_{j=2}^\infty \frac{1}{2^{j-1}}   \sum_{j'=0}^j\calN \Par{\ApproxProx, F, \lambda,{2^{-(3+2j')}\delta}} = \calN \Par{\ApproxProx, F, \lambda,{2^{-3}\delta}}\\
& \quad\quad\quad+\calN \Par{\ApproxProx, F, \lambda,{2^{-5}\delta}}+ \sum_{j=2}^\infty \frac{1}{2^{j-2}}  \calN \Par{\ApproxProx, F, \lambda,{2^{-(3+2j)}\delta}} .	
\end{align*}
\end{proof}

This shows that we can implement $\UnbiasedProx^\delta$ using $\ApproxProx^\delta$, similar to the case without additive error $\delta$, as in~\Cref{lem:MLMC}. Now we are ready to provide a complete proof of~\Cref{prop:improved-catalyst-delta}, which shows the correctness and complexity of~\Cref{alg:catalyst-delta}. 

\begin{proof}[Proof of~\Cref{prop:improved-catalyst-delta}]

First of all we recall the notation of filtration  
$\filt = \sigma(x_1, v_1, \ldots, x_t, v_t)$, $\xopt_{t} =\arg\min_{x\in\xset} F^\lambda_{s_{t-1}}(x)$, $\gopt_{t+1}=\lambda\left(s_{t}-\xopt_{t+1}\right)$, $v_{t+1}^\star = v_t- (\alpha_{t+1})^{-1} (s_t - \xopt_{t+1})$ and $x'$ as the minimizer of $F:\xset\rightarrow\R$ (see~\Cref{apdx:framework} for more detailed discussion).

The majority of the proof still lies in showing the potential decreasing lemma as in~\Cref{lem:sapm-step-bound}, while also taking into account the extra additive error $\delta$ when implementing oracles $\ApproxProx^\delta$ and $\UnbiasedProx^\delta$.

Following~\eqref{eq:catalyst-eq-1}, we recall the inequality that 
	\begin{flalign}
& \frac{1}{\alpha_{t+1}^2}\left(F\left(\xopt_{t+1}\right)-F\left(x'\right)\right)+\frac{\lambda}{2}\norm{v_{t+1}^\star-x'}^2\nonumber\\
  & \hspace{4em} \le \frac{1}{\alpha_{t}^2}\prn*{F(x_{t})-F(x')}+\frac{\lambda}{2}\norm{v_{t}-x'}^2-\frac{\lambda}{\alpha^2_{t+1}}\Veu_{\xopt_{t+1}}\Par{s_t}-\frac{1}{\alpha_{t}^2}V^F_{\xopt_{t+1}}\Par{x_t}.\label{eq:catalyst-eq-1-delta}
	\end{flalign}
	Thus, by definition of $x_{t+1}$, $\delta_{t+1}$ and $\ApproxProx^\delta$ we have that
	\begin{flalign*}
	\Ex*{F\left(x_{t+1}\right)|\filt}& \le\Ex*{F\left(x_{t+1}\right)+\frac{\lambda}{2}\left\Vert x_{t+1}-s_t\right\Vert ^{2}|\filt} \\
	& \le F\left(\xopt_{t+1}\right)+\frac{7}{8}\lambda\Veu_{\xopt_{t+1}}\Par{s_t} + \frac{5}{24} V^F_{x_{t+1}^\star}(x_t)+\delta_{t+1}
	\end{flalign*}
Similarly to~\eqref{eq:catalyst-eq-4} and its analysis, we also have by definition of $\UnbiasedProx^\delta$ that 
\begin{flalign*}
\Ex*{\frac{\lambda}{2}\norm{v_{t+1}-x'}^2|\filt} & = \norm{v_{t+1}^\star-x'}^2+\Ex*{\frac{\lambda}{2}\norm{v_{t+1}-v^\star_{t+1}}^2|\filt}\\
& \le \frac{\lambda}{2}\norm{v_{t+1}^\star-x'}^2+\frac{\lambda}{2\alpha_{t+1}^2}\Ex*{\norm{\tx_{t+1}-x^\star_{t+1}}^2|\filt}\\
& \le \frac{\lambda}{2}\norm{v_{t+1}^\star-x'}^2+\frac{\lambda}{2\alpha_{t+1}^2}\Par{\frac{\lambda \Veu_{x^\star_{t+1}}(s_t)}{4\lambda}+\frac{V^F_{x^\star_{t+1}}(x_t)}{4\lambda}+\frac{2\delta_{t+1}}{\lambda}}.
\end{flalign*}

Plugging these back into~\eqref{eq:catalyst-eq-1-delta}, we conclude that
	\begin{flalign*}
		& \frac{1}{\alpha_{t+1}^2}\left(\Ex*{F\left(x_{t+1}\right)|\filt}-F\left(x'\right)\right) +\frac{\lambda}{2}\Ex*{\left\Vert v_{t+1}-x'\right\Vert ^{2}|\filt}\\
		& \hspace{15em} \le \frac{1}{\alpha_{t}^2}\prn*{F(x_{t})-F(x')}+ \frac{\lambda}{2}\left\Vert v_{t}-x'\right\Vert ^{2}+\frac{2\delta_{t+1}}{\alpha_{t+1}^2}.
	\end{flalign*}
Recursively applying this bound for $t=0,1,\cdots, T-1$ and together with the $\WarmStart$ guarantee we have
\begin{flalign*}
& \frac{1}{\alpha_T^2} \left(\Ex*{F\left(x_{t+1}\right)|\filt}-F\left(x'\right)\right) +\frac{\lambda}{2}\Ex*{\left\Vert v_{t+1}-x'\right\Vert ^{2}|\filt}\\
& \hspace{15em} \le \frac{1}{\alpha_{0}^2}\prn*{F(x_{0})-F(x')}+ \frac{\lambda}{2}\left\Vert x_{0}-x'\right\Vert ^{2}+\sum_{t\in[T]}\frac{2\delta_{t}}{\alpha_{t}^2}\\
& \hspace{3em} \implies \frac{1}{\alpha_T^2}\left(\Ex*{F\left(x_{t+1}\right)|\filt}-F\left(x'\right)\right) \le \frac{3}{2}\lambda R^2+\sum_{t\in[T]}\frac{2\delta_{t}}{\alpha_{t}^2}\le 4\lambda R^2\\
& \hspace{3em} \implies \left(\Ex*{F\left(x_{t+1}\right)|\filt}-F\left(x'\right)\right)\le \epsilon,
\end{flalign*}
where we use the choice of $\delta_t = \frac{1}{4t^2}\alpha_t^2\lambda R^2$ and that $\sum_{t\in[T]}\frac{1}{t^2} \le \pi^2/6\le 2$. This shows the correctness of the algorithm.

The algorithm uses $O(1)$ call to $\WarmStart$. At each iteration $t+1$, by guarantee of implementing $\UnbiasedProx^\delta$ using $\MLMC$ in~\Cref{lem:MLMC-delta}, we have the query complexity with respect to $\ApproxProx^\delta$ is in expectation
\begin{align*}
	\calN \Par{\ApproxProx, F, \lambda,\delta_{t+1}}+\calN \Par{\ApproxProx, F, \lambda,{2^{-3}\delta_{t+1}}}+\calN \Par{\ApproxProx, F, \lambda,{2^{-5}\delta_{t+1}}}\\+ \sum_{j=2}^\infty \frac{1}{2^{j-2}}  \calN \Par{\ApproxProx, F, \lambda,{2^{-(3+2j)}\delta_{t+1}}} = O\Par{\sum_{j=0}^\infty\frac{1}{2^j}\calN \Par{\ApproxProx, F, \lambda,{2^{-2j}\delta_{t+1}}}}
\end{align*}
which implies the total oracle complexity through calling $\ApproxProx^\delta$ by summing over $t = 0,1,\cdots, T-1$.

The strongly-convex case follows by a similar analysis as in the proof of~\Cref{prop:catalyst-sc}. We show by in duction
\begin{equation}\label{eq:catalyst-delta-sc-guarantee}
	\E\left[F\Par{x^{(k)}}-F(x')+\frac{\lambda}{2}\norm{x^{(k)}-x'}^2\right]\le \frac{4}{2^{k-1}}\lambda R^2,~\text{for}~k=0,1,\cdots,K,
	\end{equation}
taking into account that by choice of $\delta^{(k+1)}_{t+1}$, the contribution of the additive errors is always bounded by $\frac{1}{2^k}\lambda R^2$. This choice also implies the expected oracle complexity due to calling $\ApproxProx^{\delta_t^{(k)}}$ differently at each epoch and iteration. 
\end{proof}

The additive errors allowed by this framework are helpful to the task of minimizing the max-structured convex objective $F(x) = \max_{y\in\yset}f(x,y)$. This is because we can then use accelerated gradient descent to solve~$\max_yf(x,y)$ for the best-response oracle needed in~\Cref{line:minimax-br-agd} to high accuracy before calling calling~\Cref{alg:mirror-prox} , and show that $\MirrorProx$ formally implements a $\ApproxProx^\delta$. The resulting gradient complexity has an extra logarithmic term on $\delta$, but only shows up on a low-order $\widetilde{O}(\sqrt{L/\mu})$ terms.

\begin{corollary}[Implementation of~$\ApproxProx^\delta$ for minimizing max-structured function]\label{coro:approxprox-minimax-delta}
Given the minimization of max-structured problem in~\eqref{def:problem-minimax}, a centering point $s$, points $\xinit$, $\xprev$, one can use accelerated gradient descent to solve to additive error $\delta$ for~\eqref{line:minimax-br-agd}  and use ~\Cref{lem:approxprox-minimax} to implement the procedure $\ApproxProx^\delta_{F,\mu}(s;\xinit,\xprev)$. It uses a total of $O\Par{L/\mu+\sqrt{L/\mu}\log(L(R')^2/\delta)}$ gradient queries.
\end{corollary}

\begin{proof}[Proof of~\Cref{coro:approxprox-minimax-delta}]
	Given the initialization $\xinit$, we first use accelerated gradient descent~\citet{nesterov1983method} to find a $\delta$-approximate solution of $\max_{y\in\yset}f(\xprev,y)$ (which we set to be $\yinit$). We recall the definition of $\ybr_{\xprev}\defeq \arg\max_{y\in\yset}f(\xprev,y)$ and thus
	\begin{equation}\label{eq:minimax-opt-of-yinit-delta}
	\begin{aligned}
		\mu\Veu_{\yinit}\Par{\ybr_{\xprev}} \le f(\xprev,\yinit)-f(\xprev,\ybr_{\xprev}) \le \delta
	\end{aligned}
	\end{equation}
	using $O(\sqrt{L/\mu}\log(L(R')^2/\delta))$ gradient queries.
	
Then, we invoke $\MirrorProx$ with $\phi(x,y) = f(x,y) + \frac{\mu}{2}\norm{x-s}^2$, initial point $\xinit, \yinit$, and number of iterations $T = \frac{64(L+\mu)}{\mu} $. The rest of the proof is essentially the same as in~\Cref{lem:approxprox-minimax}, with the only exception that when bounding RHS of~\eqref{eq:minimax-almost-final}, we note that by choice of $\yinit$ and the error bound in~\eqref{eq:minimax-opt-of-yinit-delta}, it becomes
\begin{align*}
& \frac{\mu}{2}\Veu_{\yinit}\Par{\yopt}\le \mu\Veu_{\ybr_{\xprev}}\Par{\yopt} +  \mu\Veu_{\yinit}\Par{\ybr_{\xprev}}\le  f\Par{\xprev,\ybr_{\xprev}} - f\Par{\xprev,\yopt}+ \delta\le  V^F_{\xopt}\Par{\xprev}+\delta.
\end{align*}

Plugging this new bound with additive error $\Theta(\delta)$ back in~\eqref{eq:minimax-almost-final}, we obtain 
\begin{align*}
F_s^\mu\Par{x_K}-F_s^\mu\Par{x^\star} & \le \frac{\mu}{16}\Veu_{\xinit}\Par{\xopt}+\frac{1}{8}V^F_{\xopt}\Par{\xprev}+\delta \le \frac{1}{8}\Par{\mu\Veu_{\xopt}\Par{\xinit}+V^F_{\xopt}\Par{\xprev}}+\delta.
\end{align*}

Thus the procedure implements $\ApproxProx^\delta_{F,\mu}(s;\xinit,\xprev)$. The total gradient complexity is the  complexity in $\MirrorProx$ same as~\Cref{lem:minimax-warm-start} plus the extra complexity in implementing $\brorcl_F\Par{\cdot}$ using accelerated gradient descent, which sums up to $O\Par{L/\mu+\sqrt{L/\mu}\log\Par{L(R')^2/\delta}}$ as claimed.
\end{proof}

\thmminimax*

\begin{proof}[Proof of~\Cref{thm:minimax-delta}]

For the non-strongly-convex case, $T = O\Par{R\sqrt{\mu/\epsilon}}$, the correctness of the algorithm follows directly from the non-strongly-convex case of~\Cref{prop:improved-catalyst-delta}, together with~\Cref{coro:approxprox-minimax-delta} and~\Cref{lem:minimax-warm-start}. For the query complexity, calling $\WarmStart$-$\textsc{Minimax}$  to implement the procedure of $\WarmStart$ to $\mu R^2$ error requires  $O\Par{L/\mu\log(L/\mu)+\sqrt{L/\mu}\log\Par{R'/R}}$ gradient queries by~\Cref{lem:minimax-warm-start}. Following~\Cref{prop:improved-catalyst-delta}, denote $\calN \Par{\ApproxProx, F, \mu,\delta}$ to be the gradient complexity of implementing $\ApproxProx^\delta_{F,\mu}$: we have $\calN \Par{\ApproxProx, F, \mu,\delta}=O\Par{L/\mu+\sqrt{L/\mu}\log(L(R')^2/\delta)}$ by~\Cref{coro:approxprox-minimax-delta}. Consequently, the total gradient complexity for implementing all $\ApproxProx^\delta$ is in expectation 
\begin{align*}
O\Par{\sum_{t\in[T]}\sum_{j=0}^\infty\frac{1}{2^j}\calN \Par{\ApproxProx, F, \mu,{2^{-2j}\delta_{t}}}} & = O\Par{R\sqrt{\frac{\mu}{\epsilon}}\Par{\frac{L}{\mu}+\sqrt{\frac{L}{\mu}}\log\Par{\frac{L(R'+R)^2}{\epsilon}}}}\\
& = O\Par{\frac{LR}{\sqrt{\mu\epsilon}}+\sqrt{\frac{LR^2}{\epsilon}}\log\Par{\frac{L(R'+R)^2}{\epsilon}}},
\end{align*}
where we use $\delta_t\ge \Omega\Par{\frac{1}{T^4}\mu R^2}
=\Omega(\epsilon/\sqrt{T})$ and choice of $T=O\Par{R\sqrt{\mu/\epsilon}}$.

Summing the gradient query complexity from both $\WarmStart$ and $\ApproxProx^\delta$ procedures give the final complexity.	

For the $\gamma$-strongly-convex case, the correctness of the algorithm follows directly from the strongly-convex case of~\Cref{prop:improved-catalyst-delta}, together with~\Cref{coro:approxprox-minimax-delta} and~\Cref{lem:minimax-warm-start}. The query complexity for calling one $\WarmStart$-$\textsc{Minimax}$ remains unchanged. Following~\Cref{prop:improved-catalyst-delta}, denote $\calN \Par{\ApproxProx, F, \mu,\delta}$ to be the gradient complexity of implementing $\ApproxProx^\delta_{F,\mu}$: we have $\calN \Par{\ApproxProx, F, \lambda,\delta}=O\Par{L/\mu+\sqrt{L/\mu}\log(L(R')^2/\delta)}$ by~\Cref{coro:approxprox-minimax-delta}. Consequently, the total gradient complexity for implementing all $\ApproxProx^\delta$ is in expectation 
\begin{align*}
& O\Par{\sum_{k\in[K]}\sum_{t\in[T]}\sum_{j=0}^\infty\frac{1}{2^j}\calN \Par{\ApproxProx, F, \mu,{2^{-2j}\delta_{t}^{(k)}}}}\\
& \quad\quad =  O\Par{\sqrt{\frac{\mu}{\gamma}}\log \Par{\frac{LR^2}{\epsilon}}\Par{\frac{L}{\mu}+\sqrt{\frac{L}{\mu}}\log\Par{\frac{\mu L(R'+R)^2}{\epsilon\gamma}}}}\\
& \quad\quad = O\Par{\frac{L}{\sqrt{\mu\gamma}}\log \Par{\frac{LR^2}{\epsilon}}+\sqrt{\frac{L}{\gamma}}\log\Par{\frac{\mu L(R'+R)^2}{\epsilon\gamma}}\log \Par{\frac{LR^2}{\epsilon}}},
\end{align*}
where we use $\delta_t^{(k)}\ge \Omega\Par{\frac{1}{2^K\cdot T^4}\mu R^2}$ and choice of  $K$ and $T$.

\end{proof}

 \notarxiv{
\section{Discussion}\label{sec:discussion}

This paper proposes an improvement of the APPA/Catalyst acceleration framework, providing an efficiently attainable Relaxed Error Criterion for the Accelerated Prox Point method ($\ouralg$) that eliminates logarithmic complexity terms from previous result while maintaining the elegant black-box structure of APPA/Catalyst.

The main conceptual drawback of our proposed framework (beyond its reliance on randomization) is that efficiently attaining our relaxed error criterion requires a certain degree of problem-specific analysis as well as careful subproblem solver initialization. In contrast, APPA/Catalyst rely on more standard and readily available linear convergence guarantees (which of course also suffice for $\ouralg$).

Nevertheless, we believe there are many more situations where efficiently meeting the relaxed criterion is possible. These include variance reduction for min-max problems, smooth min-max problems which are (strongly-)concave in $y$ but not convex in $x$, and problems amenable to coordinate methods. All of these are settings where APPA/Catalyst is effective~\citep{yang2020catalyst,frostig2015regularizing,lin2017catalyst} and our approach can likely be provably better. 

Moreover, even when proving improved rates is difficult, $\ApproxProx$ can still serve as an improved stopping criterion. This motivates further research into practical variants of $\ApproxProx$ that depend only on observable quantities (rather than, e.g.\ the distance to the true proximal point).

 }

\end{document}